\documentclass[reqno]{amsart}

\usepackage{graphicx,subcaption}
\usepackage[utf8]{inputenc}
\usepackage[backend=bibtex,minbibnames=5,maxbibnames=5,maxcitenames=5]{biblatex}
\usepackage{amssymb}
\usepackage{amsmath}
\usepackage{graphicx}
\usepackage{amsthm,amsmath}
\usepackage{amsfonts}
\usepackage{graphicx}
\usepackage{enumitem}
\usepackage{bm}
\usepackage[utf8]{inputenc}
\numberwithin{equation}{section}
\newtheorem{theo}{Theorem}[section]
\newtheorem{theorem}{Theorem}[section]

\newtheorem{assumpintro}{Assumption}[section]
\newtheorem{lem}[theo]{Lemma}
\newtheorem{lemma}[theo]{Lemma}

\newtheorem{prop}[theo]{Proposition}
\newtheorem{proposition}[theo]{Proposition}

\theoremstyle{definition}
\newtheorem{defi}[theo]{Definition}
\newtheorem{definition}[theo]{Definition}
\theoremstyle{remark}
\newtheorem{rem}[theo]{Remark}

\usepackage{lmodern}
\usepackage{caption}
\usepackage{color}

\newcommand\tempskipped[1]{}

\newcommand\C{\mathbb{C}}

\newcommand\R{\mathbb{R}}

\newcommand\E{\mathbb{E}}

\newcommand\cst{\operatorname{cst}}

\renewcommand\Re{\operatorname{Re}}
\renewcommand\Im{\operatorname{Im}}

\newcommand\osc{\operatorname{osc}}

\usepackage{tikz}
\usetikzlibrary{calc}
\usepackage{graphicx}

\newcommand\cZ{\mathcal{Z}}
\newcommand\cE{\mathcal{E}}
\newcommand\cQ{\mathcal{Q}}
\newcommand\cS{\mathcal{S}}
\newcommand\cR{\mathcal{R}}
\newcommand\cX{\mathcal{X}}
\newcommand\cT{\mathcal{T}}

\newcommand\dm{\diamond}
\newcommand\Dm{\diamondsuit}

\def\LipKd{{\mbox{\textsc{Lip(}$\kappa$\textsc{,}$\delta$\textsc{)}}}}
\def\ExpFat{{{\mbox{\textsc{Exp-Fat(}$\delta$\textsc{,}$\rho$\textsc{)}}}}}
\def\OneExpFat{{{\mbox{\textsc{Exp-Fat(}$\delta$\textsc{,}$1$\textsc{)}}}}}
\newcommand\Unif{{\mbox{\textsc{Unif(}$\delta$\textsc{)}}}}

\newcommand\vcirc[1]{v^\circ_1,\ldots,v^\circ_{#1}}
\newcommand\svcirc[1]{\sigma_{v^\circ_1}\ldots \sigma_{v^\circ_{#1}}}
\newcommand\vbullet[1]{v^\bullet_1,\ldots,v^\bullet_{#1}}
\newcommand\muvbullet[1]{\mu_{v^\bullet_1}\ldots\mu_{v^\bullet_{#1}}}

\bibliography{Completebibliography.bib}

\begin{document}

\title{Crossing estimates for the Ising model on general s-embeddings}

\author[Rémy Mahfouf]{Rémy Mahfouf$^\mathrm{A}$}

\thanks{\textsc{${}^\mathrm{A}$ Université de Genève.}}

\thanks{\emph{E-mail:} \texttt{remy.mahfouf@unige.ch}}

\maketitle

\begin{abstract}
We prove Russo–Seymour–Welsh-type crossing estimates for the FK–Ising model on general s-embeddings whose origami map has an asymptotic Lipschitz constant strictly smaller than $1$, provided it satisfies a mild non-degeneracy assumption. This result extends the work of \cite{Che20} and provides a general framework to prove that the usual connection probabilities between the boundaries of large scale boxes remain bounded away from $0$ and $1$. We also explain that one cannot prove similar estimates without such an assumption on the origami map, and this allows us to propose a notion of a critical model for generic planar graphs, which can also be rephrased from the perspective of the associated propagator operator. Our theorem re-establishes, along the way, corresponding results in almost all previously known setups and also treats new ones of interest.
\end{abstract}

\section{Introduction, main results and perspectives}\label{sec:introduction}
\setcounter{equation}{0}

\subsection{General context} The Ising model, introduced a century ago by Lenz, is one of the most studied models in statistical mechanics. Its planar version (i.e., the model on a planar graph with nearest-neighbor interactions) has received extensive attention from both physicists and mathematicians and gives rise to numerous local and global observables that can be computed exactly (see, e.g., the monographs \cite{friedli-velenik-book,mccoy-wu-book,palmer2007planar}). In this article, we focus on the model with no external magnetic field and, contrary to usual conventions, we assign spins to \emph{faces} (denoted by $G^\circ$) of a planar graph $G$ (whose vertices are denoted by $G^{\bullet}$). When $G$ is a finite connected graph and $\beta > 0$ a positive number (called the inverse temperature), one can attach to each edge $e\in E(G)$ separating the two faces $v^\circ_{\pm}(e)\in G^\circ$ a coupling constant $J_e>0$ and construct a discrete probabilistic model, whose partition function is given by
\begin{equation}
\label{eq:intro-Zcirc}
\mathcal{Z}(G)\ :=\ \sum_{\sigma:G^\circ\to\{\pm 1\}}\exp\big[\,\beta\sum_{e\in E(G)}J_e\sigma_{v^\circ_-(e)}\sigma_{v^\circ_+(e)}\,\big].
\end{equation}
The domain walls representation (see e.g. \cite[Sec 1.2]{CCK}) allows to rewrite $\cZ(G)$ as
\[
\textstyle \cZ(G)\ =\ 2\prod_{e\in E(G)}(x(e))^{-1/2}\times\,\sum_{C\in\cE(G)}\prod_{e\in C}x(e), \text{ where } x(e):=\exp[-2\beta J_e]
\]
and $\cE(G)$ denotes the set of even subgraphs of $G$. By naturally identifying an edge $e$ of $G$ to the associated face $z(e)$ of the bipartite graph $\Lambda(G):=G^{\bullet} \cup G^\circ$, one can construct an \emph{abstract parametrization} (i.e. a priori without any geometric interpretation) of the coupling constant $x(e)$ given by
\begin{equation}
\label{eq:x=tan-theta} \theta_{z(e)}\ :=\ 2\arctan x(e)\ \in\ (0,\tfrac{1}{2}\pi).
\end{equation}
This abstract definition is purely combinatorial and thus does \emph{not} require fixing an embedding of~$G$ into~$\C$: this fact was used by Chelkak to introduce the notion of \emph{s-embeddings} in \cite{Ch-ICM18,Che20}. The overall goal of his construction is to provide an embedding procedure that allows the study of large-scale properties of weighted planar graphs $(G,x)$ carrying critical or near-critical weights, depending on its collection of edge weights $(x(e))_{e\in E}$, including those that are locally very irregular, in a spirit similar to works such as, e.g., circle packing embeddings \cite{GGJN}, Tutte's embeddings \cite{gwynne2021tutte}, or more recently Cardy's embeddings \cite{holden2023convergence}. The notion of (near-)criticality, which we understand here to be related to the existence of a non-trivial scaling limit as well as the absence of infinite clusters for primal and dual models at the same time, does not yet seem to have been proposed for generic planar graphs except in particular cases, the most famous examples being periodic lattices \cite{cimasoni-duminil}, the Z-invariant setup \cite{BdTR2, park-iso}, or Lis' circle packings \cite{lis-kites}. Therefore, one aim of this paper is to propose a notion of criticality for general planar graphs, which can be read from the way the graph is embedded into the plane. As an example, the square lattice chosen with critical or off-critical weights will lead to embeddings with drastically different large-scale properties (see \cite[Figure 2]{Che20}). More importantly, this allows us to reformulate a notion of criticality (for crossing estimates) from the point of view of the spectral properties of the operator associated with the propagation equation \eqref{eq:3-terms} (see the question formulated in Section \ref{sub:novelties} for more details).

The route taken in \cite{Che20} was to generalize the notion of discrete fermionic observables, in the spirit of the pioneering work of Smirnov \cite{Smi-ICM06,Smirnov_Ising} based on discrete complex analysis techniques. It is centered on constructing an embedding procedure that is heavily tied to a combinatorial version of the associated fermionic local relations. In particular, the construction extends to a much larger class of graphs than simply the isoradial or periodic ones, and we explore here its application to crossing probabilities. It is worth noting that the notion of s-embeddings is encapsulated in the more general framework of \emph{t-embeddings} or \emph{Coulomb gauges}, in the context of the bipartite dimer model \cite{KLRR,CLR1,CLR2}. This fact allows one to benefit from the regularity theory of discrete harmonic and holomorphic functions developed in \cite[Section 6]{CLR1}, as well as an existence statement for finite planar graphs in \cite[Section 7]{KLRR}. The s-embedding setup has already proved its relevance in \cite[Theorem 1.2]{Che20} by settling the question of conformal invariance for the critical double-periodic graphs (the criticality condition in this setup was derived by Cimasoni and Duminil-Copin in \cite[Theorem 1.1]{cimasoni-duminil}), but also for regular graphs with an origami function satisfying $\cQ=O(\delta)$ (see Definition \ref{def:cQ-def} for more details). In the critical double-periodic case, even finding the correct canonical embedding \cite[Lemma 2.3]{Che20} and proving the convergence of FK interfaces to SLE(16/3) remained open for nearly a decade. This allowed one to go one step further regarding universality with respect to the local lattice details (see also similar results on critical and near-critical isoradial grids for correlation functions in \cite{ChSmi2,CHI,park-iso,CHI-mixed,park2018massive,hon-smi,Cim-universality}). We hope that the present framework will help to study the Ising model in some random environments (see the example of a random triangulation decorated with the Ising model presented in Section \ref{sub:explicit-construction}), targeting critical random maps equipped with the Ising model—which is conjectured to converge to Liouville Quantum Gravity (e.g., see~\cite{duplantier-sheffield-LQG})—as well as deterministic graphs with random coupling constants (e.g., $\mathbb{Z}^2$ with i.i.d.\ coupling constants studied numerically in \cite{wang1990two}). We emphasize that the current paper does not treat these last two difficult questions. Still, it is worth mentioning that the question of the i.i.d.\ coupling constants recently saw some progress in \cite{Mah25}, opening a route for a full understanding of this conjecture.

\subsection{Definition of the embedding and the associated scale}

In order to keep the presentation compact, we postpone to Section \ref{sec:definitions} the precise definition of the construction of a proper s-embeddings $\cS $ and give a rather informal and concrete review on what a tiling obtained following the construction of \cite{Che20} looks like. One starts with a \emph{weighted} planar graph $(G,x)$ with the combinatorics of the plane (or of the sphere in the finite case), whose vertices are denoted by $G^{\bullet}$ and faces by $G^{\circ}$. This graph is defined up to homeomorphisms preserving the cyclic order of edges at each vertex. The graph $\Lambda(G):=G^{\bullet}\cup G^{\circ} $ can be viewed as a bipartite graph with edges connecting neighbors of different color in $ \Lambda(G)$. Fix a quadrilateral $z=(v_0^\bullet v_0^\circ  v_1^\bullet  v_1^\circ)$, corresponding to a face of $\Lambda(G)$. A proper s-embedding of $G$ can be viewed as a map $\cS :\Lambda(G) \rightarrow \mathbb{C} $, such that all its edges are straight segments, and that all the faces of $\Lambda(G) $ (except maybe the outer face in the case of a finite graphs) are \emph{tangential quadrilaterals}. More precisely, each face $z$ of $\Lambda(G)$ is mapped to a  quadrilateral $\cS^{\diamond}(z):=(\cS(v^{\bullet}_{0} )\cS(v^{\circ}_{0} )\cS(v^{\bullet}_{1})\cS(v^{\circ}_{1}))$ which is tangential to a circle. To be more precise, one requires that the four lines containing the edges $\cS^{\diamond}(z)$ are tangential to a circle whose radius is denoted by $r_z$, including the case of non-convex quads. The embedding is called \emph{proper and non-degenerate} if different faces do not overlap each other and if none of the quadrilaterals $\cS^{\diamond}(z)$ degenerates to a segment. It is possible to recover the Ising weight attached to an edge from the angles of the associated tangential quadrilateral using the relation \eqref{eq:theta-from-S}. In particular, the overall idea is not based upon finding special weights that fit an embedded graph, but goes the other way around. Moreover, there are typically many different pictures for the same abstract graph, and all are as legitimate from the discrete complex analysis perspective. Those embeddings are stable under rotation, translation, scaling and complex conjugations.

The second object of crucial importance in the s-embeddings framework is the so-called origami map $\cQ $, recalled in Definition \ref{def:cQ-def}. In words, the origami map $\cQ :\Lambda(G) \rightarrow \mathbb{R} $ is a real valued function, defined up to additive constant, such that its increments between two neighboring vertices $\cS(v^\bullet) \sim \cS(v^\circ)$ are given by the local rule $\cQ(\cS(v^\bullet))-\cQ(\cS(v^\circ)) := |\cS(v^\bullet)-\cS(v^\circ) | $. This means that $\cQ$ adds edge lengths when going from vertices of $ \cS(G^{\circ})$ to vertices of $\cS(G^{\bullet})$ and subtracts lengths when traveling in the other direction. This definition is indeed locally consistent as the alternate sum of edge-lengths in a tangential quadrilateral vanishes. It is easy to see that the function $\cQ $ is automatically $1$-Lipschitz (as a map in the $\cS $ plane). In particular, it appears that some criticality of the model regarding crossing estimates can be simply analyzed via the behavior of the function $\cQ $.

The preceding sentence looks at first sight unclear, since to define the notion of large scale behavior, one should first define a notion of \emph{scale} of the embedding. A priori, there is no natural notion of mesh size of a lattice in a highly irregular grid formed of tangential quadrilaterals. This can be done using the origami map $\cQ $ and the assumption \LipKd\,, as originally defined in \cite{CLR1}. 

\begin{assumpintro}[\LipKd] We say that the embedding $\cS $ satisfies the assumption $\textup{Lip}(\kappa,\delta)$  for some positive constant $\kappa<1$ and some $\delta >0 $ if for any $v,v' $ vertices of $\Lambda(G) $
\begin{equation}
\label{eq:LipKd}
|\cQ(v')-\cQ(v)|\le\kappa\cdot |\cS(v')-\cS (v)|\quad \text{if}\quad |\cS (v')-\cS (v)|\ge\delta.
\end{equation}
\end{assumpintro}

This allows to define the notion of \emph{scale} of the embedding $\cS$.

\begin{definition}
We say that an s-embedding $\cS $ covering an open set $U\subseteq \mathbb{C}$ has a scale $\delta $ for the constant $ \kappa <1$ if 
\begin{equation}
\delta = \delta^{\kappa} = \inf \{ \tilde{\delta} >0, \, \textrm{Lip}(\kappa, \tilde{\delta}) \textrm{ holds} \}.
\end{equation}
In that case, we write $\cS = \cS^{\delta} = \cS^{\delta^\kappa} $ (leaving the $\kappa $ superscript unwritten).
\end{definition}

In words, for some positive $\kappa<1$, the scale of the embedding is the minimal length $\delta$ at which $\cQ$ becomes $\kappa$-Lipschitz. Regarding the results stated in this paper, the dependence on $\kappa$ will play a role in the bounds of the crossing estimates but not in the qualitative results themselves, as long as $\kappa$ remains bounded away from $1$. The dichotomy in the behaviors of the statistical mechanics model occurs when the origami map $\cQ$ has an optimal Lipschitz constant (at large graph distances) strictly smaller than $1$ or exactly equal to $1$. In particular, when speaking about the scaling limit of a sequence of s-embeddings $(\cS^\delta)_{\delta>0}$, it is taken along subsequences of s-embeddings $(\cS^{\delta_n})_{\delta_n}$ with $\delta_{n}\rightarrow 0$ as $n\rightarrow \infty$, and all s-embeddings $\cS^{\delta_n}$ satisfy $\textrm{Lip}(\kappa,\delta_n)$ for the \emph{same} $\kappa<1$. For grids with angles bounded from below and edge lengths comparable to some $\delta$ (setup denoted by \Unif\ in \cite{Che20}), the definition of the scale using \LipKd\ coincides (up to an $O(1)$ factor) with $\delta$. It is notable that there exists an even more natural way to embed a planar graph as a space-like surface in Minkowski space $\mathbb{R}^{2,1}$, taking the origami map as the third coordinate (see the discussion in Section \ref{sub:optimality}). The planar s-embedding setup is then also stable under some naturally associated isometries in $\mathbb{R}^{2,1}$. This passage to $\mathbb{R}^{2,1}$ is not artificial, as it is the natural environment in which to study the continuum scaling limit of fermionic observables in the s-embeddings context (see \cite[Section 2.7]{Che20} or the convergence statements developed in \cite[Chapter 6]{MahPHD}). In the rest of the article, the estimates of the type $O$ will depend on the parameter $\kappa<1$.

\begin{figure}
\includegraphics[clip, scale=0.40]{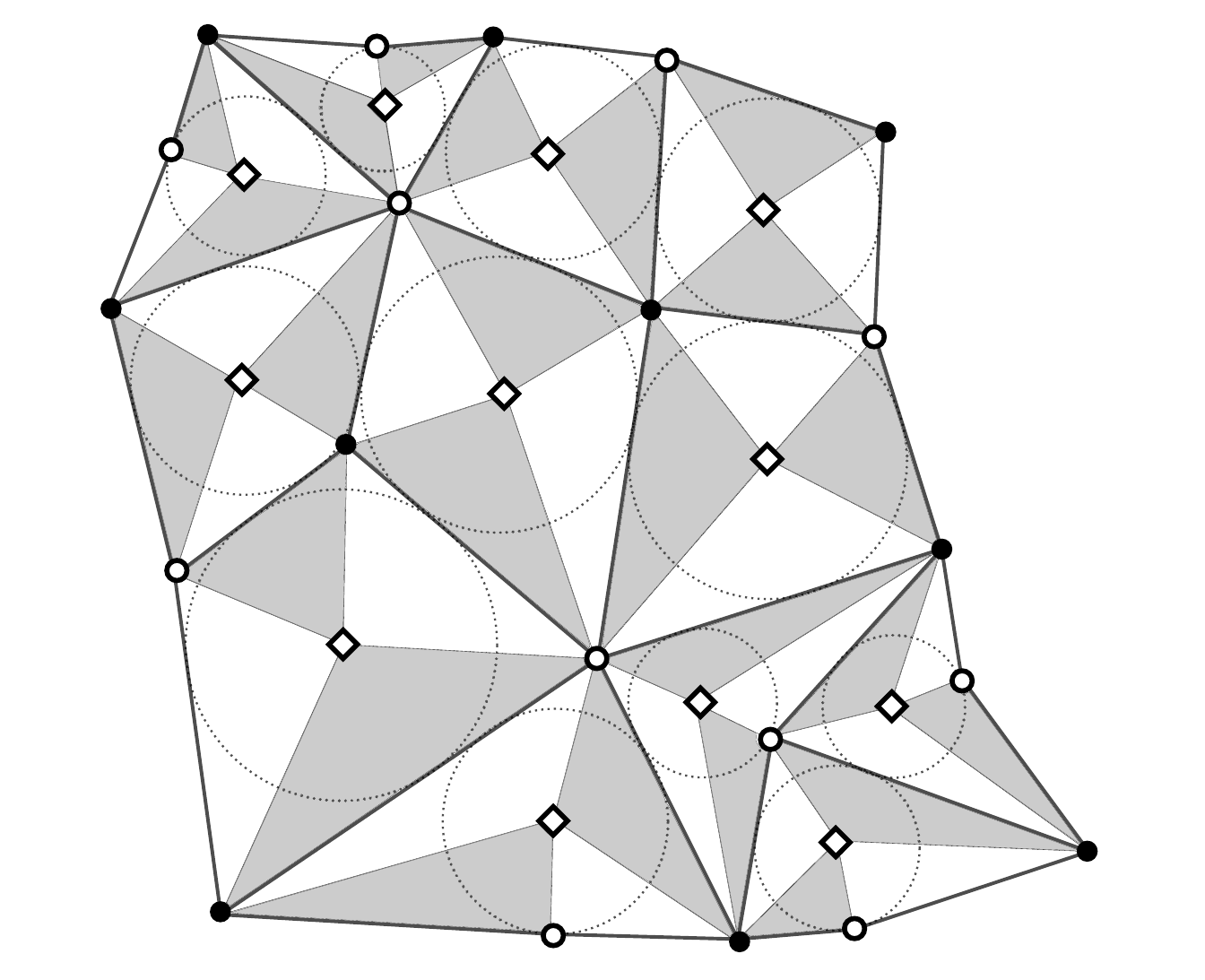}
\caption{A piece of an s-embedding. The vertices of $G^\bullet $ are denoted by black dots, those of $G^\circ $ by white dots, and the center of tangential quadrilaterals with diamonds. The tangential circles are dashed. The bipartite splinting of each face of $\Lambda(G)$ in four triangles corresponds to the dimer identification under the s/t-embeddings correspondence (see \cite[Section 2.3]{Che20} for more details.) that follows the bosonisation rules of \cite{dubedat2011exact}.}
\end{figure}

\subsection{Main Results}

In the current paper, we prove Russo–Seymour–Welsh crossing estimates (see \cite[Chapters 4 and 5]{duminil-parafermions} for precise definitions and for the relation between the spin-Ising model and the FK model, as well as the statement on the critical square lattice) for the FK–Ising model on an s-embedding that satisfies the condition \LipKd\ for some $\kappa < 1$. We believe this is the crucial assumption of physical significance for the theorems of the present paper to hold. The estimates are established starting at some scale $O(\rho)$, provided that the embedding satisfies the non-degeneracy assumption \ExpFat. The assumption \LipKd\ is optimal and cannot be weakened: we present graphs (corresponding to known off-critical systems) where both \LipKd\ (for any $\kappa < 1$) and the RSW property fail. Moreover, as explained in Section \ref{sub:optimality}, one \emph{cannot} hope to prove crossing estimates bounded away from $0$ and $1$ without an assumption of this kind. Indeed, this assumption prevents the construction of a \emph{different embedding of the same statistical mechanics model} that is heavily stretched in one direction, while its optimal Lipschitz constant can be made arbitrarily close to $1$.

In this article, \emph{we do not rely on any bounded angle property, comparable edge-length assumptions, symmetry, or translation invariance}. Instead, we use the scale defined via the assumption \LipKd\ together with discrete complex analysis techniques. Our approach answers a question posed in \cite[Section 1.4 (I)]{Che20} and treats the general case $\cQ\not\to 0$, relaxing constraints on the local and global geometry of the embedding. This provides one of the most flexible frameworks known to date where crossing estimates are available. In this setup, the system behaves qualitatively like the usual critical model on the square lattice: macroscopic wired clusters exist, correlations decay polynomially, and FK interfaces are precompact, among other features.

As already mentioned, we believe that the main conceptual assumption to prove Theorems \ref{thm:positive-magnetization} and \ref{thm:RSW} is exactly \LipKd\ (see the discussion in Section \ref{sub:optimality}). Our proof relies on discrete complex analysis techniques, which require establishing the precompactness of discrete s-holomorphic functions and avoiding potential pathological behaviors, such as exponential blow-up of discrete observables. To achieve this, we add a mild assumption denoted \ExpFat, which prescribes the admissible level of local degeneracy for the embedding. This restriction ensures the precompactness of s-holomorphic functions via Theorem \ref{thm:F-via-HF} and follows the formalism of \cite[Assumption 1.2]{CLR1}.
\begin{assumpintro}\label{assump:ExpFat} We say that a proper s-embeddings $\cS^\delta$ satisfies the assumption \ExpFat\ for some $\rho >0 $ on an open subset $U\subset\C$ if 
\begin{center} after removing all quads $(\cS^\delta)^\dm(z)$ with $r_z\ge \delta\exp(-\rho \delta^{-1})$ from $U$, all vertex-connected components have a diameter at most $\rho$.
\end{center}
\end{assumpintro}

Under the assumptions \LipKd\ and \ExpFat, we show that boxes of size larger than $O(\rho)$ satisfy the usual strong Russo–Seymour–Welsh box crossing property. Before diving into the details, let us informally explain how this reads on the critical square lattice. Fix a large square of size $\Lambda_n:=n\times n$ and consider the Ising model on the faces of this square, with four wired/free/wired/free alternating boundary arcs. This means that the spins attached to the upper horizontal boundary represent a single spin $\sigma_{\textrm{top}}$, while the spins attached to the lower horizontal boundary represent \emph{another} single spin $\sigma_{\textrm{bot}}$. Finally, the spins on the vertical boundaries of $\Lambda_n $ are individual. In this setup, the strong box crossing property at criticality states that the correlation $\mathbb{E}[\sigma_{\textrm{top}}\sigma_{\textrm{bot}}]$ remains bounded away from $0$ and $1$, uniformly in $n$, in sharp contrast with the off-critical phase of the model, where this correlation degenerates to $0$ or $1$. For the random cluster model at the critical value, one can consider the annulus $A_n$ between the boxes $\Lambda_{n/2}$ and $\Lambda_n$, equipped with free boundary conditions at both the inner and outer boundaries of $A_n$. In this setup, the strong box crossing property asserts that, with a positive probability (uniformly bounded from below in $n$), there exists a circuit of open edges in $A_n$ surrounding $\Lambda_{n/2}$. Such crossing estimates were first proven for Bernoulli percolation in the works of Russo \cite{Rus78,Rus81} and Seymour and Welsh \cite{SW78}. The first proof for the FK–Ising model on straight rectangles is due to Duminil-Copin, Hongler, and Nolin \cite{DCHN}, and it was later generalized by Chelkak, Duminil-Copin, and Hongler to rough shapes (depending on the domain's extremal length) in \cite{CDH}. The proof of this theorem begins by establishing a lower bound for the boundary correlation in the spin-Ising model with four alternating boundary conditions. 

Fix $\ell>0$ and consider the square $[-20\ell;20\ell]^2$ centered at the origin of the complex plane. In the following theorem, we consider a piece of a proper s-embedding $\cS^{\delta}$ covering $[-20\ell;20\ell]^2$, satisfying the assumptions \LipKd\ and \ExpFat. We say that a simply connected domain $\cR^\delta \subset \cS^\delta(\Lambda(G))$ is an approximation up to $10\delta$ of the rectangle $\mathcal{R}:=[-\ell;\ell]\times [-3\ell; 3\ell] \subset [-20\ell;20\ell]^2$ if the Hausdorff distance between the discrete boundary of $\cR^\delta$ (seen as a piecewise linear path formed by edges of $\cS^\delta(\Lambda(G))$) and the boundary of $\mathcal{R}$ is smaller than $10\delta$. 

One can split the boundary of $\cR^\delta$ into four arcs, composed of segments of $\cS^{\delta}(\Lambda(G))$ together with four edges of $\cS^{\delta}(\Lambda(G))$ identified with the corners $a^\delta,b^\delta,c^\delta,d^\delta$ (see Figure \ref{fig:Intro}). This allows us to define the Ising weights on $\cR^\delta$, attached to the geometry of $\cS^{\delta}$ via \eqref{eq:theta-from-S}. We then define the probability measure $\mathbb{P}^{\circ \bullet \circ \bullet}_{\cR^\delta}$ corresponding to the Ising model on $\cR^\delta$ with \emph{wired} boundary conditions on the arcs $(b^\delta c^\delta)^\circ$ and $(d^\delta a^\delta)^\circ$, approximating the horizontal segments of $[-\ell;\ell]\times [-3\ell; 3\ell]$, and \emph{free} boundary conditions on the arcs $(a^\delta b^\delta)^\bullet$ and $(c^\delta d^\delta)^\bullet$, approximating the vertical segments. This means that all the spins along the arc $(d^\delta a^\delta)^\circ$ (drawn in red in Figure \ref{fig:Intro}) represent a \emph{single spin} $\sigma_{\textrm{top}}$, while all the spins along the arc $(b^\delta c^\delta)^\circ$ (drawn in green in Figure \ref{fig:Intro}) represent \emph{another single spin} $\sigma_{\textrm{bot}}$. One then has the following theorem.
\begin{theorem}\label{thm:positive-magnetization} 
In the previous setup, there exist positive constants $L_0=L_0(\kappa) $ and $c_0=c_0(\kappa) $, only depending on $\kappa < 1 $, such that for any $l\geq L_0 \rho$ one has
\[
\E^{\circ\bullet\circ\bullet}_{\cR^\delta}[\sigma_{\mathrm{top}}\sigma_{\mathrm{bot}}]\ \ge\ c_0(\kappa).
\]
\end{theorem}
The theorem is a consequence of monotonicity with respect to boundary conditions (see e.g. \cite[Section 4]{duminil-parafermions} and references therein) and the proof made in a similar statement made in Section \ref{sec:proof}, in the case of the special discretizations constructed in Section \ref{sec:geometry}. Note that the same theorem holds for the dual model, defined using Kramers-Wannier duality. The next theorem classically follows from Theorem \ref{thm:positive-magnetization}. Denote the annulus
\begin{equation*}
	\Box(\ell):= \Big( [-3\ell,3\ell] \times [-3\ell,3\ell] \Big)\setminus\Big((-\ell,\ell)\times(-\ell,\ell) \Big).
\end{equation*}
We keep working with a proper $s$-embedding satisfying \LipKd\,, and \ExpFat\, covering the domain  $[-20\ell;20\ell]^2 $. In that context, fix an approximation $\Box^{\delta}(\ell)$ of $\Box(\ell)$ up to $10\delta$ (in the Hausdorff sense), whose boundary is made by edges of $\cS^{\delta}(\Lambda(G))$. This allows to define the probability measure $\mathbb{P}^{\operatorname{free}}_{\Box^{\delta}(\ell)}$ for the random cluster model on $\Box^{\delta}(\ell)$ with weights coming from the Edwards-Sokal coupling of the Ising model on $\cS^\delta$, with free boundary conditions at both the inner and outer boundaries of $\Box^{\delta}(\ell)$. Recall that we call a \emph{circuit} of open edges in $\Box^{\delta} (\ell)$ a sequence of open neighbouring edges (for the random cluster model) that surround the inner boundary of $\Box^{\delta}(\ell)$. Then one has the following theorem.
\begin{theorem}\label{thm:RSW}
In the previous setup, there exist positive constants $L'_0 $,$p_0$, only depending on $\kappa$, such that for any $\ell \geq L'_0 \rho $, one has
\[
 \mathbb{P}^{\operatorname{free}}_{\Box^{\delta} (\ell)}\bigl[\,\mathrm{there~exists~a~circuit~of~open~edges~in}~{\Box^{\delta} (\ell)}\,\bigr]\ \ge\ p_0.
\]
A similar uniform estimate holds for the dual model.
\end{theorem}

\begin{figure}
\hspace{-10mm}\begin{minipage}{0.33\textwidth}
\includegraphics[clip, width=1.4\textwidth]{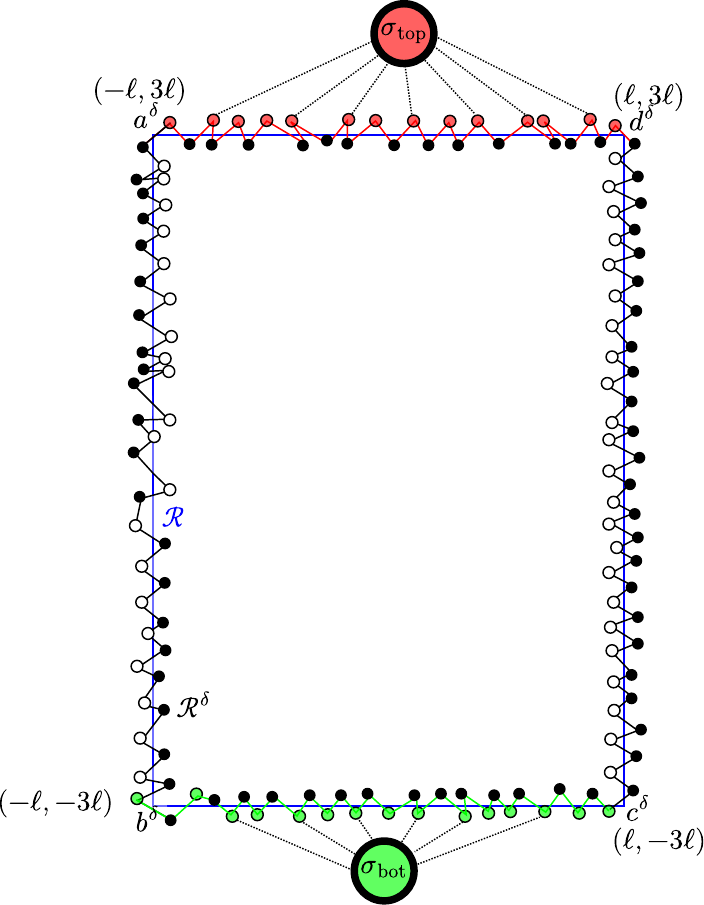}
\end{minipage}\hskip 0.15\textwidth \begin{minipage}{0.33\textwidth}
\includegraphics[clip, width=1.4\textwidth]{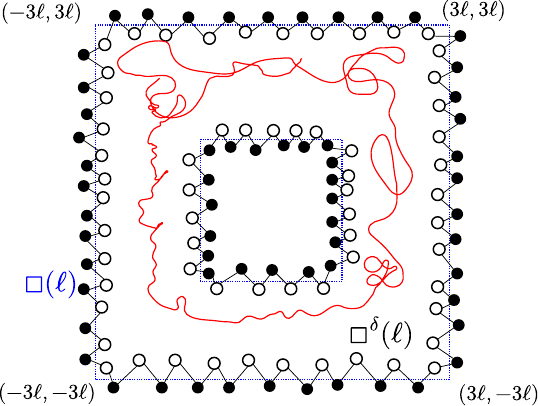}
\end{minipage}
\caption{Setup of Theorems \ref{thm:positive-magnetization} and \ref{thm:RSW}. (Left): Rectangle \textcolor{blue}{$\mathcal{R}$} approximated by the discrete domain $\mathcal{R}^\delta \subset \cS^{\delta}$, whose respective boundaries lie at distance at most $ 10\delta $ from each other. All the spins along the red arc $(d^\delta a^\delta)^\circ$ represent a single spin $\sigma_{\textrm{top}}$ and the all spins along the green arc $(b^\delta c^\delta)^\circ$ (drawn in green in Figure \ref{fig:Intro}) represent a single spin $\sigma_{\textrm{bot}}$. (Right): Annulus \textcolor{blue}{$\Box((\ell)$} approximated by $\Box^{\delta}(\ell) \subset \cS^{\delta}$,  whose respective boundaries lie at distance at most $ 10\delta $ from each other. The event of Theorem \ref{thm:RSW} corresponds to the existence of a circuit of open edges (drawn in red) that surrounds the inner free boundary of $\Box^{\delta}(\ell)$.}
	\label{fig:Intro}
\end{figure}

Once Theorem \ref{thm:positive-magnetization} is derived, this second theorem follows directly from the strategy of \cite[Proposition~2.10]{duminil-garban-pete}, recalled in detail in \cite[Section 5.6]{Che20}. In the case of s-embeddings satisfying the assumption \Unif, both theorems hold at every scale. It is even possible to weaken the local assumption \ExpFat\ by requiring the presence of a large surrounding circuit formed of sufficiently fat faces (i.e., with a large enough inscribed circle). Let us note that the constants $L_0$ and $L'_0$ can, in principle, be made explicit in terms of $\kappa$, by carefully following the arguments in \cite[Section 6]{CLR1} and \cite[Section 2]{Che20}. The assumption \ExpFat\ is scale-invariant in the variable $\delta \rho^{-1}$. Thus, one can informally restate the two previous theorems as follows, assuming $\rho=1$: if an s-embedding has a $\kappa<1$-Lipschitz origami map starting at some scale $\delta$, and if all connected components of tangential quadrilaterals with radius smaller than $\delta \exp(-\delta^{-1})$ are of size at most $1$, then the usual RSW theory holds at macroscopic distances. One can see in the proof that the assumption on $\delta \exp(-\delta^{-1})$-fat quads can also be weakened to $\delta \exp(-\gamma \delta^{-1})$-fat quads for any fixed $\gamma>0$.

To apply the aforementioned strategy, one must first ensure the possibility of constructing a \emph{proper} s-embedding associated to the abstract weighted graph $(G,x)$, with faces of $\Lambda(G)$ delimited by straight segments and without overlaps. The existence of proper s-embeddings associated to any weighted finite graph $G$ \textbf{always holds}, without any particular requirement on the combinatorial structure of the outer face (see Section \ref{sub:explicit-construction}), and such embeddings can even be constructed using an \emph{explicit algorithm} from \cite{KLRR}. For infinite grids, it remains an open question whether one can find a proper full-plane embedding in the generic case; this often requires a careful solution to the propagation equation \eqref{eq:3-terms} (see the brief discussion at the end of the construction of s-embeddings using a pair of holomorphic functions in Section \ref{sub:explicit-construction}, which presents a natural condition on the behavior at infinity of solutions to the propagation equation that lead to proper embeddings, following \cite[Appendix]{CLR2}). We highlight that, among the consequences of the current paper (see Section \ref{sub:explicit-construction} for detailed discussion), one obtains:
\begin{itemize}
\item An alternative derivation of the RSW property for the FK-Ising model on critical Z-invariant isoradial lattices, as in \cite{ChSmi2}. Our result extends beyond the scope of Chelkak and Smirnov, as it allows the bounded angle property to be replaced by the assumption \ExpFat. This replacement, in particular, allows a more concise derivation of the RSW property for the FK representation of the quantum Ising model \cite[Section 5]{DCLM-universality}.

\item An alternative derivation of the RSW property for the \emph{massive} Z-invariant model on isoradial grids, as in \cite{park-iso}, using the re-embedding procedure of \cite[Section 3.3]{Cim-universality}.

\item A new derivation of the RSW property for the Layered Ising model presented in Section \ref{sub:explicit-construction}, with i.i.d. weakly random coupling constants that are typically \emph{very large} compared to the classical characteristic length. In particular, randomness in successive layers averages sufficiently to preserve the criticality of the model. The same strategy provides a concise proof of the RSW property for the \emph{massive} square lattice.

\item An alternative derivation of the RSW property for doubly-periodic graphs, as given in \cite{Che20}. In particular, there is no need to use the doubly-periodic canonical embedding to deduce that Theorems \ref{thm:positive-magnetization} and \ref{thm:RSW} hold. More generally, our results apply to `flat' origami functions, including circle patterns of Lis with bounded angles \cite{lis-kites}, already treated in \cite{Che20}.

\item A new derivation of the RSW property for the FK-Ising model on circle patterns introduced in \cite{lis-kites}, replacing the bounded angle property by the \ExpFat\ assumption (the limiting origami map in this setup is automatically $0$). In particular, our result applies to finite pieces of random triangulations arising from the discrete mating-of-trees model of Duplantier, Gwynne, Miller, and Sheffield \cite{duplantier2021liouville,gwynne2021tutte}, with Ising weights naturally attached to the associated circle packing. In \cite{GGJN}, it was proven that this random map model has no large circles in bounded regions with high probability; combined with refined estimates on the typical number of vertices, this fits our framework.

\item A new derivation of the RSW property for tilings of the plane by tangential quadrilaterals (e.g., the IC nets constructed by Akopyan and Bobenko \cite{akopyan2018incircular}) with an origami map whose Lipschitz constant is asymptotically smaller than $1$. In Section \ref{sub:explicit-construction}, we present a method to use any pair $(f,g)$ of holomorphic functions on the unit disc satisfying $\Im[\overline{f}g]>0$ to construct a new critical lattice. This method, which can be generalized, provides a simple way to construct many new critical lattices from existing ones via a discrete Weierstrass parametrization of the space-like surface $(\int  \Im[  2fg], \int \Im[  f^2-g^2], \int \Im[  f^2+g^2]) \subseteq \mathbb{R}^{2,1} $, following \cite[Section 3.3]{Cim-universality}
\end{itemize}
To prove Theorem \ref{thm:positive-magnetization}, we use the flexibility of the s-embeddings setup to extend the discretization of a well-chosen rectangle in $\cS^{\delta}$ and paste it to a piece of the square lattice (see Figure \ref{fig:extension_shape} in Section \ref{sec:geometry}). Once this extension is performed, the core of the proof proceeds by contradiction: we assume that Theorem \ref{thm:positive-magnetization} fails and then identify an inconsistency in the boundary behavior of discrete observables and their continuous counterparts. Adding these pieces of the square lattice greatly simplifies the boundary analysis in the contradiction argument and provides a much more digestible proof of the inconsistency.

Let us finally mention that the criticality of a given proper tiling of tangential quads can also be interpreted from the perspective of the dependency on the inverse temperature. Namely, as explained by Lis at the end of \cite[Section 6]{lis-kites}, in the case of s-embeddings satisfying the assumption \Unif, one can assign weights $x(e) = \tan \frac{\theta_e}{2}$, where the abstract angle $\theta_e$ is recovered using the geometrical formula \eqref{eq:theta-from-S}. Consider the model at inverse temperature $\beta$ in the usual formulation \eqref{eq:intro-Zcirc}, where the coupling constants are given by $\tanh(J_e) = x_e$. Then the primal model (with spins on $G^\circ$) exhibits positive magnetization for $\beta>1$, while the dual model (with spins on $G^\bullet$) exhibits positive magnetization for $\beta<1$. The result of \cite{Che20}, which we strengthen in the present article, ensures that the critical point is indeed $\beta=1$, since there is no simultaneous positive magnetization for the primal and dual critical models, as a consequence of RSW-type estimates.

\subsection{Novelties of the paper, related works and open questions}\label{sub:novelties}

This article provides a general framework that we hope will open a path to generalize already known results to new graph settings, and already allows the construction of a substantial number of new `critical' Ising models, for instance using tilings derived from a pair of discretized holomorphic functions. In all previously known setups, symmetries, integrability, a bounded number of neighbors, and the finite energy property play a key role in the analysis; here, we bypass these requirements. The idea of pasting a piece of an already understood grid to obtain a concise and more transparent contradiction (compared to the arguments developed in \cite{ChSmi2} and \cite{Che20}) is new and relies on the fact that boundary-to-boundary connection probabilities of the FK model can be studied in larger domains, provided that the (abstract) layers where we `weld' the two graphs do not completely break connectivity. From our perspective, this allows us to formulate a notion of criticality for the Ising model on a generic planar graph, following \cite[$\mathbf{(II)}$ in Section 1.5]{Che20}:
\begin{center}
Which are the spectral properties of the propagation operator \eqref{eq:3-terms} (or the associated Kac-Ward matrix) that imply the existence of a complex valued solution $\mathcal{X}$ to the propagation equation such that the associated s-embedding $\cS_{\cX}$ satisfies \LipKd\, for a large enough $\delta$?
\end{center}
This question has been investigated in a periodic setup in a joint work with Chelkak and Hongler (see \cite[Section 5]{CHM-zig-zag-Ising}) and there is related to spectral characteristics of the operator near the bottom of its spectrum. In that article, one even relates features of the canonical embedding in the Euclidean plane to the asymptotics of the integrated density of states near $\lambda = 0$ (see \cite[Equation (5.8)]{CHM-zig-zag-Ising}). Generalizing this understanding would allow a precise reformulation of the present notion of criticality from the perspective of the Kac-Ward operator.

We expect that off-critical models lead to embeddings with an optimal Lipschitz constant equal to one, as discussed in Section \ref{sub:optimality} regarding near-critical and off-critical models. One may also attempt to prove the convergence of crossing probabilities by expressing them in the four-point setup via quasi-conformal uniformization, in a fashion similar to the critical or near-critical case \cite[Section 5]{park-iso}. This approach is currently being investigated in an article in preparation with Park \cite{MahPar23a}. More generally, we view the current paper as a first step towards proving the existence of a scaling limit for the Ising model on general s-embeddings. The associated continuous limit and its relation to Lorentz geometry and quasi-conformal maps (see \cite{CLR2} and \cite[Section 2.7]{Che20}) will be developed jointly with Park \cite{MahPar23a,MahPar23b}, following the route started in the author’s thesis \cite[Chapter 6]{MahPHD}. In particular, our result ensures, in the classical sense, the precompactness of FK interfaces on general s-embeddings satisfying \LipKd\ and \ExpFat\ (see, e.g., \cite{KemSmi2}).

An interesting line of research would be to extend the strategy of \cite{DMT21} to prove crossing estimates that are uniform with respect to the local structure of the lattice and boundary conditions, by bootstrapping Theorem \ref{thm:RSW} for \Unif-type grids. We are currently unable to handle this problem, since such a strategy requires a comparison between primal and dual arm exponents in the half-plane, which is not straightforward without additional symmetries. From our perspective, the most interesting question is to use the meta-framework developed here to study crossing probabilities on random planar maps (with appropriately chosen coupling constants) or on $\mathbb{Z}^2$ with random coupling constants.

\noindent {\bf Acknowledgements.} The author is deeply indebted to Dmitry Chelkak for introducing him to this field and constant support. We are also grateful to Sung-Chul Park for valuable discussions and carefully reading earlier versions of this manuscript as well as Mikhail Basok for crucial help regarding Beltrami equations. 
An earlier version of this text was part of the thesis of the author, completed when he was a student at Ecole Normale Supérieure. R.M. is grateful for the support of the institution. The author is also grateful to Niklas Affoter, Nikolai Bobenko, Cédric Boutillier, Béatrice de Tilière, Hugo Duminil-Copin,  Trishen Gunaratnam, Konstantin Izyurov, Emmanuel Kammerer, Dmitrii Krachun, Benoit Laslier, Marcin Lis, Ioan Manolescu, Seginus Mowlavi, Mendes Oulamara, Christophoros Panagiotis, Romain Panis, Sanjay Ramassamy, Stanislav Smirnov and Yijun Wan for useful discussions and helpful remarks. Finally, we would like to thank the anonymous referees for comments that improved the readability of this manuscript. This research is funded by Swiss National Science Foundation and the NCCR SwissMAP.

\section{Definitions and crash introduction to the s-embeddings formalism}\label{sec:definitions}
\setcounter{equation}{0}

We recall in this section the construction of s-embeddings introduced in \cite[Section 3]{Che20} and the regularity theory of s-holomorphic functions, both based upon a complexification procedure of the Kadanoff-Ceva formalism. The notations we use in this paper follow exactly those of \cite{Che20} and agree with those of \cite[Section~3]{CCK} and \cite{Ch-ICM18}. We proceed below without giving any proof, referring to  \cite[Section 2]{Che20} for more details. The overall idea of the construction is to start with an abstract weighted planar graph and construct an embedding where discrete complex analysis techniques are available.

\begin{figure}
\begin{minipage}{0.325\textwidth}
\includegraphics[clip, width=1.2\textwidth]{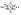}
\end{minipage}\hskip 0.10\textwidth \begin{minipage}{0.33\textwidth}
\includegraphics[clip, width=1.2\textwidth]{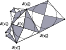}
\end{minipage}
\caption{(Left) Local notations of a graph given an arbitrary embedding of a planar graph for a quad $z \in \diamondsuit (G)$. Vertices of the primal graph $G^\bullet $ are drawn as black dots, vertices of the dual graph $G^\circ $ (which correspond to faces of $G^\bullet $) are drawn as white dots, and corners (which correspond to edges of the bipartite graph $G^\bullet \cup G^\circ $) are drawn as triangles. We present here a piece of the \emph{double cover} of the corner graph that branches around $z$. Neighboring corners in the double cover are linked with dashes. (Right) A small piece of the associated face in an s-embedding, together with neighboring faces. Each face in this proper s-embedding is a tangential quadrilateral, not necessarily convex. The bipartite spliting of each face of $\Lambda(G)$ in four triangles corresponds to the dimer identification under the s/t-embeddings correspondence (see \cite[Section 2.3 ]{Che20} for more details.)}
\label{fig:graph-notations}
\end{figure}

\subsection{Notation and Kadanoff--Ceva formalism}\label{sub:notation}

We fix $G$ a planar graph (allowing multi-edges and vertices of degree two  but forbidding loops and vertices of degree one) with the combinatorics of the plane or of the sphere, considered up to homeomorphisms preserving cyclic ordering of edges around each vertex. In the sphere case, we prescribe one of the faces of $G$ and call it the outer face of $G$. We denote $G=G^\bullet$ the original graph whose vertices are denoted by $v^\bullet\in G^\bullet$ and $G^\circ$ its dual, whose vertices are denoted by $v^\circ\in G^\circ$. The faces of the bipartite graph $\Lambda(G):=G^\circ\cup G^\bullet$ (with natural incidence relation) are in bijection with edges of $G$. We also denote $\Dm(G)$ the graph dual to $\Lambda(G)$, whose vertices are often denoted by \mbox{$z\in\Dm(G)$} and called quads. Finally, we denote by $\Upsilon(G)$ the medial graph of $\Lambda(G)$. The vertices of $\Upsilon(G)$ are in bijection with edges $(v^\bullet v^\circ)$ of $\Lambda(G)$. The vertices $c\in\Upsilon(G)$  are called corners of $G$. To make the formalism consistent, one needs to consider several double covers of $\Upsilon(G)$, see e.g. \cite[Fig.~27]{Mercat-CMP} or \cite[Fig 3.A]{Che20} for relevant pictures. Denote by  $\Upsilon^\times(G)$ the double cover that branches over all faces of $\Upsilon(G)$ (each $v^\bullet \in G^\bullet, v^\circ\in G^\circ , z\in \Dm(G)$). When $G$ is finite, this definition of the double cover remains meaningful as $\#(G^\bullet)+\#(G^\circ)+\#(\Dm(G))$ is even due to the Euler theorem. Given $\varpi=\{\vbullet{m},\vcirc{n}\}\subset \Lambda(G)$ where $n,m$ are even, denote by $\Upsilon^\times_\varpi(G)$ the double cover of $\Upsilon(G)$ branching over all its faces except those $\varpi$, and by $\Upsilon_\varpi(G)$ the double cover of $\Upsilon(G)$  branching only over those $\varpi$. We call a \emph{spinor} a function defined on one of the aforementioned double covers whose value at two different lifts of the same corner differ by a multiplicative factor $-1$.

In this paper, we consider the Ising model on faces of $G$, including the outer face in the disc case, i.e. the model assigns $\pm1 $ random variables to vertices of $G^\circ$ with a partition function given by \eqref{eq:intro-Zcirc}. The domain walls representation \cite[Section 1.2]{CCK} (also called low-temperature expansion) assigns a spin configuration $\sigma:G^\circ\to\{\pm 1\}$ to a subset $C$ of edges of $G$ that  separates spins of opposite signs; this expansion is a $2$-to-$1$ mapping from spin configurations onto the set $\cE(G)$ of even subgraphs of $G$, depending on the value of the spin of the outer face.

Given $\vcirc{n}\in G^\circ$ where $n$ is even, fix a subgraph $\gamma^\circ=\gamma_{[\vcirc{n}]}\subset G^\circ$ with odd degree at vertices of $\vcirc{n}$ and even degree at all other vertices of $G^\circ$. One can represent such configuration as a collection of paths on $G^\circ$ linking pairwise vertices of $\vcirc{n}$. Denote
\[
x_{[\vcirc{n}]}(e)\ :=\ (-1)^{e\cdot\gamma_{[\vcirc{n}]}}\,x(e),\quad e\in E(G),
\]
where $e\cdot\gamma=0$ if $e$ doesn't cross $\gamma$ and $e\cdot\gamma=1$ otherwise. One can see that 
\begin{equation}
\label{eq:Esigma}
\textstyle \mathbb E\big[\svcirc{n}\big]\ =\ {x_{[\vcirc{n}]}(\cE(G))}\big/{x(\cE(G))},
\end{equation}
where $x(\cE(G)):=\sum_{C\in\cE(G)}x(C)$, $x(C):=\prod_{e\in C}x(e)$, and similarly for~$x_{[\vcirc{n}]}$.

For $m$ even and $\vbullet{m}\in G^\bullet$, fix again subgraph $\gamma^\bullet=\gamma^{[\vbullet{m}]}\subset G^\bullet$ with even degree at all vertices of $G^\bullet$ except those of $\vbullet{m}$. Following the formalism of Kadanoff and Ceva \cite{kadanoff-ceva}, one changes the signs of the interaction constants $J_e\mapsto -J_e$ on edges $e\in\gamma^\bullet$. This inversion (which is equivalent to replacing $x(e)$ by $x(e)^{-1}$ along the edges of $\gamma^\bullet$) makes the model anti-ferromagnetic near $\gamma^\bullet$, favoring locally configurations with nonaligned spins, and is denoted by the notation $\muvbullet{m}$. More precisely, we introduce the random variable (which depends on the choice of $\gamma^\bullet$)
\[
\textstyle \muvbullet{m}\ :=\ \exp\big[-2\beta\sum_{e\in\gamma^{[\vbullet{m}]}}J_e\sigma_{v^\circ_-(e)}\sigma_{v^\circ_+(e)}\,\big]\,,
\]
the domain walls representation shows that (e.g. \cite[Propositon 1.3]{CCK})
\begin{equation}
\label{eq:Emu}
\textstyle \mathbb E\big[\muvbullet{m}\big]\ =\ x(\cE^{[\vbullet{m}]}(G))\big/{x(\cE(G))},
\end{equation}
where $\cE^{[\vbullet{m}]}$ denotes the set of subgraphs of $G$ with even degrees at all vertices except at those of $\vbullet{m}$ which have odd degrees at the last mentioned. Let us mention that when taking expectations in \eqref{eq:Emu}, the result does not  depend on the choice of the path $\gamma^\bullet$. One can also generalize \eqref{eq:Esigma} and \eqref{eq:Emu} mixing the presence of spins and disorder, which reads as (e.g. \cite[Propositon 3.3]{CCK})
\begin{equation}
\label{eq:Emusigma}
\textstyle \mathbb E\big[\muvbullet{m}\svcirc{n}\big]\ =\ x_{[\vcirc{n}]}(\cE^{[\vbullet{m}]}(G))\big/{x(\cE(G))},
\end{equation}
where $\muvbullet{m}$ are understood as above. The \emph{sign} of this last expression does depend on the parity number of intersections between $\gamma^\circ$ and $\gamma^\bullet$. There is no canonical way to choose of that sign in \eqref{eq:Emusigma} staying on the Cartesian product $(G^\bullet)^{\times m}\!\times (G^\circ)^{\times n}$. However, one can first fix $\cS:\Lambda(G)\to\C$ to be an arbitrarily chosen embedding of $G$, and consider a natural double cover of this Cartesian product, whose branching structure is the one of the spinor $[\,\prod_{p=1}^m\prod_{q=1}^n(\cS(v^\bullet_p)-\cS(v^\circ_q))\,]^{1/2}$. As discussed in great details in \cite[Section 2.2]{CHI-mixed}, the expectations of the form \eqref{eq:Emusigma} can be seen as spinors on the above described double cover of $(G^\bullet)^{\times m}\!\times (G^\circ)^{\times n}$. When treating mixed correlation of the type \eqref{eq:Emusigma}, an extension of the usual Kramers-Wannier duality (again \cite[Propositon 3.3]{CCK}) implies that the roles played by the graphs $G^\bullet$ and $G^\circ$ are now equivalent.

We are now going to consider special correlators of the form \eqref{eq:Emusigma}, in the case where one of the disorders $v^\bullet(c)\in G^\bullet$ and one of the spins $v^\circ(c)\in G^\circ$, are neighbors in $\Lambda(G)$, and separated by a corner $c \in \Upsilon(G)$. In that case, one can \emph{formally} denote the fermion at the corner $c$ by
\begin{equation}
\label{eq:KC-chi-def}
\chi_c:=\mu_{v^\bullet(c)}\sigma_{v^\circ(c)},
\end{equation}
Using equation \eqref{eq:Emusigma}, one can then define the Kadanoff--Ceva \emph{fermionic observables} by
\begin{equation}
\label{eq:KC-fermions}
X_{\varpi}(c):=\E[\,\chi_c \mu_{v_1^\bullet}\ldots\mu_{v_{m-1}^\bullet}\sigma_{v_1^\circ}\ldots\sigma_{v_{n-1}^\circ} ].
\end{equation}
This definition is purely abstract and doesn't require an embedding. Following the above remarks, one can see that $X_\varpi(c)$ is a priori defined up to the sign, but the definition becomes fully legitimate when passing to $\Upsilon^\times_\varpi(G)$. Around a quad $z=(v_0^\bullet,v_0^\circ,v_1^\bullet,v_1^\circ)$ whose vertices are listed in the counterclockwise order (see \cite[Figure 3.A]{Che20} for the notation), that Kadanoff--Ceva fermionic observables satisfy a local linear propagation equation, whose coefficients are determined by the Ising interaction parameters. This propagation equation appeared in the works of \cite{dotsenko1983critical}, \cite{perk1980quadratic} and \cite[Section 4.3]{Mercat-CMP}) and reads as follows:
\begin{equation}
\label{eq:3-terms}
X(c_{pq})=X(c_{p,1-q})\cos\theta_z+X(c_{1-p,q})\sin\theta_z,
\end{equation}
where the corner $c_{pq}$ is identified as $c_{pq}=(v^\bullet_p v^\circ_q)$, the lifts of $c_{pq}$, $c_{p,1-q}$ and of $c_{1-p,q}$ to $\Upsilon^\times_\varpi(G)$ are neighbors, and the angle $\theta_z$ corresponds to the abstract parametrization \eqref{eq:x=tan-theta} of the edge of $G^{\bullet}$ which corresponds to the quad centered at $z$. One can easily show that solutions to \eqref{eq:3-terms} are automatically spinors on $\Upsilon^\times_\varpi(G)$.

\medskip

We conclude this reminder on Kadanoff-Ceva correlators by recalling the definition of the spinor $\eta_c $, that is a special solution (i.e. with some geometrical interpretation) to the propagation equation \eqref{eq:3-terms} on isoradial grids. Given any embedding $\cS:\Lambda(G)\to\C$ of $\Lambda(G)$ into the complex plane, denote following \cite{ChSmi2}
\begin{equation} \label{eq:def-eta}
\eta_c:=\varsigma\cdot \exp\big[-\tfrac{i}{2}\arg(\cS(v^\bullet(c))-\cS(v^\circ(c)))\big],\qquad \varsigma:=e^{i\frac{\pi}{4}},
\end{equation}
where the prefactor $\varsigma=e^{i\frac\pi 4}$ is chosen for convenience reasons. As explained previously, one can once again avoid the sign ambiguity in the definition \eqref{eq:def-eta} by passing to the double cover $\Upsilon^\times(G)$, understanding the products $\eta_c X_\varpi(c):\Upsilon_\varpi(G)\to\C$ as defined on the double cover $\Upsilon_\varpi(G)$ that only branches over $\varpi$. Note that below, we use the notation \eqref{eq:def-eta} even when $\cS$ is not isoradial.

\subsection{Definition of s-embeddings}\label{sub:semb-definition}

We present now in a concise way the explicit embedding procedure introduced in \cite[Section 6]{Ch-ICM18} and then developed in great details in \cite{Che20}. We start by recalling the concrete definition of an s-embedding given in \cite[Definition 2.1]{Che20}, using the Kadanoff-Ceva formalism recalled in Section \ref{sub:notation}. The general idea is to use a solution to \eqref{eq:3-terms} to construct the picture. 

\begin{definition}\label{def:cS-def}
Let $(G,x)$ be a weighted planar graph with the combinatorics of the plane and $\cX:\Upsilon^\times(G)\to\C$ a solution to the propagation equation \eqref{eq:3-terms}. We say that $\cS=\cS_\cX:\! \Lambda(G)\to\C$ is an s-embedding of $(G,x)$ associated to $\cX$ if for each $c\in\Upsilon^\times(G)$, we have
\begin{equation}
\label{eq:cS-def}
\cS(v^\bullet(c))-\cS(v^\circ(c))=(\cX(c))^2.
\end{equation}
For $z\in\Dm(G)$, denote by $\cS^\dm(z)\subset\C$ the quadrilateral whose vertices are points $\cS(v_0^\bullet(z))$, $\cS(v_0^\circ(z))$, $\cS(v_1^\bullet(z))$, $\cS(v_1^\circ(z))$. The s-embedding $\cS$ is said to be \emph{proper} if the quadrilaterals $\cS^\dm(z) =(\cS(v_0^\bullet(z)) \cS(v_0^\circ(z)) \cS(v_1^\bullet(z)) \cS(v_1^\circ(z)))$ do not overlap with each other, and is said to be non-degenerate if no quads $\cS^\dm(z)$ degenerates to a segment. In particular, the convexity of $\cS^\dm(z)$ is not required.
\end{definition}
Let us make clear that it is not at all automatic that, given any solution $\cX $ to the propagation equation, the associated picture $\mathcal{S}_{\cX}$ is proper, and finding a solution to \eqref{eq:3-terms} that leads to a non-degenerate proper picture is a non-trivial step. One can also extend the definition of $\cS$ to the set $\Dm(G)$ by setting \cite[Equation (2.5)]{Che20}
\begin{equation}\label{eq:cS(z)-def}
\begin{array}{l}
\cS(v_p^\bullet(z))-\cS(z):=\cX(c_{p0})\cX(c_{p1})\cos\theta_z,\\[2pt]
\cS(v_q^\circ(z))-\cS(z):=-\cX(c_{0q})\cX(c_{1q})\sin\theta_z,
\end{array}
\end{equation}
where $c_{p0}$ and $c_{p1}$ (respectively, $c_{0q}$ and $c_{1q}$) are neighbors on $\Upsilon^\times(G)$.
The propagation equation \eqref{eq:3-terms} implies directly the consistency of both definitions \eqref{eq:cS-def} and \eqref{eq:cS(z)-def}.

\bigskip
The second object of crucial relevance in the s-embeddings framework is the so-called origami map. It is the large scale properties of the origami map that will indicate whether one can interpret the abstract graph to a (near)-critical system or not. We recall now the definition \cite[Definition 2.2]{Che20} which can also be found (with appropriate Ising/dimers identifications) in \cite{KLRR,CLR1}.

\begin{definition}\label{def:cQ-def}
Given~$\cS=\cS_\cX$, one can construct the \emph{origami} function denoted by \mbox{$\cQ=\cQ_\cX:\Lambda(G)\to\R$}, as a real valued function (defined up to a global additive constant) by declaring its increments between $v^{\bullet}(c) $ and $v^\circ (c) $ to be
\begin{equation}
\label{eq:cQ-def}
\cQ(v^\bullet(c))-\cQ(v^\circ(c))\ :=\ |\cX(c)|^2\,=\,|\cS(v^\bullet(c))-\cS(v^\circ(c))|\,.
\end{equation}
\end{definition}

Once again, the propagation equation \eqref{eq:3-terms} implies directly the consistency of the definition \eqref{def:cQ-def}. In words, this implies that alternates sum of edge-length around a quad vanishes. This ensures the image $\cS^\dm(z)$ of a quad into the complex plane is a quadrilateral tangential to a circle centered at the point $\cS(z)$ given by \eqref{eq:cS(z)-def}. The point $\cS(z)$ is the intersection point of the four bisectors of the angles of the tangential quadrilateral $\cS^\dm(z)$. We denote by $r_z$ the radius of the circle, which can be recovered from the values of $\chi $, using e.g. \cite[Equation (2.7)]{Che20}. If one denotes $\phi_{vz}$ the half-angle of the quad $\cS^\dm(z)$ at $\cS(v)$, the Ising weight $\theta_z$ (in the parametretrization \eqref{eq:x=tan-theta}) can be recovered from the angles in $\cS $ plane using the formula \cite[Equation (2.8)]{Che20}
\begin{equation}
\label{eq:theta-from-S}
\tan\theta_z\ =\ \biggl(\frac{\sin\phi_{v_0^\bullet z}\sin\phi_{v_1^\bullet z}}{\sin\phi_{v_0^\circ z}\sin\phi_{v_1^\circ z}}\biggr)^{\!1/2}.
\end{equation}

As explained in \cite[Section 2.3]{Che20} (see also \cite[Section 7]{KLRR}), one can see that if $\cS$ is proper and non-degenerate, the map $\cS:\Lambda(G)\cup\Dm(G)\to\C$ is a \emph{t-embedding} $\cT$ and $\eta_{c^\bullet}=\eta_{c^\circ}:= \overline{\varsigma}\eta_c$ is an origami square root (see \cite[Definition~2.4]{CLR1}) of $\cT$. This identification allows (e.g.  \cite[Appendix]{CLR1}) to extend the origami map $\cQ$ into a piece-wise linear way to the \emph{entire} plane (taking complex values inside the faces of the t-embedding), and not only on edges of $\Lambda(G)$.

\subsection{S-holomorphic functions and associated functions ${H_F}$ and $I_\mathbb{C} $}\label{sub:HF-def}
We recall now the central notion of \emph{s-holomorphic functions} (generalized to s-embeddings in \cite{Che20}), introduced first for the critical square grid by Smirnov \cite[Definition 3.1]{Smirnov_Ising} and generalized to the isoradial context by Chelkak and Smirnov in \cite[Definition~3.1]{ChSmi2} (it is in the latter paper that this name was coined). That notion is at the heart of the use of discrete complex analysis techniques to study the Ising model. Under the link (briefly mentioned above) between s and t-embeddings, the s-holomorphic functions are special cases of \emph{t-holomorphic} functions introduced in the dimers context in \cite[Definition 3.2]{CLR1}, which allows to use the regularity theory proved in that later paper. Below, we denote by $\textrm{Pr}[\beta,\eta\mathbb{R}]:=1/2\big(\beta +\overline{\beta} \eta^2 \big)$ the usual projection of the complex number $\beta$ on the line $\eta\mathbb{R}$. We recall now the general definition on s-holomorphic functions, given in \cite[Definition 2.4]{Che20}.

\begin{definition}\label{def:s-hol}
A function $F$ defined on a subset of $\Dm(G)$ is called s-holomorphic if
\begin{equation}
\label{eq:s-hol}
\textrm{Pr}[F(z),\eta_c\R] = \textrm{Pr}[F(z'),\eta_c\R]
\end{equation}
for each pair of quads $z,z'\in\Dm(G)$ adjacent to the same edge $(v^\circ(c)v^\bullet(c))$ in $\cS$.
\end{definition}

The next proposition generalizes beyond the isoradial setup (e.g. \cite[Lemma~3.4]{ChSmi2}) the (bijective) link between real valued solutions to the propagation equation \eqref{eq:3-terms} and s-holomorphic functions. This link was already presented in \cite[Proposition 2.5]{Che20} and in \cite[Appendix]{CLR1}.

\begin{proposition}\label{prop:shol=3term} Let $\cS=\cS_\cX$ be a proper s-embedding and $F$ an s-holomorphic on a subset of $\Dm(G)$. Then, the spinor $X$ defined at corners \mbox{$c\in\Upsilon^\times(G)$} belonging to the face $z\in\Dm(G)$ by 
\begin{align}
X(c)\ &:=\ |\cS(v^\bullet(c))-\cS(v^\circ(c))|^{\frac{1}{2}}\cdot\Re[\overline{\eta}_c \cdot F(z)] \notag\\
 &=\ \Re[\overline{\varsigma}\cX(c)\cdot F(z)]\ =\ \overline{\varsigma}\cX(c)\cdot\mathrm{Pr}[F(z),\eta_c\R]
\label{eq:X-from-F}
\end{align}
satisfies the propagation equation \eqref{eq:3-terms} around $z$. Conversely for $X:\Upsilon^\times(G)\to\R$ a unique real valued solution to \eqref{eq:3-terms}, there exists a s-holomorphic function~$F$ such that~\eqref{eq:X-from-F} is fulfilled.
\end{proposition}
When $F$ and $X$ are linked by \eqref{eq:X-from-F}, one can reconstruct the value of $F$ at $z\in \diamondsuit $ from the values of $X$ at any pair of corners $c_{pq}(z)\in\Upsilon^\times(G)$, e.g. \cite[Corollary 2.6]{Che20}
\begin{equation} \label{eq:F-from-X}
F(z)\ =\ -i\varsigma\cdot\frac{\overline{\cX(c_{01}(z))}\,X(c_{10}(z))-\overline{\cX(c_{10}(z))}\,X(c_{01}(z))} {\Im[\,\overline{\cX(c_{01}(z))}\,\cX(c_{10}(z))}.
\end{equation}

\medskip
To study the regularity of s-holomorphic functions as well as the local behavior of their scaling limit, one uses their 'integral' while one uses generalization of the 'integral' of the imaginary part of their square introduced by Smirnov in \cite{Smirnov_Ising} to derive their boundary behavior. The former is heavily studied in \cite[Proposition 6.15]{CLR1} and will be useful deriving local regularity theory for discrete functions as well as the local equation satisfied by subsequential limits in continuum, while the latter was introduced by Smirnov in \cite{Smirnov_Ising} on the critical square grid and has since been exploited in several different contexts to identify the scaling limit of fermionic observables. We start with the integral  $I_{\mathbb{C}}$ of an s-holomorphic function. In \cite[Section 2.5 of]{Che20}, s-holomorphic functions are described as gradients of harmonic functions on the associated S-graphs, specializing in the Ising context the technology developed in \cite[Section~4.2]{CLR1} for t-holomorphic functions. Given an s-holomorphic function $F$ on $\diamondsuit(G)$, one can define in the continuum  plane (up to a global additive constant) \cite[Section 2.3]{Che20}
\begin{equation}\label{eq:def-I_C}
I_{\mathbb{C}}[F]:= \int \big( \overline{\varsigma}Fd\cS + \varsigma \overline{F} d\cQ \big)
\end{equation}
Let $v_{1,2}^{\bullet}, v_{1,2}^{\circ}$ be vertices of the quad $z\in \diamondsuit(G)$. Then one has for $\star \in \{ \bullet, \circ \} $
\begin{equation}
I_{\mathbb{C}}[F](v_{2}^{\star}) - I_{\mathbb{C}}[F](v_{1}^{\star}) = \overline{\varsigma} F(z) [ \cS(v_{2}^{\star}) - \cS(v_{1}^{\star})] + \varsigma \overline{F(z)}[ \cQ(v_{2}^{\star}) - \cQ(v_{1}^{\star})].
\end{equation}

As for the imaginary part of the primitive of the square $H$, we start by recalling first its combinatorial definition (i.e. which doesn't require any particular embedding into the plane) for spinors on $\Upsilon^\times(G)$ satisfying \eqref{eq:3-terms}. Afterwards, we precise its analytic interpretation in the context of s-embeddings. This definition, which represents a generalisation of the original work of Smirnov, can be found in the following form in \cite[Definition 2.8]{Che20}.
\begin{definition}
\label{def:HX-def} Given $X$ a spinor on $\Upsilon^\times(G)$ satisfying \eqref{eq:3-terms}, one defines the function $H_X$ up to a global additive constant on $\Lambda(G)\cup\Dm(G)$ by setting
\begin{equation}
\label{eq:HX-def}
\begin{array}{rcll}
H_X(v^\bullet_p(z))-H_X(z)&:=&X(c_{p0}(z))X(c_{p1}(z))\cos\theta_z, & p=0,1,\\[2pt]
H_X(v^\circ_q(z))-H_X(z)&:=&-X(c_{0q}(z))X(c_{1q}(z))\sin\theta_z,& q=0,1,\\[2pt]
H_X(v^\bullet_p(z))-H_X(v^\circ_q(z))&:=&(X(c_{pq}(z)))^2,
\end{array}
\end{equation}
similarly to~\eqref{eq:cS-def} and~\eqref{eq:cS(z)-def}.
\end{definition}
The consistency of the above definition follows once again from the propagation equation \eqref{eq:3-terms}. Passing to an s-embedding $\cS$ of $(G,x)$, one can use the correspondence between $X$ and $F$ recalled in Proposition \ref{prop:shol=3term} to interpret $H_{X}$ via the s-holomorphic function $F$ associated to $X$. More precisely, it is possible to define \cite[Equation (2.17)]{Che20}
\begin{equation}
\label{eq:HF-def}
H_F:=\int\Re(\overline{\varsigma}^2F^2d\cS+|F|^2d\cQ)=\int (\Im(F^2d\cS)+\Re(|F|^2d\cQ)),
\end{equation}
on $\Lambda(G)\cup\Dm(G)$. The function $H_F$ extends linearly to a piece-wise affine function on each faces of the t-embedding $\cT=\cS$ (but not on each face of $\diamondsuit(G)$ as each face of $\cT$ has its own origami square root $d\cQ$), at least if one stays in the bulk of $G$ (see \cite[Proposition 3.10]{CLR1}). The next lemma links the definitions \eqref{eq:HX-def} and \eqref{eq:HF-def}, proving that they are in fact the \emph{same} function.

\begin{lemma}{\cite[Lemma 2.9]{Che20}} Let $F$ be defined $\Dm(G)$ and $X$ be defined on $\Upsilon^\times(G)$ related by the identity~\eqref{eq:X-from-F}. Then, the functions~$H_F$ and~$H_X$ coincide up to a global additive constant.
\end{lemma}
If $\cS $ is an isoradial grid, the origami map $\cQ$ is constant on both $G^\bullet$ and $G^\circ$ (as all the edges of any quad $\cS^{\diamond}(z) $ are all of the same length), thus $H_F$ is the primitive of $\Im[F^2d\cS]$, recovering the  original definition given in \cite[Section~3.3]{ChSmi2}. We now recall the comparison principle for functions $H_F=H_X$ associated with s-holomorphic functions. This statement is due to Park and can be found in \cite[Proposition 2.11]{Che20}. In particular, when one of the observables in the following proposition is identically $0$, this proposition becomes a maximum principle. 

\begin{proposition} \label{prop:HF-comparison} Let spinors $X,Y:\Upsilon^\times(G)\to\R$ both satisfy the propagation equation \eqref{eq:3-terms} and the associated functions $H_X,H_Y:\Lambda(G)\cup\Dm(G)\to\R$ be defined via \eqref{eq:HX-def}. Then, the difference $H_X-H_Y$ cannot have an extremum at an interior vertex of its domain of definition.
\end{proposition}
In particular if $H_X$ is bounded at the boundary of a domain (which is trivially the case for the observables defined in Section \ref{sec:proof}), then $H_X=H_F$ is bounded everywhere in the domain.

\subsection{Regularity theory for s-holomorphic functions}\label{sub:regularity}
In this short subsection we recall in a concise way the regularity theory of s-holomorphic functions, which is developed in \cite[Section 2.6]{Che20}, adapting to the Ising context results coming from \cite[Section~6]{CLR1}. Our proof of crossing estimates is based on a contradiction between the discrete behavior and its limiting counterpart, and thus we need to explain how to extract sub-sequential limits from discrete observables. We start by recalling that regularity theory of s-holomorphic functions requires adding one (mild) geometrical constraints on potential local degeneracies of the embedding, following \cite[Assumption 1.2]{CLR1}. In that case, an s-holomorphic function $F$ satisfies a standard Harnack-type estimate that controls $|F|^2 $ via the gradient of the function $H_F$, except in some pathological scenario where an exponential blow up in $\delta^{-1} $ occurs. 

\begin{theorem}{\cite[Theorem 2.18]{Che20}} \label{thm:F-via-HF} For each~$\kappa<1$ there exist constants~$\gamma_0=\gamma_0(\kappa)>0$ and \mbox{$C_0=C_0(\kappa)>0$} such that the following alternative holds. Let~$F$ be a s-holomorphic function defined in a ball of radius~$r$ drawn over an s-embedding~$\cS$ satisfying the assumption~$\textup{Lip}(\kappa,\delta)$. Then,
\[
\begin{array}{rcl}
\text{either}\ \max_{\{z:\cS(z)\in B(u,\frac{1}{2}r)\}}|F|^2&\le& C_0r^{-1}\cdot\osc_{\{v:\cS(v)\in B(u,r)\}}H_F\\[4pt]
\text{or}\ \max_{\{z:\cS(z)\in B(u,\frac{3}{4}r)\}}|F|^2&\ge& \exp(\gamma_0r\delta^{-1})\cdot C_0r^{-1}\osc_{\{v:\cS(v)\in B(u,r)\}}H_F
\end{array}
\]
provided that~$r\ge\cst\cdot\delta$ for a constant depending only on~$\kappa$.
\end{theorem}

\begin{rem}\label{rem:regularity}
We will use in our specific case Theorem \ref{thm:F-via-HF} to deduce that some properly chosen s-holomorphic functions $F^\delta$ form a precompact family in the topology of the uniform convergence on compacts as $\delta\to 0$, at least in a suitable open subset of the special domain constructed in \ref{sec:geometry}. Indeed, those functions are uniformly bounded and $\beta$-H\"older (see \cite[Theorem 2.18]{Che20}) on scales above $\cst(\kappa)\cdot \delta$. In that case, instead of using a reasoning similar to \cite[Corollary 2.20]{Che20}, we will use the existence of fat-enough boundary tangential quadrilaterals to rule out the second alternative of Theorem \ref{thm:F-via-HF}, allowing to prove boundedness and then compactness of the observables of interest.
\end{rem}

\subsection{Subsequential limits of s-holomorphic functions} \label{sub:shol-limits}
We discuss now the behavior of subsequential limits of s-holomorphic functions, under the general hypothesis $\LipKd\ $, following \cite[Section 2.7]{Che20}.
In what follows, we work with proper s-embeddings $\cS^\delta$ all satisfying the assumption \LipKd\ as $\delta\to 0$ \emph{for the same constant} $\kappa <1 $, such that their respective images cover a given ball $U=B(u,r)\subset\C$. As the functions $\cQ^{\delta}$ all are $\kappa$-Lipschitz above scale $O(\delta)$ and are defined up to an additive constant, there exist a sub-sequence $\delta_k\to 0$ and a $\kappa $-Lipschitz function $\vartheta:U\to\R$ such that uniformly on compacts of $U,$ 
\begin{equation}
\label{eq:Q-to-theta}
\cQ^{\delta_{k}}\circ (\cS^{\delta_{k}})^{-1}\to\vartheta.
\end{equation}
Assuming now we are in the first alternative of Theorem~\ref{thm:F-via-HF} and that \eqref{eq:Q-to-theta} holds, consider $f:U\to\C$ a subsequential limit of s-holomorphic functions~$F^\delta$ on~$\cS^\delta$. Then following \cite[Proposition 2.21]{Che20} and setting $\varsigma=e^{i\frac{\pi}{4}}$ as in \eqref{eq:def-eta}, the differential form $\tfrac{1}{2}(\overline{\varsigma}fdz+\varsigma\overline{f}d\vartheta)$ is closed. This comes as a natural counterpart of the consistency in the definition of the primitive $I_{\mathbb{C}}$ in \eqref{eq:def-I_C} (as contour integrals of discrete functions vanish before passing to the limit).
With a consistent choice of additive constants, the associated functions $H_{F^\delta}$ also converge uniformly on compact subsets of $U$ to $h:=\int(\Im(f^2dz)+|f|^2d\vartheta)$.

The previous condition on closeness of $\tfrac{1}{2}(\overline{\varsigma}fdz+\varsigma\overline{f}d\vartheta)$ is not  easily tractable and hard to interpret in terms of the local relation satisfied by $f$. In \cite[Section 2.7]{Che20}, Chelkak provides a nicer description when passing to the conformal parametrization of the an appropriate surface in the Minkowski space $\R^{2,1}$, equipped with the inner product of signature $(2,1)$.

On first recalls that the function $\vartheta$ is $\kappa$-Lipschitz, thus differentiable almost everywhere. One can then consider (as in \cite[Equation (2.26)]{Che20}) an orientation-preserving \emph{conformal parametrization} of the space-like surface $(z,\vartheta(z))_{z\in U}$, equipped with a positive metric coming from the ambient Minkowski space
\begin{equation}
\label{eq:zeta-param}
 \mathbb{D} \ni\zeta\ \mapsto\ (z,\vartheta)\in U\times\R\  \subset \C\times\R \ \cong \R^{2+1}
\end{equation}
As noted in \cite[below Equation (2.26)]{Che20}, in the case where $\vartheta $ is a smooth function, the angles (measured in $\R^{2,1}$) of infinitesimal increments are preserved by the mapping \eqref{eq:zeta-param} if and only if \emph{everywhere} in $\mathbb{D}$, one has \cite[Equation (2.27)]{Che20}
\begin{equation}
\label{eq:conf-param}
z_\zeta\overline{z}_\zeta\ =\ (\vartheta_\zeta)^2\quad \text{and}\quad |z_\zeta|>|\vartheta_\zeta|\ge |\overline{z}_\zeta|\,,
\end{equation}
where $z_\zeta:=\partial z/\partial\zeta$ (similarly, $\overline{z}_\zeta$ and $\vartheta_\zeta$) stands for the usual Wirtinger derivatives. 

Without assuming any smoothness assumption on $\vartheta$, one can note that conformal parametrization of \eqref{eq:conf-param} can be equivalently rewritten as a \emph{quasi-conformal} map $z\mapsto \zeta(z)$ (solution to a Beltrami equation):
\begin{equation}\label{eq:Beltrami-equation}
\zeta_{\bar{z}}=\mu(z)\zeta_{z},
\end{equation}
(or in an equivalent way to $z_{\zeta}= - \overline{\mu(\zeta)} z_{\zeta}$), with the Beltrami coefficient $\mu$ given by the equation
\begin{equation}\label{eq:Beltrami-coefficient}
\frac{\bar{\mu}}{1+|\mu|^2}= - \frac{\vartheta_{z}^2}{1-2|\vartheta_z|^2}.
\end{equation}
This equivalence can be seen by plugging the identity $\vartheta_{\zeta}=\vartheta_{z}z_{\zeta} + \vartheta_{\bar{z}} \bar{z}_{\zeta} $ (using the fact that $\vartheta$ is real valued) into the condition \eqref{eq:conf-param}, which rewrites as $-\overline{\mu}=(\vartheta_{z} - \overline{\mu}\vartheta_{\bar{z}})^2 $. Since $\vartheta_{\bar{z}}=\overline{\vartheta_z} $, this justifies the equality \eqref{eq:Beltrami-coefficient}. Moreover the function $z\mapsto \vartheta(z)$ is $\kappa$-Lipschitz for some $\kappa <1$, which proves that $|\vartheta_{z}| \leq \frac{\kappa}{2} $ and ensures that $|\mu|\leq \text{cst}(\kappa) <1 $ (see \cite[Equation (2.30)]{Che20}). One can now proceed as follows
\begin{itemize}
\item Compute the Beltrami coefficient $\mu \in L^{\infty}$ from the equation \eqref{eq:Beltrami-coefficient} \emph{almost everywhere}, as $\vartheta$ is differentiable almost everywhere for the Lebesgue measure.
\item Use the Ahlfors-Bers's measurable Riemann mapping theorem \cite[Chapter 5]{AIM} to solve the Beltrami equation \eqref{eq:Beltrami-equation}, constructing a quasi-conformal uniformization $\zeta:U\mapsto \mathbb{D} $ such that \eqref{eq:conf-param} holds almost everywhere in $\mathbb{D}$.
\end{itemize}

\bigskip
In the $\zeta\in \mathbb{D}$ parametrization, one can make a convenient change of variables as in \cite[Equation (2.28)]{Che20}, by defining the functions
\begin{equation}
\label{eq:f-to-phi-change}
\phi (\zeta) :=\ \overline{\varsigma}f(z(\zeta))\cdot (z_\zeta)^{1/2}+\varsigma\overline{f(z(\zeta))}\cdot (\overline{z}_\zeta)^{1/2}\,
\end{equation}
Under this change of variables, $I_{\mathbb{C}}[f]:= \int \overline{\varsigma}f(z)dz + \varsigma \overline{f(z)}d\vartheta $ reads as
\begin{equation}
\label{eq:Ic-to-phi-change}
g(\zeta)= \displaystyle \int \overline{\varsigma} \phi(\zeta) \cdot z_{\zeta}^{\frac{1}{2}} d\zeta+ \varsigma \overline{\phi(\zeta)}\cdot (z_{\bar{\zeta}})^{\frac{1}{2}}d\bar{\zeta}.
\end{equation}
Computing the Wirtinger derivatives in the $\zeta $ variable and using the almost everywhere relation \eqref{eq:conf-param}, one sees directly that $g$ satisfies a \emph{conjugate} Beltrami equation
\begin{equation}\label{eq:g_beltrami}
g_{\overline{\zeta}} = i \overline{\nu}\cdot \overline{g_{\zeta}} ~~~~~~~~ \text{ with } ~~~~~~~~ \nu:= - \frac{(\overline{z}_{\zeta})^{\frac{1}{2}}}{(z_{\zeta})^{\frac{1}{2}}}=-\frac{\vartheta_{\zeta}}{z_{\zeta}} 
\end{equation}
where the Beltrami coefficient $\nu$ is bounded away from $1$ (see \cite[Equation (2.30)]{Che20}) this bound depends again on $\kappa $). Indeed, one can see from the identity $\vartheta_{\zeta}=\vartheta_{z}z_{\zeta} + \vartheta_{\bar{z}} \overline{z}_{\zeta} $ that $|\nu|<2|\vartheta_{z}|\leq \kappa <1 $.  

The focus on the function $g$ is twofold: beyond being also a solution to the conjugate Beltrami equation \eqref{eq:g_beltrami}, it is constructed as the primitive of  a continuous differential form and thus inherits some additional a priori regularity. Provided one knows that $f$ is locally bounded, one can deduce directly that
\begin{itemize}
\item The function $g$ has bounded distortion (see e.g. \cite[Equation (2.27)]{AIM} for the precise definition), smaller than $\frac{1+\kappa}{1-\kappa}$.
\item The function $g$ also belongs to $L^{1,2}_{loc} $. First recall the formula $g_{\zeta} = \overline{\varsigma}f(z(\zeta))z_{\zeta} + \varsigma \bar{f}(z(\zeta))\vartheta_{\zeta} $ and that $f(z(\zeta))$ is locally bounded. The Jacobian of $z$ satisfies $\text{Jac}(z):= | z_{\zeta}|^2 - | z_{\bar{\zeta}}|^2 \geq (1-\kappa^2) | z_{\zeta}|^2 $ as equation \eqref{eq:conf-param} ensures that $|z_{\bar{\zeta}}|\leq \kappa |z_{\zeta}| $ in the $\zeta$ parametrization.  It is then possible to use the area principle. Namely, $\int_{\Omega}  | z_{\zeta}|^2 d^{2}\zeta \leq (1-\kappa^2)^{-1} \int_{\Omega}  \text{Jac}(z)d^2z= (1-\kappa^2)^{-1} \text{Area}(\Omega)<\infty $. This ensures then that  $g$ belongs to $L^{1,2}_{loc} $.
\end{itemize}
Combining the two previous observations, one can apply the Stoïlov factorization stated in \cite[Corollary $5.3.3$]{AIM} which allows to write the factorization $g= \underline{g} \circ p $ with $p:\Omega \rightarrow \Omega$ a $\beta(\kappa)$-Hölder homeomorphism and $\underline{g}$ a holomorphic function. In particular, $g$ cannot be constant in an open set except if it is constant everywhere in the domain.

\section{Construction of a special discrete domain} \label{sec:geometry}
In this section, we explain how one can start from a given piece of an s-embedding satisfying the assumptions \LipKd\, and \OneExpFat\, and construct a discrete simply connected domain containing two pieces of the square lattice, which will allow a rather simplified proof of Theorem \ref{thm:positive-magnetization}. Once Theorem \ref{thm:positive-magnetization} is proven in those special domains, it is  rather straightforward to transfer it to usual discretization of rectangles by standard monotonicity arguments for the Ising model. In order to keep the presentation as simple as possible,  we start this section with a subsection containing reminders on rather simple geometrical features of tangential quadrilaterals. As one will see in the following lines, extending a finite piece of a tilling formed by tangential quadrilaterals and pasting it to a piece of the square lattice, while remaining in the tangential quad setup (i.e. requiring that all faces $\Lambda(G)$ are tangential quadrilaterals) is not obvious at first sight. Still, one notes that in the case where all boundary vertices are aligned along a straight line, there exist rather straightforward extension strategies, either by using symmetry arguments or by pasting layers of kites and squares.

\subsection{Features of tangential quadrilaterals}\label{sub:geometrical_features}
In the following list of claims, we recall simple geometrical features on tangential quadrilaterals. The claims are illustrated by Figure  \ref{fig:geometrical-claims}. Working on an s-embeddings that satisfy \LipKd\ for some fixed $\kappa <1 $, all the edges of any quad of $\cS^{\delta}(\diamondsuit(G)) $ are of length smaller than $\delta $. Indeed, along the edges of a quad $z$, the origami map $\cQ^{\delta} $ is real a linear function with rate of growth $\pm 1$. To lighten the notations, we identify each vertex of $\Lambda(G) $ with its image in $\cS^\delta(\Lambda(G) )$ in the embedding.

Let $z=(v_{0}^{\circ}v_{0}^{\bullet}v_{1}^{\circ}v_{1}^{\bullet})$ be a tangential quadrilateral with an inscribed circle of radius $r_z$ and centered at $\hat{z}$. In what follows, we do not ask for the convexity of $z$, but always assume its edges are all of length smaller than $\delta$. Then one has:
\begin{enumerate}[label=(\Alph*)]
\item The four bisectors of the angles of the quadrilateral $z$ intersect at $\hat{z}$.
\item The area of the tangential quadrilateral is the product of the radius of its inscribed circle $r_z$ by the half-perimeter of $z$. In the following paragraphs, in several occasions, we will bound from below the radius of a tangential quad. This is done by bounding from below the area of $z$ and from above its perimeter.
\item Let $\phi_{v,z} $ be one of the half-angles of $z$ at $v$ (i.e. the angle formed by an edge containing $v$ and the bisector that links $v$ to $\hat{z}$). Then, provided the angle $\phi_{v,z} $ is smaller than $\frac{\pi}{4}$, there exists a universal constant $C$ such that $\phi_{v,z} \geq C \tan \phi_{v,z} \geq C \frac{r_z}{\delta}$. The left inequality is true when $\phi_{v,z} $ goes to $0$ while the right inequality is a direct computation on the straight triangle formed by $v,\hat{z} $ and the orthogonal projection of $\hat{z}$ on one of the edges containing $v$. 

\item Fix an edge $e=[v^{\circ}v^{\bullet}]$ of a tangential quadrilateral $z$ attached to vertices $v^{\circ}$ and $v^{\bullet}$ whose respective angles are $\phi_{v^{\circ},z} $ and $\phi_{v^{\bullet},z} $. The length of the edge $e$ equals $r_{z}(\cot(\phi_{v^{\circ},z})+\cot(\phi_{v^{\bullet},z})) \geq r_{z}\sin(\phi_{v^{\circ},z}+\phi_{v^{\bullet},z}) $.

\item Fix three vertices $v_{0,1}^{\circ}$ and $v_{0}^{\bullet} $. The set $\mathcal{C} $ of points $v^{\bullet}$ in the plane such that $(v_{0}^{\circ}v_{0}^{\bullet}v_{1}^{\circ}v^{\bullet})$ is a tangential quadrilateral is an \emph{hyperbola} (potentially degenerated to a straight line) passing through $v_{0}^{\bullet} $ and $v_{1}^{\bullet} $. Indeed, denoting $(x,y)$ the coordinates of $v^{\bullet}$ and writing the equality $|v^{\bullet}-v_{0}^{\circ}| -|v^{\bullet}-v_{1}^{\circ}|= |v_{0}^{\bullet}-v_{0}^{\circ}| -|v_{0}^{\bullet}-v_{1}^{\circ}| $, one recovers the algebraic equation of an hyperbola. The conical $\mathcal{C}$ is clearly unbounded and admits the bisector of the angle $(v_{0}^{\circ}v_{0}^{\bullet}v_{1}^{\circ}) $ as an asymptote when $v^{\bullet}$ goes to infinity.

\item Let $T=(JKL)$ be a triangle (whose vertices are labeled in the counter-clockwise order) and consider $\frak C$ the circle of radius $r$ which is tangential to its three sides. Consider e.g. $M$ the point at the intersection of $\frak C$ and the segment $[JL]$. Then one can view the quadrilateral $(JKLM)$(still labeling vertices in the counter-clockwise order) as a tangential quadrilateral, whose tangential circle is $\frak C$.

\end{enumerate}

We say that a tangential quadrilateral $z$ is $\beta$-fat is the associated tangential circle has a radius $r_z \geq \beta $.

\begin{figure}
\begin{minipage}{0.325\textwidth}
\includegraphics[clip, width=1.2\textwidth]{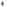}
\end{minipage}\hskip 0.01\textwidth \begin{minipage}{0.33\textwidth}
\includegraphics[clip, width=1.2\textwidth]{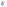}
\end{minipage}
\begin{minipage}{0.325\textwidth} 
\includegraphics[clip, width=\textwidth]{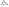}
\end{minipage}
\caption{(Left) Notations for the tangential quadrilateral with center $\hat{z}$. The four bisectors intersect at the center $\hat{z}$ of the tangential quadrilateral. (Middle) Hyperbola $\mathcal{C}$ drawn in blue denoting in claim (E) the set of points such that $(v_{0}^{\circ}v_{0}^{\bullet}v_{1}^{\circ}v^{\bullet})$ is a tangential quadrilateral.  (Right) Transformation described in (F) of a triangle into a tangential quadrilateral by adding as a vertex one tangency point of $\frak C$ with one of the sides of the triangle.}
\label{fig:geometrical-claims}
\end{figure}

\subsection{Horizontal alignment of the lattice}\label{sub:construction_extension}

We now present the welding procedure of a piece of a given s-embedding $\cS^{\delta}$ to a piece of the square lattice. The local constrain of all quadrilaterals being tangential is rather unpleasant to handle, thus we will first straighten the boundary by aligning along the same line all its vertices, which provides a simpler situation to handle. The first step is to \emph{slice} horizontally a piece of an s-embedding to extend it afterwards by periodic layers. We present now a very concrete method to replace the intersection of a tangential quadrilateral and a closed half-plane by a similar picture but this time with bottom boundary vertices all aligned along the border of that half-plane. We give the construction for vertical half-planes, but it is rather straightforward to adapt for general half-planes.

\begin{defi}
Let $z\in \diamondsuit(G)$ a tangential quadrilateral, $\overline{z} $ the topological closure of $z$ and $y$ a horizontal level  intersecting $z$ away from its vertices. We say that $Z \subseteq \overline{z}$ is a \emph{horizontal alignment} from above of $z$ at level $y$ if
\begin{itemize}
\item $Z$ is the union of at most $3$ tangential quadrilaterals (see Figure \ref{fig:alignement}).
\item $Z=\overline{z} \cap (iy + \overline{\mathbb{H}}) $ i.e. $Z$ is the intersection of $z$ and the close half plane above level $y$.
\end{itemize}
\end{defi}

One can define in a similar fashion horizontal alignment from below, which concerns the close half plane below level $y$. The next lemma ensures that it is possible to perform a horizontal alignment (from above or from below) of a given tangential quadrilateral.

\begin{lem}\label{lem:horizontal_allignement}
Let $z$ a tangential quadrilateral and $y$ a vertical level (represented by the line $\mathbb{R}+iy$ in $\mathbb{C}$) that intersects $z$ at a level that doesn't contain its highest vertex. Then there exist an horizontal alignment from above of $z$ at level $y$.
\end{lem}

\begin{proof}
The proof is made by an explicit construction using the facts recalled above. We make a dichotomy depending on the number of vertices of $z$ below level $y$. We encourage the reader to follow the proof with the pictures of Figure \ref{fig:alignement}.

\underline{\textbf{(1) There is only one vertex of $z$ below level $y$}}

Up to swapping colors, one can assume that the vertex below level $y$ is $v^{\bullet}_0$. Let $\mathcal{C}$ the be hyperbola (claim (E)) of points $ v^{\bullet}$ such that $(v_{0}^{\circ}v_{1}^{\bullet}v_{1}^{\circ}v^{\bullet})$ is a tangential quadrilateral. This hyperbola is a continuous curve \emph{inside} $z$ containing the points $v^{\bullet}_{0}$ and $v^{\bullet}_{1} $. A continuity argument ensures the existence of a point $\tilde{v}^{\bullet}$ with $\Im[\tilde{v}^\bullet]=y$ such that $(v_{0}^{\circ}v_{1}^{\bullet}v_{1}^{\circ}\tilde{v}^\bullet)$ is a tangential quadrilateral. As a consequence (e.g. checking the alternate sum of edge-lenghts), the quadrilateral $(v_{0}^{\circ}v_{0}^{\bullet}v_{1}^{\circ}\tilde{v}^\bullet)$ is also  tangential. Set now the points $\tilde{v}^{\bullet}_{0} $ and $\tilde{v}^{\bullet}_{1} $ to be respectively the intersections of the segments $[v^{\circ}_{0}v^{\bullet}_{0}] $ and $[v^{\circ}_{0}v^{\bullet}_{1}] $ with the level $y$. Using the claim (F), one can transform the triangles $(\tilde{v}^{\bullet} v^{\circ}_{0}\tilde{v}^{\bullet}_{0}) $ and $(\tilde{v}^{\bullet} v^{\circ}_{1}\tilde{v}^{\bullet}_{1}) $ as tangential quadrilaterals by adding repectively the vertices $\tilde{v}^{\circ}_{0,y} $ and $\tilde{v}^{\circ}_{1,y} $ to the segments $[\tilde{v}^{\bullet}_{0} \tilde{v}^{\bullet}] $ and $[\tilde{v}^{\bullet}_{1} \tilde{v}^{\bullet}] $.

\underline{\textbf{(2) There are two vertices of $z$ below level $y$}}

We make a dichotomy of the two subcases that appear here.

(a) The two vertices that lie below level $y$, $v_{0}^{\bullet}$ and $v_{0}^{\circ}$, are of different colors. Consider first $\mathcal{C}_{1}$ the hyperbola of points $ v^{\bullet}$ such that $(v_{0}^{\circ}v_{0}^{\bullet}v_{1}^{\circ}v^{\bullet})$ is a tangential quadrilateral linking continuously $v^{\bullet}_{0}$ to $v^{\bullet}_{1} $  inside $z$. There exists once again of a point $\tilde{v}^{\bullet}$, with $\Im[\tilde{v}^\bullet]=y$, such that $(v_{0}^{\circ}v_{0}^{\bullet}v_{1}^{\circ}\tilde{v}^\bullet)$ is a tangential quadrilateral. As previously, the quadrilateral $(v_{0}^{\circ}v_{1}^{\bullet}v_{1}^{\circ}\tilde{v}^{\bullet})$ is tangential. One then reapply the same strategy, this time to the tangential quadrilateral $(v_{0}^{\circ}v_{1}^{\bullet}v_{1}^{\circ}\tilde{v}^\bullet)$ to construct inside $(v_{0}^{\circ}v_{0}^{\bullet}v_{1}^{\circ}\tilde{v}^\bullet)$ a tangential quadrilateral $(v_{1}^{\circ}v_{1}^{\bullet}\tilde{v}^{\circ}\tilde{v}^\bullet)$ with $\Im[\tilde{v}^\circ]=y$. Set this time $\tilde{v}^{\bullet}_{0} $ and $\tilde{v}^{\circ}_{0} $ the respective intersections of the segments $[v^{\circ}_{1}v^{\bullet}_{0}] $ and $[v^{\bullet}_{1}v^{\circ}_{0}] $ with the line $y$. Using claim (F), one can construct two tangential quadrilaterals out of the triangles $(\tilde{v}^{\bullet}_{0}v^{\circ}_{1}v^{\bullet}_{0}) $ and $(\tilde{v}^{\circ}_{0,y} v^{\bullet}_{1}v^{\circ}_{0})$ by adding respectively one white vertex to the segment $[\tilde{v}^{\bullet}_{0}v^{\bullet}] $ and one black vertex to the segment $[\tilde{v}^{\circ}_{0}\tilde{v}^{\circ}] $, both located on the axis $y$.

(b) The two vertices $v_{0}^{\bullet}$ and $v_{1}^{\bullet}$ below level $y$ are of the same color (that we assume to be black here up to swapping colors). In that case, $(v_{0}^{\circ}v_{0}^{\bullet}v_{1}^{\circ}v^{\bullet})$ is non convex and $\overline{z} \cap (iy + \overline{\mathbb{H}})$ is formed by two triangles containing respectively $v^{\circ}_{0} $ and $v^{\circ}_{1}$. Denote by $\tilde{v}^{\bullet}_{00},\tilde{v}^{\bullet}_{01},\tilde{v}^{\bullet}_{10},\tilde{v}^{\bullet}_{11}$ the intersections of respectively $[v^{\bullet}_{0}v^{\circ}_{0}],[v^{\bullet}_{0}v^{\circ}_{1}],[v^{\bullet}_{1}v^{\circ}_{0}],[v^{\bullet}_{1}v^{\circ}_{1}] $ and this axis $y$. Then the triangles $(\tilde{v}^{\bullet}_{10}\tilde{v}^{\bullet}_{10} v^{\circ}_{0}) $ and $(\tilde{v}^{\bullet}_{01}\tilde{v}^{\bullet}_{11} v^{\circ}_{1}) $ can be viewed as tangential quadrilaterals by adding two white vertices to the segments $[\tilde{v}^{\bullet}_{10}\tilde{v}^{\bullet}_{00}] $ and $[\tilde{v}^{\bullet}_{01}\tilde{v}^{\bullet}_{11}]$ using claim (F).

\underline{\textbf{(3) There are three vertices of $z$ below level $y$}}

Assume e.g. that $v^{\bullet}_{1} $ lies above level $y$. The consider $\tilde{v}^{\circ}_{0} $ and $\tilde{v}^{\circ}_{1} $ the intersections of $y$ and the segments $[v^{\bullet}_{1}v^{\circ}_{0}] $ and $[v^{\bullet}_{1}v^{\circ}_{1}] $. Then one can view the triangle $(v^{\bullet}_{1}\tilde{v}^{\circ}_{0}\tilde{v}^{\circ}_{1}) $ as a tangential quadrilateral by adding a black vertex to the segment $[\tilde{v}^{\circ}_{0}\tilde{v}^{\circ}_{1}]$, using again claim (F).

\begin{figure}
\begin{minipage}{0.40\textwidth}
\includegraphics[clip, width=1\textwidth]{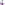}
\end{minipage} \begin{minipage}{0.40\textwidth}
\includegraphics[clip, width=1\textwidth]{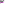}
\end{minipage} \begin{minipage}{0.40\textwidth}
\includegraphics[clip, width=1\textwidth]{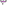}
\end{minipage} \begin{minipage}{0.40\textwidth}
\includegraphics[clip, width=1\textwidth]{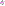}
\end{minipage} 
\caption{Top: (Left) Case 1 (Right) Case 2a. Bottom: (Left) Case 2b (Right) Case 3}\label{fig:alignement}
\end{figure}

\end{proof}

In order to apply regularity theory for s-holomorphic functions and extract subsequential limits of discrete observables, we will see that it is enough to construct a domain with a boundary formed by fat enough quads. As we will see below, if one doesn't choose carefully the level at which it performs an horizontal alignment, it might create new tangential quads whose radius is too small and forbid the use of regularity theory for s-holomorphic functions developed in Section \ref{sub:regularity}. We explain now that, if one starts with a quad in $\cS^{\delta} $ whose radius not so small, there is only a small share of vertical levels intersecting for which horizontal alignments will produce tangential quads with a much smaller radius.The next definition precises this idea.

\begin{defi}
Let $y$ be a horizontal a level that intersects the tangential quad $z$. We say that $y$ is a $\beta$-\emph{bad level} for $z$ if one of the tangential quadrilaterals constructed with the algorithm of the horizontal alignment from above of Lemma \ref{lem:horizontal_allignement} has a radius $r_z$ smaller than $\beta$. The complement of $\beta$-bad levels are called $\beta$-\emph{good levels}.
\end{defi}

The next proposition upper bounds the share of bad levels in a tangential quadrilateral $z$ whose radius is larger than $\exp(-\gamma \delta^{-1})$. Informally speaking, the proof shows that $\beta-$bad levels (for a small $\beta$ compared to $\exp(-\gamma \delta^{-1})$) are only those close to the horizontal lines containing vertices of $z$.

\begin{prop}\label{prop:bad-levels}
Let $z$ be a tangential quad with a radius $r_z \geq \exp(-\gamma\delta^{-1}) $ and whose edges are all of length smaller than $\delta$. Then provided $\delta$ is small enough, the (vertical) one dimensional Lebesgue measure of $\exp(-40\gamma \delta^{-1})$-bad levels intersecting $z$ is at most $4\exp(-4\gamma\delta^{-1})$.
\end{prop}

\begin{proof}
We work with $\delta$ chosen small enough. Using claim (C) on the  features of tangential quadrilateral, the angles $\phi_{v,z}$ are bounded from below by $\delta^{-1}\exp(-\gamma\delta^{-1}) \geq \exp(-\gamma \delta^{-1}) $, as $r_z \geq \exp(-\gamma\delta^{-1})$ and while edge-lengths of the boundary segments of $z$ are smaller than $\delta$. Fix an horizontal level $y$, at a vertical distance at least $\exp(-4\gamma\delta^{-1})$ from the (at most) $4$ axes containing the vertices of $z$ and perform an horizontal alignment at level $y$ following Lemma \ref{lem:horizontal_allignement}. We keep exactly the same notations and treat separately the four subcases presented there, depending on the number of vertices below level $y$. We still denote $\hat{z}$ the center of the tangential quadrilateral.

\underline{\textbf{(1) There is only one vertex of $z$ below the axis $y$}}

One claims that, provided $\delta $ is small enough, the area of the tangential quadrilaterals $(v^{\circ}_{0}v^{\bullet}_{1}v^{\circ}_{1}\tilde{v}^{\bullet})$ and $(v^{\circ}_{0}v^{\bullet}_{0}v^{\circ}_{1}\tilde{v}^{\bullet})$ can be bounded from below by the quantity \newline $\frac{1}{2}\exp(-\gamma\delta^{-1})\exp(-4\gamma\delta^{-1})\exp(-4\gamma\delta^{-1}) \geq \exp(-10\gamma\delta^{-1}) $. To do that, we are first lower bound the area the of tangential quadrilateral $(v^{\circ}_{0}v^{\bullet}_{0}v^{\circ}_{1}\tilde{v}^{\bullet})$, then lower bound one of its angles and derive the statement for $(v^{\circ}_{0}v^{\bullet}_{1}v^{\circ}_{1}\tilde{v}^{\bullet})$ and $(v^{\circ}_{0}v^{\bullet}_{0}v^{\circ}_{1}\tilde{v}^{\bullet})$. 

Consider e.g. the triangle (filled in blue in the left Figure \ref{fig:geometrical-claims}) denoted by \newline  $T_{(v^{\circ}_{0}v^{\bullet}_{0}v^{\circ}_{1}\tilde{v}^{\bullet})}(v^{\bullet}_{0},\exp(-4\gamma\delta^{-1}))= T(v^{\bullet}_{0},\exp(-4\gamma\delta^{-1})) \subseteq (v^{\circ}_{0}v^{\bullet}_{0}v^{\circ}_{1}\tilde{v}^{\bullet}) $, isoscele at $v^{\bullet}_{0} $ such that
\begin{itemize}
\item The symmetry axis of $T(v^{\bullet}_{0},\exp(-4\gamma\delta^{-1}))$ is the bisector of the angle $(v^{\circ}_{0} v_{0}^{\bullet} v_{1}^{\circ})$
\item Two sides of $T(v^{\bullet}_{0},\exp(-4\gamma\delta^{-1}))$ belong respectively to the segments $[v^{\circ}_{0} v_{0}^{\bullet}]$ and $[v^{\circ}_{1} v_{0}^{\bullet}]$.
\item The height of $T(v^{\bullet}_{0},\exp(-4\gamma\delta^{-1}))$ (along the bisector of the angle $(v^{\circ}_{0} v_{0}^{\bullet} v_{1}^{\circ})$) is $\exp(-4\gamma\delta^{-1})$.
\end{itemize}
Since the angle $\widehat{\phi_{v^{\bullet}_{0},z}}$ is bounded from below by $\exp(-\gamma\delta^{-1}) $, the area $A_{(v^{\circ}_{0}v^{\bullet}_{0}v^{\circ}_{1}\tilde{v}^{\bullet})}$ of $T(v^{\bullet}_{0},\exp(-4\gamma\delta^{-1}))$ and thus the area of $(v^{\circ}_{0}v^{\bullet}_{0}v^{\circ}_{1}\tilde{v}^{\bullet})$ is larger than $\exp(-10\gamma\delta^{-1})$. Moreover, the perimeter $\text{Per}_{(v^{\circ}_{0}v^{\bullet}_{0}v^{\circ}_{1}\tilde{v}^{\bullet})}$ of $(v^{\circ}_{0}v^{\bullet}_{0}v^{\circ}_{1}\tilde{v}^{\bullet})$ is at most $4\delta$ (it is a general fact that if a convex polygon lies inside another one, the perimeter of the outer one is larger than the interior one). 

\begin{enumerate}
\item $r_{(v^{\circ}_{0}v^{\bullet}_{0}v^{\circ}_{1}\tilde{v}^{\bullet})} = 2A_{(v^{\circ}_{0}v^{\bullet}_{0}v^{\circ}_{1}\tilde{v}^{\bullet})} \text{Per}_{(v^{\circ}_{0}v^{\bullet}_{0}v^{\circ}_{1}\tilde{v}^{\bullet})}^{-1} \geq \frac{1}{2} \exp(-10\gamma\delta^{-1}) \cdot \delta^{-1} \geq \exp(-10\gamma\delta^{-1})$ using claim (B) when $\delta$ is small enough.
\item The angles $\widehat{v^{\bullet}_{0} v^{\circ}_{0} \tilde{v}^{\bullet}} $ and $\widehat{v^{\bullet}_{0} v^{\circ}_{1} \tilde{v}^{\bullet}} $ are both bounded from below by the quantity $Cr_{(v^{\circ}_{0}v^{\bullet}_{0}v^{\circ}_{1}\tilde{v}^{\bullet})} \delta^{-1} \geq C\exp(-10\gamma\delta^{-1})\cdot \delta^{-1} \geq \exp(-10\gamma\delta^{-1}) $, where the absolute constant $C$ comes from claim (C) and $\delta$ is chosen small enough.
\item One can compute directly the area of the triangle $[v^{\bullet}_{0} v^{\circ}_{0} \tilde{v}^{\bullet}] $, which is exactly equal to $\frac{1}{2} \sin \widehat{v^{\bullet}_{0} v^{\circ}_{0} \tilde{v}^{\bullet}}~ |\tilde{v}^{\bullet}_{0}-v^{\circ}_{0}||\tilde{v}^{\bullet}-v^{\circ}_{0}|  \geq \frac{1}{4} \exp(-10\gamma \delta^{-1}) \cdot \exp(-4\gamma \delta^{-1}) \cdot \exp(-4\gamma \delta^{-1}) \geq \exp(-20\gamma \delta^{-1}) $ (as $\sin \widehat{v^{\bullet}_{0} v^{\circ}_{0} \tilde{v}^{\bullet}} \geq \frac{1}{2} \exp(-10\gamma \delta^{-1})$ and both distances $|\tilde{v}^{\bullet}_{0}-v^{\circ}_{0}|$ and $|\tilde{v}^{\bullet}_{0}-v^{\circ}_{0}||\tilde{v}^{\bullet}-v^{\circ}_{0}|$ are larger than $\exp(-4\gamma \delta^{-1})$). Since the perimeter of $[v^{\bullet}_{0} v^{\circ}_{0} \tilde{v}^{\bullet}] $ is smaller than $10\delta$, one gets the lower bound $ r_{v^{\bullet}_{0} v^{\circ}_{0}  \tilde{v}^{\bullet} \tilde{v}^{\circ}_{0,y}} \geq \exp(-20\gamma\delta^{-1})$, repeating exactly the area/perimeter argument given in step 1.
\end{enumerate}
The same result holds for the triangle  $(v^{\bullet}_{0} v^{\circ}_{1} \tilde{v}^{\bullet})$ viewed as a tangential quadrilateral.

\underline{\textbf{(2) There are two vertices of opposite color of $z$ below the axis $y$}}

Recall that the horizontal alignment is constructed by modifying along the appropriate hyperbola $(v^{\bullet}_{0}v^{\circ}_{1}v^{\bullet}_{1}v^{\circ}_{0}) $ into $(\tilde{v}^{\bullet}v^{\circ}_{1}v^{\bullet}_{1}v^{\circ}_{0}) $ and then modifying along the appropriate hyperbola $(\tilde{v}^{\bullet}v^{\circ}_{1}v^{\bullet}_{1}v^{\circ}_{0}) $ into $(\tilde{v}^{\bullet}v^{\circ}_{1}v^{\bullet}_{1}\tilde{v}^{\circ}) $.

Repeating the arguments of the case with one vertex below the axis $y$, one gets

\begin{itemize}
\item The angles $ \widehat{v^{\bullet}_0 v^{\circ}_{1}\tilde{v}^{\bullet}}$ and $ \widehat{\tilde{v}^{\circ}v^{\bullet}_{1}v^{\circ}_{0}} $ are larger than $\exp(-10\gamma\delta^{-1})$, as in (2). The radii of the tangential quadrilaterals associated to the triangles $(\tilde{v}^{\bullet}_{0}v^{\circ}_{1}\tilde{v}^{\bullet}) $ and  $(\tilde{v}^{\circ}_{0}v^{\bullet}_{1}\tilde{v}^{\circ}) $ are then larger than $\exp(-20\gamma\delta^{-1})$ as (3).

\item The tangential quadrilateral $(v^{\circ}_{0}\tilde{v}^{\bullet}v^{\circ}_{1}v^{\bullet}_{1}) $ has an area at least $\exp(-10\gamma \delta^{-1})$ which implies that both the angle $\widehat{\tilde{v}^{\bullet}v^{\circ}_{1}v^{\bullet}_{1}} $ and the radius $r_{v^{\circ}_{0}\tilde{v}^{\bullet}v^{\circ}_{1}v^{\bullet}_{1}} $ are larger than $\exp(-10\gamma\delta^{-1})$.
\item One can now consider the triangle $T_{(v^{\circ}_{0}\tilde{v}^{\bullet}v^{\circ}_{1}v^{\bullet}_{1})}(v^{\circ}_{1},\exp(-12\gamma\delta^{-1}))=$ \newline $T(v^{\circ}_{1},\exp(-12\gamma\delta^{-1}))$, isoscele at $v^{\circ}_{1}$ whose height along the bisector of $(v^{\circ}_{0}\tilde{v}^{\bullet}v^{\circ}_{1}v^{\bullet}_{1})$ is $\exp(-12\gamma\delta^{-1})$. This triangle is contained in $(v^{\circ}_{0}\tilde{v}^{\bullet}v^{\circ}_{1}v^{\bullet}_{1})$ and has an area area is $\frac{1}{2} \sin \widehat{\tilde{v}^{\bullet}v^{\circ}_{1}v^{\bullet}_{1}} \cdot \exp(-12\gamma\delta^{-1}) \cdot \exp(-12\gamma\delta^{-1}) \geq \exp(-40\gamma\delta^{-1})$. One can then conclude that $r_{v^{\circ}_{1}\tilde{v}^{\bullet}\tilde{v}^{\circ}_{1}v^{\bullet}_{1}} \geq \exp(-40\gamma\delta^{-1})$ as in (1).
\end{itemize}

\underline{\textbf{(3) Remaining cases}}
To handle the remaining cases (corresponding to 2b and 3 in Lemma \ref{lem:horizontal_allignement}) it is sufficient to note that in e.g. the case where three vertices of $z$ are below the axis $y$, the triangle $T_{(v^{\bullet}_{0}v^{\circ}_{0}v_{1}^{\bullet}v^{\circ}_{1})}(v^{\bullet}_{1},\exp(-4\gamma\delta^{-1}))=T(v^{\bullet}_{1},\exp(-4\gamma\delta^{-1}))$ is \emph{inside} the triangle $(\tilde{v}^{\circ}_{0} v^{\bullet}_{1} \tilde{v}^{\circ}_{0})$ has an area at least $\exp(-10\gamma\delta^{-1})$ as the angle of $z$ at $v^{\bullet}_{0}$ is bounded from below by $\exp(-\gamma\delta^{-1})$. In particular the tangential quadrilateral associated to the triangle $(\tilde{v}^{\circ}_{0} v^{\bullet}_{1} \tilde{v}^{\circ}_{0})$ has then a radius $r_z \geq \exp(-10\gamma\delta^{-1})$ as in (3).

All together, this proves that the one dimensional Lebesgue measure (in the vertical direction) of $\exp(-40\gamma\delta^{-1})$-bad levels intersecting $z$ is at most $4\exp(-4\gamma\delta^{-1})$. 

\end{proof}

Let us note that the same proof ensures that Proposition \ref{prop:bad-levels} also holds replaing the initial hypothesis $r_z \geq \exp( -\gamma\delta^{-1})$ by $r_z \geq \delta\exp( -\gamma\delta^{-1})$, provided $\gamma $ is fixed and $\delta $ is small enough.

\subsection{Construction of the special domain $\mathcal{R}^{\delta}_{\mathbf{ext}} $}

Fix a finite piece of an s-embedding satisfying \LipKd\, (for some $\kappa <1$) and \OneExpFat\ covering a square centered at the origin of the plane, denoted $\mathcal{C}_0 $, of width $40L_0 $ with $L_0 = \max(\frac{120000}{\gamma_0},120) $, where the constant $\gamma_0 $ comes from Theorem \ref{thm:F-via-HF} applied to the parameter $\kappa $. In what follows, the function $o_{\delta \to 0}(1)$ will denote a \emph{positive} function that converges to $0$ as $\delta $ goes to $0$ (and can in fact be made explicit). In the following construction, we work with a parameter $ \delta$ small enough, which we plan on using in a regime where it converges to $0$. We describe now step by step the construction of the discrete domain $\mathcal{R}^{\delta}_{\mathbf{ext}} $ we will work with, that will satisfy the assumption $\mathrm{Lip}(\kappa,5\delta)$ and have a fat-enough boundary. We advise the reader to follow the construction using Figure \ref{fig:extension_shape}.

\textbf{Step 1 : Extracting a boundary of fat-enough quads} 

Using \OneExpFat\, holds, all the vertex connected components of $\delta \exp(-\delta^{-1})$-fat quads are of diameter at most $1$. It is then immediate that there exist four sequences of neighboring tangential quadrilaterals, respectively denoted $\mathcal{N}_{\delta} $, $\mathcal{W}_{\delta} $ ,$\mathcal{S}_{\delta} $ and $\mathcal{E}_{\delta} $ (named following usual compass coordinates), such that:
\begin{itemize}
\item All the quads of $\mathcal{N}_{\delta} $, $\mathcal{W}_{\delta} $, $\mathcal{S}_{\delta} $ and $\mathcal{E}_{\delta} $ have a tangential circle whose radius is larger than $\delta \exp(-\delta^{-1}) $.
\item $\mathcal{N}_{\delta} $ is an injective sequence of neighboring tangential quads, linking inside $\mathcal{C}_0 $ the segments $ \{\pm 20L_0 \} \times [19L_0;20L_0] $ \emph{above} the horizontal level $y=+19L_0 $. It is drawn in green in Figure \ref{fig:extension_shape}.
\item $\mathcal{W}_{\delta} $ is an injective sequence of neighboring tangential quads, linking inside $\mathcal{C}_0 $ the segments $[-20L_0;-19L_0] \times  \{ \pm 20L_0 \}   $ \emph{on the left} the vertical level $x=-19L_0 $. It is drawn in pink in Figure \ref{fig:extension_shape}.
\item $\mathcal{S}_{\delta} $ is an injective sequence of neighboring tangential quads, linking inside $\mathcal{C}_0 $ the segments $ \{ \pm 20L_0 \} \times [-20L_0;-19L_0] $ \emph{below} the horizontal level $y=-19L_0 $. It is drawn in green in Figure \ref{fig:extension_shape}.
\item $\mathcal{E}_{\delta} $ is an injective sequence of neighboring tangential quads, linking inside $\mathcal{C}_0 $ the segments $[19L_0;20L_0] \times  \{ \pm 20L_0 \}   $ \emph{on the right} the vertical level $x=19L_0 $. It is drawn in pink in Figure \ref{fig:extension_shape}.
\end{itemize}

One can then easily construct a topological rectangle concatenating $\mathcal{N}_{\delta} $, $\mathcal{W}_{\delta} $ ,$\mathcal{S}_{\delta} $ and $\mathcal{E}_{\delta} $, whose boundary is only formed by $\delta \exp(-\delta^{-1})$-fat quads linking the four extremal squares of width $L_0 $ at the corners of $\mathcal{C}_0 $. From now on, we will still denote naturally $\mathcal{N}_{\delta} $, $\mathcal{W}_{\delta} $ ,$\mathcal{S}_{\delta} $ and $\mathcal{E}_{\delta} $ the boundary parts of the topological rectangle constructed via the concatenation of the four arcs. At the end of this first step, one has in particular two arcs, labeling them (starting from the bottom part of the graph) respectively $\mathcal{W}_{\delta} = (z^{w_\delta}_{k})_{1 \leq k \leq |\mathcal{W}_{\delta}|} $ and $\mathcal{E}_{\delta} = (z^{e_\delta}_{k})_{1 \leq k \leq |\mathcal{E}_{\delta}|}$. Those two sequences of quads are injective and satisfy $z^{w_\delta}_{k} \sim z^{w_\delta}_{k+1}$  and $z^{e_\delta}_{k} \sim z^{e_\delta}_{k+1}$ in $\diamondsuit(G)$. Both $\mathcal{W}_{\delta} $ and $\mathcal{E}_{\delta} $ are represented in pink in Figure \ref{fig:extension_shape}. The rest of the construction will be combination of symmetry arguments, horizontal alignments and explicit extensions.
  
\textbf{Step 2: Perform an horizontal alignment from above at the appropriate blue level}

The goal of this step is to find a vertical level to apply an horizontal alignment from above to \emph{all} tangential quadrilaterals that lie between $\mathcal{W}_{\delta} $ and $\mathcal{E}_{\delta} $ intersecting that level, while forcing that leftmost and rightmost created tangential quadrilateral (which are respectively obtained from alignment of quads of $\mathcal{W}_{\delta} $ and $\mathcal{E}_{\delta} $) to be fat enough. We explain now the construction in greater details.

There are at most $O( \delta^{-2} \exp( 2\delta^{-1}))$ tangential quadrilaterals in $\mathcal{C}_0$ with a radius $r_z \geq \exp (- \delta^{-1})$. In particular, there exist a horizontal level $y_{b}^{\delta}= -18L_0 + o_{\delta \rightarrow 0}(1) $ (drawn in blue in Figure \ref{fig:extension_shape}) which remains at a vertical distance at least $20\exp(-4\delta^{-1}) $ from \emph{all} vertices of $\Lambda(G) \cap \mathcal{C}_0$ belonging to one of the $\delta \exp(-\delta^{-1})$-fat tangential quads. Up to an arbitrary small vertical shift, we can also assume that $y_{b}^{\delta}$ doesn't intersect any vertex of $\Lambda(G)$ in $\mathcal{C}_0 $. Let $1\leq k_{\mathcal{W}_{\delta}} \leq |\mathcal{W}_{\delta}| $ be the index such that the tangential quadrilateral $z_{k_{\mathcal{W}_{\delta}}}$ intersects the vertical level $y_{b}^{\delta} $ and such that for any $l > k_{\mathcal{W}_{\delta}} $, all quads of $\mathcal{W}_{\delta} $ lie \emph{strictly above} the level $y_{b}^{\delta} $. In particular, $z_{k_{\mathcal{W}_{\delta}}}^{w_\delta}$ is the last quadrilateral of $\mathcal{W}_{\delta} $  that intersects $y_{b}^{\delta} $ (meaning that for all $l > k_{\mathcal{W}_{\delta}} $, the quad $z^{w_\delta}_{l}$ lies strictly above $y_{b}^{\delta}$). Define in a similar way the integer $k_{\mathcal{E}_{\delta}}$ such that  $z^{e_\delta}_{k_{\mathcal{E}_{\delta}}}$ is the last quadrilateral of $\mathcal{E}_{\delta} $ that intersects $y_{b}^{\delta}$. One then performs a horizontal alignment at level $y_{b}^{\delta}$  to \emph{all} tangential quadrilaterals that lie (horizontally) between $z^{w_\delta}_{k_{\mathcal{W}_{\delta}}}$ and $z^{e_\delta}_{k_{\mathcal{E}_{\delta}}}$ and intersect $y_{b}^{\delta}$. The new quads (and the associated vertices of $\Lambda(G)$) are drawn in blue solid. Due to Proposition \ref{prop:bad-levels}, after this horizontal alignment $z^{w_\delta}_{k_{\mathcal{W}_{\delta}}}$ and $z^{e_\delta}_{k_{\mathcal{E}_{\delta}}}$ are each transformed in at most $3$ tangential quadrilaterals with a radius $r_z \geq \exp(-40 \delta^{-1}) $. 

\textbf{Step 3 : Find the (dashed) red symmetrization level}

Recall that the leftmost and the rightmost tangential quadrilaterals attached to the level $y^{\delta}_{b}$ and between $\mathcal{E}_{\delta}$ and $\mathcal{W}_{\delta}$ have both a radius $r_z \geq \exp(-40 \delta^{-1}) $. The level $y^{\delta}_{r} = y^{\delta}_{b} + 10\exp(-160 \delta^{-1}) $ (again up to arbitrary small vertical shift if needed) can now be used to perform an horizontal alignment \emph{from above} at the level $y^{\delta}_{r}$ to all quadrilaterals intersecting $y^{\delta}_{r}$ (drawn in red dashes in Figure \ref{fig:extension_shape}). One obtains thanks to this procedure a 'strip' (filled in light pink) of height $\tilde{\delta} =  10\exp(-1600 \delta^{-1})$ between $\mathcal{W}_{\delta} $ and $\mathcal{E}_{\delta} $, formed by tangential quadrilaterals. Using again Proposition \ref{prop:bad-levels}, the leftmost and the rightmost tangential quadrilaterals of this 'strip' have a radius $r_z \geq \exp(-160 \delta^{-1}) $. Denote $S^{r}_{b}$ this strip of width $\tilde{\delta}$ and $S^{b}_{r}$ the symmetric picture with respect to $y^{\delta}_{b}$. Given that $\tilde{\delta} << \exp(-40 \delta^{-1})$, it is easy to check that there is indeed only one leftmost and one rightmost tangential quadrilaterals in the strip (meaning that there those extremal quads indeed connect the top to the bottom of the strip), as the alignment from above at level $y^{\delta}_{b}$ will corresponds this time to the case $3$ of the construction of Lemma \ref{lem:horizontal_allignement}.

\textbf{Step 4 : Perform the symmetrization}

We are now going to extend the picture obtained \emph{at the end of step 2} below the blue level $y^{\delta}_{b}$ by symmetrizing the strip between $y^{\delta}_{b}$ and $y^{\delta}_{r}$. In order to do that, paste (below $y^{\delta}_{b}$) alternatively the strips $S^{b}_{r}$ and $S^{r}_{b}$ until we reach the level $y=-40L_0+ o_{\delta \to 0}(1)$, starting with a strip of the type $S^{b}_{r}$ and finishing with a strip of the type $S^{r}_{b}$. Finally, paste below the lowest blue level of the strip region an horizontal symmetrization of the domain between $y_{\delta}^{b}$, the parts of $\mathcal{W}_{\delta} $,  $\mathcal{E}_{\delta} $ above $y_{\delta}^{b}$ and $\mathcal{N}_{\delta}$ (or more accurately for vertical arcs parts that correspond to indexes indexes larger than respectively $k_{\mathcal{W}_{\delta}}$ and $k_{\mathcal{E}_{\delta}}$). Those symmetrized sequence of quads are naturally denoted $\overline{\mathcal{W}_{\delta}} $,  $\overline{\mathcal{E}_{\delta}} $ (drawn in pink) and $\overline{\mathcal{N}_{\delta}} $ (drawn in green). Here the fact that all tangential quadrilaterals of the arcs $\mathcal{W}_{\delta} $ and $\mathcal{E}_{\delta} $ whose respective labels are larger than  $k_{\mathcal{W}_{\delta}} $ and $k_{\mathcal{E}_{\delta}} $ are above $y_{b}^{\delta} $ ensure that the obtained picture is proper as there is no overlap with the symmetrization. One can perform a similar construction in the upper part of the graph, swapping the role of the geographic poles in the notations.

\textbf{Step 5: Pasting squares via a layer a kites}

The strip region is a sequence of alternating strips $S^{b}_{r}$ and $S^{r}_{b}$, pasted below each other. In particular, since the construction is made using symmetries, the vertex boundary on the right of the strip region (in black solid in the bottom left Figure \ref{fig:extension_shape}) is delimited by a sequence of neighboring vertices of $\Lambda(G)$, labeled (up to swapping colors by) $ v^{\bullet}_{2k} \in G^{\bullet} $, $ v^{\circ}_{2k+1} \in G^{\circ}$, and  $v_{i} \sim v_{i+1} $ in $\Lambda(G)$. In our case, we label vertices in the natural order, from top to bottom. One can see that $\Re[v^{\bullet}_{2k} ]= \Re [v^{\bullet}_{2k+2}]$,  $\Re[v^{\circ}_{2k+1} ]= \Re [v^{\circ}_{2k+3}]$ and $\Im[v_{i}]-\Im[v_{i+2}] = 2\tilde{\delta} $. Let $x^{\delta}=\max (\Re [v^{\bullet}_{2}],\Re [v^{\circ}_{3}])$, which represents the rightmost point of the slit region. We treat the case where this rightmost point belongs to $G^{\bullet} $, a similar treatment can be easily done if it belongs to $G^{\circ} $. One can construct a vertical layer of \emph{kites} (filled in grey) formed by the points $v^{\bullet}_{2k}, v^{\circ}_{2k+1}, v^{\bullet}_{2k} $ and $\tilde{v}^{\circ}_{2k+1} = v^{\bullet}_{2k} - i\tilde{\delta} + \tilde{\delta}$. This kite has an area larger than the straight triangle  $(v^{\bullet}_{2k} v^{\bullet}_{2k+2}\tilde{v}^{\circ}_{2k+1})$ whose area is $\frac12\tilde{\delta}^2 $. In particular it is not hard to see that radius of the inscribed circle of one of those kites is at least $\tilde{\delta}^4$. Once this first layer of kites is constructed for the right part strip region, the right part of the vertex boundary of the strip region is now formed by a sequence of neighboring vertices in $\Lambda(G)$, with vertices of $G^{\bullet}$ vertically aligned, vertices of $G^{\circ}$ vertically aligned, and boundary edges (seen as vectors oriented from $G^{\bullet} $ to $G^{\circ} $) taking alternatively the values $\sqrt2 \tilde{\delta}e^{\pm i\frac{\pi}{4}} $. It is straightforward to extend this boundary into a region of the square lattice of edge length $\sqrt{2}\tilde{\delta}$ up to the vertical axis $x=40L_0+ o_{\delta \to 0}(1) $. We repeat the same construction in the upper left part of the tilling.

\textbf{Step 6: Checking the Lip condition}

One still needs to check that the constructed picture $\mathcal{R}^{\delta}_{\mathbf{ext}} $ satisfies $\mathrm{Lip}(\kappa,5\delta)$. When computing the increment of the origami map between two points of original part of the graph (before extension), it automatically holds, as well as in the square grid regions (where the origami map only takes, up to a global additive constant, the values $\sqrt{2}\tilde{\delta} $ or $0$). In the slit regions, the origami map has vertical increment (between two points vertically aligned) at most $2\tilde{\delta} << \delta $ while its horizontal increments corresponds are exactly the same as along $y^{\delta}_{b} $, which also satisfies $\mathrm{Lip}(\kappa,4\delta)$. Finally, when considering the increment of the origami map between points belonging to two different regions (original embedding, strip regions or square grid districts), it is enough to use the above observations together with the fact that the origami map is trivially $1$-Lipschitz for the tiny layers (of size much smaller than $\delta $) of transition between different regions, and thus the above results allow to conclude. 

\begin{figure}

\begin{minipage}{0.4\textwidth}
\hspace{-14mm}\includegraphics[clip, width=1.3\textwidth]{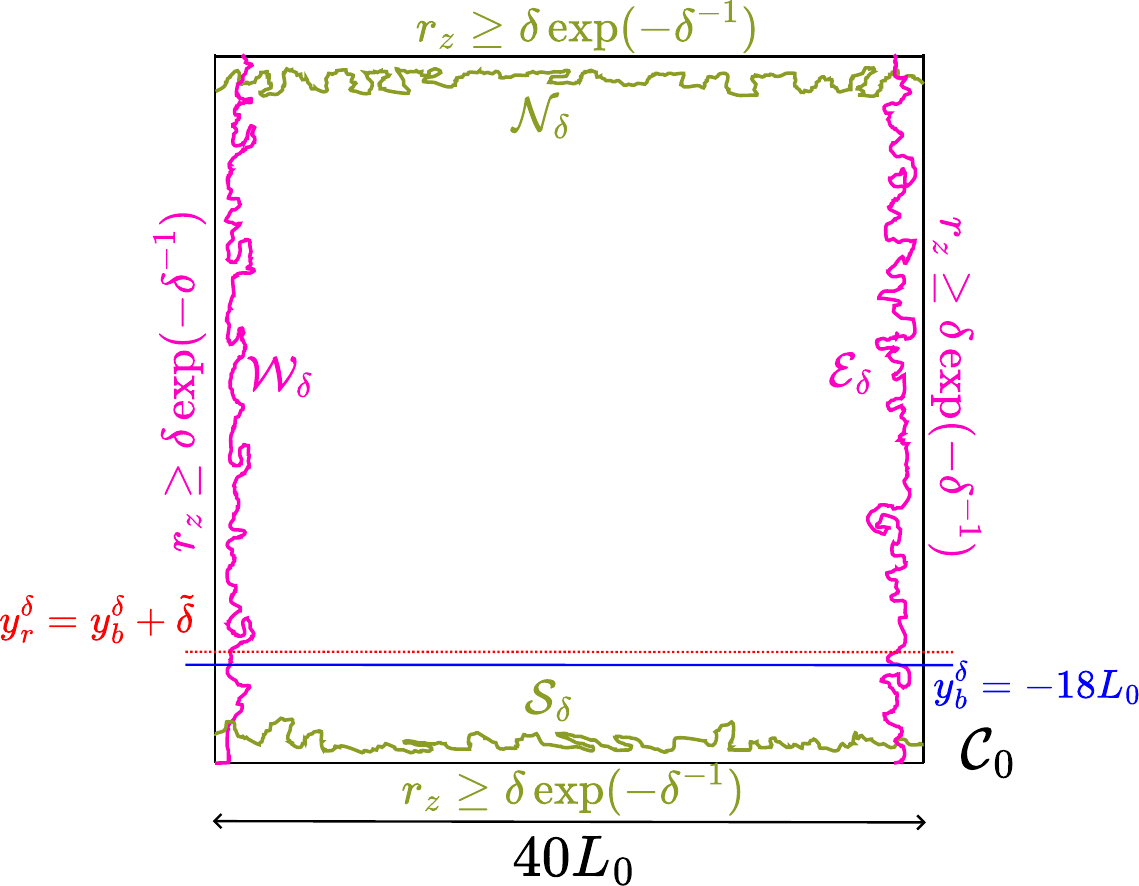}
\end{minipage}
\begin{minipage}{0.3\textwidth}
\includegraphics[clip, width=1.3\textwidth]{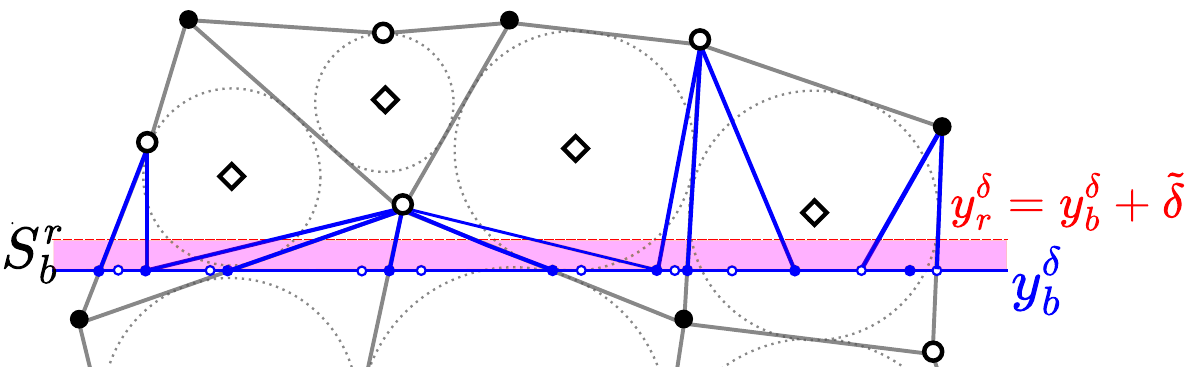}
\end{minipage}\hspace{8mm}
\begin{minipage}{0.20\textwidth}
\includegraphics[clip, width=1.3\textwidth]{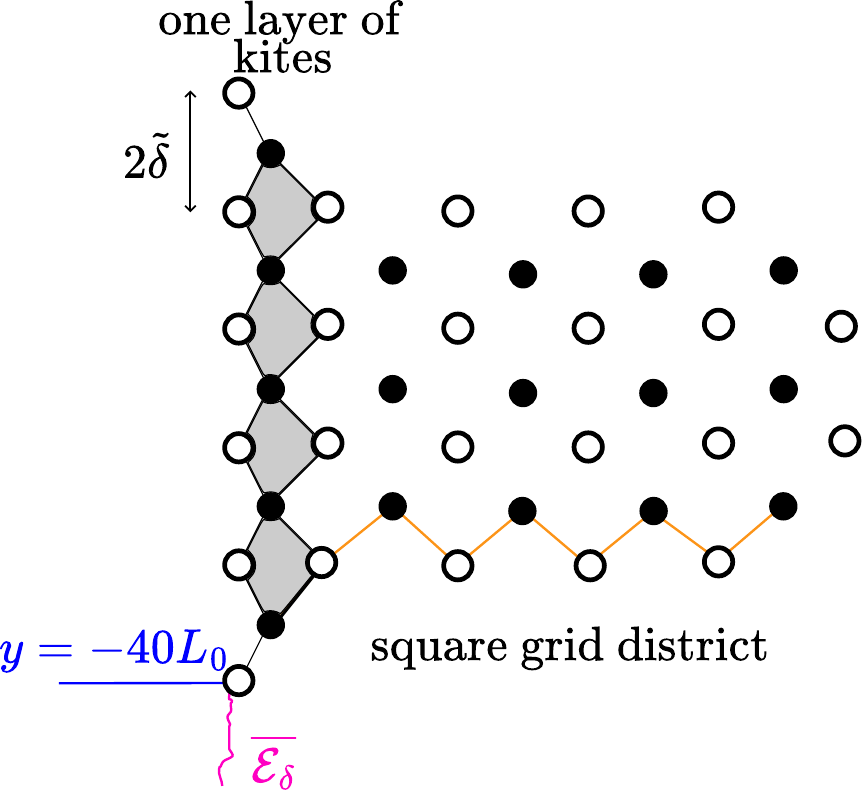}
\end{minipage}
\begin{minipage}{0.44\textwidth}
\includegraphics[clip, width=\textwidth]{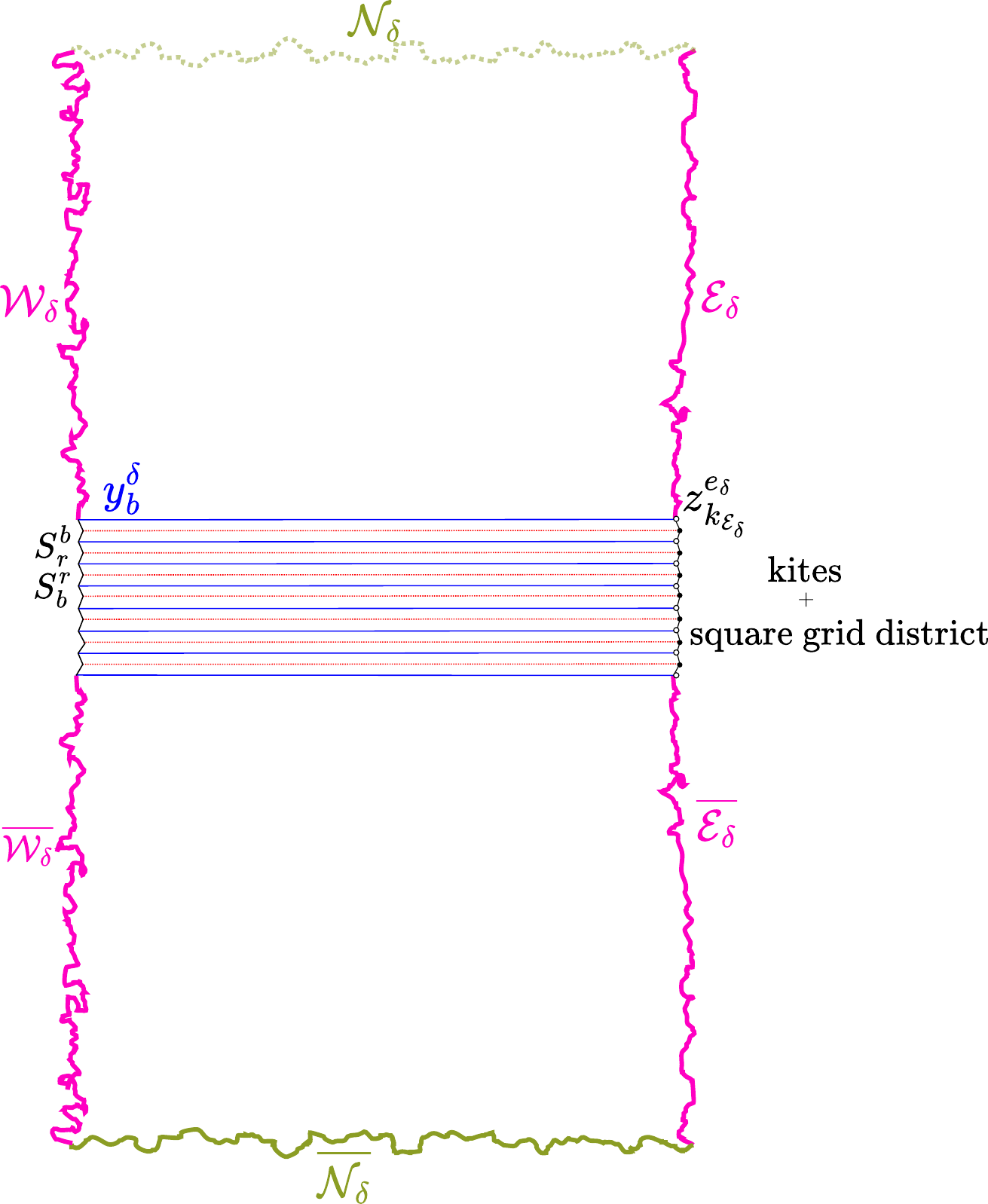}
\end{minipage}\hspace{5mm}
\begin{minipage}{0.44\textwidth}
\includegraphics[clip, width=\textwidth]{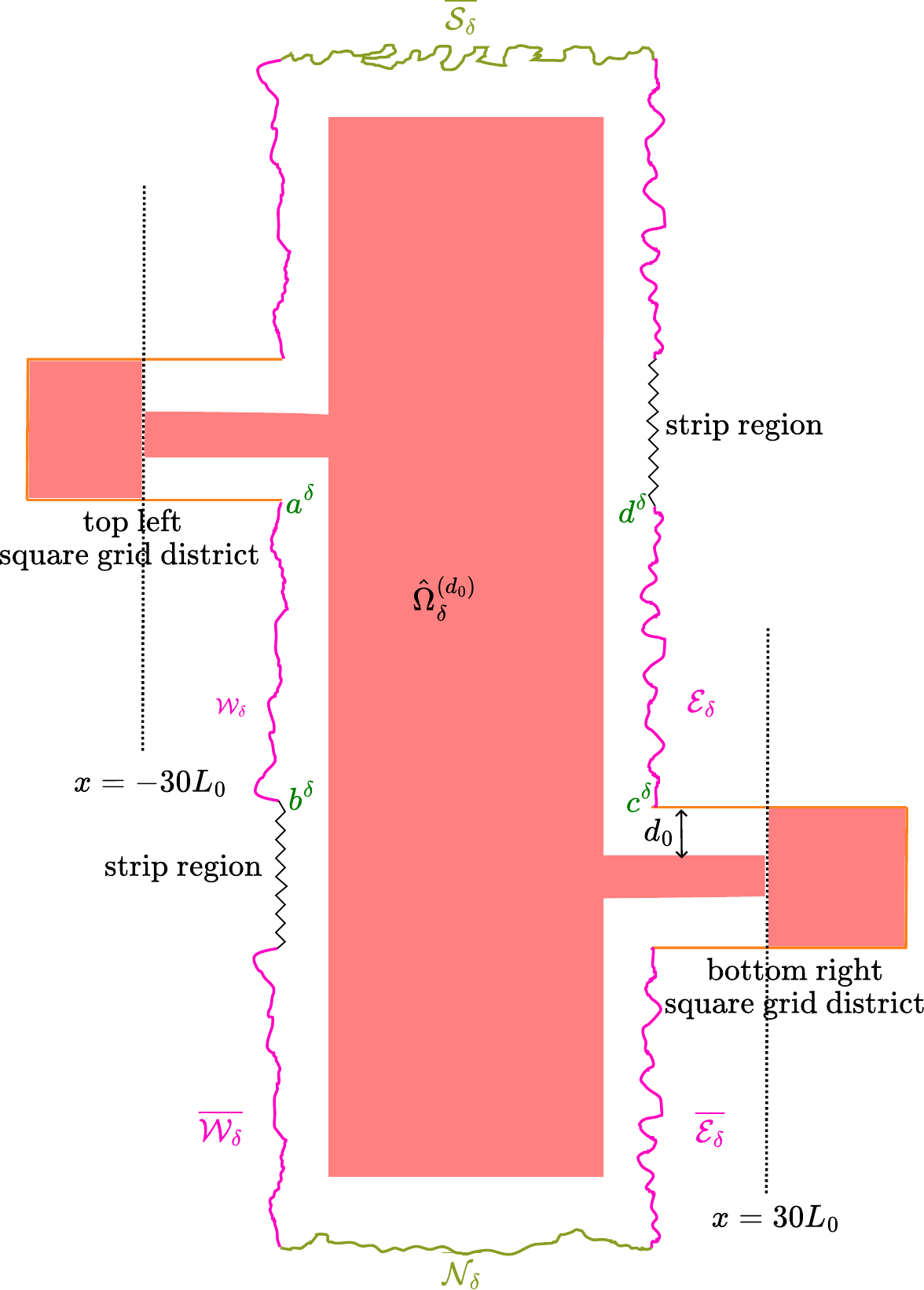}
\end{minipage}
\caption{TOP: (LEFT) The four cuts constructed in Step $1$ with $\delta \exp(\delta^{-1})$-fat quads can be concatenated to form a topological rectangle in $\mathcal{C}_0$. The blue level $y_{b}^{\delta}= -18L_0 + o_{\delta \rightarrow 0}(1) $ found in Step $2$ to perform an horizontal alignment from above. The level $y^{\delta}_{r} = y^{\delta}_{b} + \tilde{\delta} $ (red dashed) described in Step $3$ allows to construct the 'strip' $S^{r}_{b}$ (and its symmetric $S^{b}_{r}$). (MIDDLE) Horizontal alignment near $y_{b}^{\delta}$. The original embedding is in light solid while the new vertices and edges produced are drawn in blue. The strip $S^{r}_{b}$ is filled in light pink. The vertices of the horizontal alignment at level $y^{\delta}_{r}$ are not drawn. (RIGHT) Bottom boundary near the right square grid district at the end of Step $4$. The levels $y_{b}^{\delta}$ and $y_{r}^{\delta}$ ensure that one can construct a layer of fat-enough kites (filled in grey) and extend it to the square lattice. BOTTOM: (LEFT) Symmetrization performed in Step $4$, first by pasting below $y_{b}^{\delta}$ copies of the strips $S^{r}_{b}$ and $S^{b}_{r}$, then the arcs $\overline{\mathcal{W}_{\delta}} $  $\overline{\mathcal{E}_{\delta}} $ and $\overline{\mathcal{N}_{\delta}} $. (RIGHT) Description of the entire domain $\mathcal{R}^{\delta}_{\mathbf{ext}} $ with fat-enough boundary quads. The subdomain $\hat{\Omega}_{\delta}^{(d_0)}$ (filled in light red) connects the two square grid district. We impose that all vertices of $G^\circ$ on the bottom boundary of right square grid district are horizontally aligned.}\label{fig:extension_shape}
\end{figure}

We now set the boundary of $\mathcal{R}^{\delta}_{\mathbf{ext}} $ to be the domain whose boundary is formed by the part of $\mathcal{E}_{\delta}$ and $\mathcal{W}_{\delta}$ between the blue slicing levels, the symmetric (green dashes) of the arc formed by  $\mathcal{E}_{\delta}$, $\mathcal{N}_{\delta}$ $\mathcal{W}_{\delta}$ below the strip region, the left boundary of the strip region, the lower right square grid district constructed, and the similar construction performed in the upper part of the graph, forcing that \emph{all} boundary quadrilaterals of $\mathcal{R}^{\delta}_{\mathbf{ext}} $ have a radius $r_z \geq \exp(-10000 \delta^{-1})$. Moreover, one can easily impose that \emph{all} vertices of $G^\circ $ belonging to the bottom horizontal arc of the south left district are horizontally aligned (this will be relevant to obtain easily the boundary argument of discrete observables there). 
\medskip
One can now define a discrete simply connected domain $(\Omega_{\delta};a^{\delta},b^{\delta},c^{\delta},d^{\delta})$, seen as a discrete simply domain of an s-embedding with two wired boundary arcs $(b^{\delta}c^{\delta})^\circ$, $(d^{\delta}a^{\delta})^\circ$ and a two dual-wired boundary arcs $(c^{\delta}d^{\delta})^\bullet$ and $(a^{\delta}b^{\delta})^\bullet$ (see \cite[Section 6]{ChSmi2}, \cite[Figure 7]{Che20} or Figure \ref{fig:FK-domain}) using $\mathcal{R}^{\mathbf{ext}}_{\delta} $.  We started with injective sequence of fat-enough neighboring quads, which can be identified via a path $\ell^\delta $ in $\Lambda(G) $ such that all the vertices of that path only belong to quads with a large enough radius. Labeling the vertices of $\ell^\delta$  by $(v_{k})_{0\leq k \leq 2n} $ assuming e.g. that $v_{2i}=v_{2i}^{\circ} \in G^\circ$ and $v_{2i+1}=v_{2i+1}^{\bullet}  \in G^{\bullet}$, one gets for each $ v_{2i+1}^\bullet$ of $\ell^{\delta} $ sequence of counterclockwise ordered tangential quads between the vectors $\overrightarrow{ v_{2i+1}^\bullet v_{2i}^\circ} $ and $\overrightarrow{v_{2i+1}^\bullet v_{2i+2}^\circ} $ (and similarly for $ v_{2i}^\circ$). Concatenating them in the natural order  along the boundary of $\mathcal{R}^{\mathbf{ext}}_{\delta} $, one then gets a sequence of neighboring tangential quadrilaterals $(z_k)$ attached to the vertex boundary of $\mathcal{R}^{\mathbf{ext}}_{\delta} $. One can now apply a construction of the alternating Dobrushin arcs, similar to \cite[Figure 2]{ChSmi2} or \cite[Figure 7]{Che20}.  Each wired arc can be viewed as a sequence of \emph{boundary half quads} $(v_{2i}^{\circ} v_{2i+1}^{\bullet}  v_{2i+2}^{\circ} z_{i}) $ (see \cite[Section 5.3]{Che20}) with $z_i \in \diamondsuit(G)$ and all three vertices $v_{2i}^{\circ}, v_{2i+1}^{\bullet} , v_{2i+2}^{\circ} $ belonging to $z_i$. From a statistical mechanics perspectives, all the faces attached to that arc are wired, and all attached spins are in fact a single one. Similarly all disorders attached to a free arc are in fact a single disorder. The corners $a^{\delta},b^{\delta},c^{\delta},d^{\delta}$ correspond to edges of $\Lambda(G)$ linking the four arcs are located as in Figure \ref{fig:extension_shape}. In our case the two free arcs $(c^{\delta}d^{\delta})^\bullet$ and $(a^{\delta}b^{\delta})^\bullet$ correspond to the concatenation of the quads of $\mathcal{E}_{\delta}$ and $\mathcal{W}_{\delta}$ between the horizontal alignment levels, while the wired arcs $(b^{\delta}c^{\delta})^\circ$, $(d^{\delta}a^{\delta})^\circ$ correspond of quads in the remaining part of the boundary of $\mathcal{R}^{\mathbf{ext}}_{\delta} $.

\begin{figure}
\includegraphics[clip, scale=0.40]{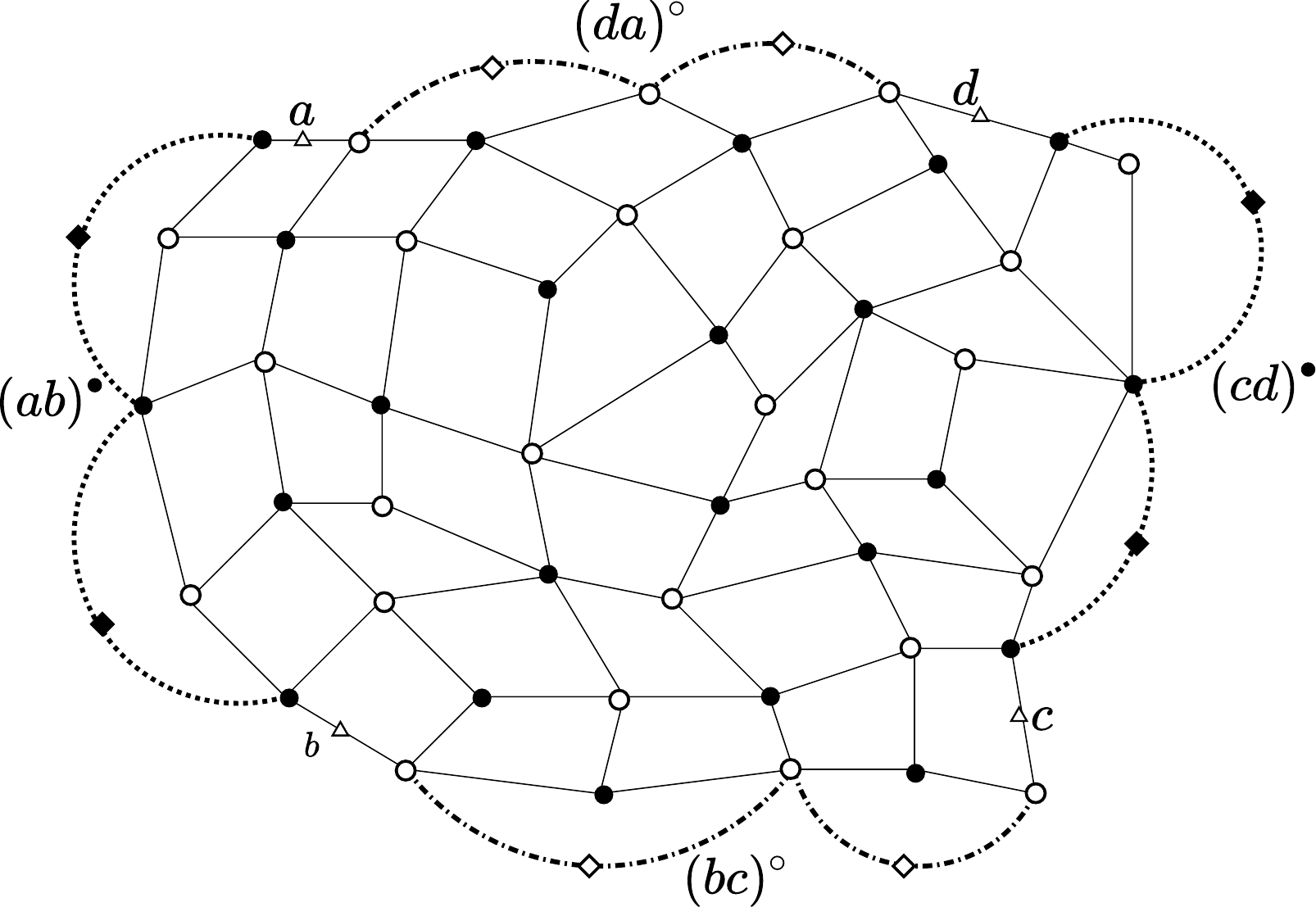}
\caption{ Domain $(\Omega_{\delta};a^{\delta},b^{\delta},c^{\delta},d^{\delta})$ seen as a discrete simply domain in $\cS^{\delta}$ with two wired boundary arcs $(b^{\delta}c^{\delta})^\circ$, $(d^{\delta}a^{\delta})^\circ$ (wired with irregular dashes) and a two dual wired boundary arcs $(c^{\delta}d^{\delta})^\bullet$ and $(a^{\delta}b^{\delta})^\bullet$ (wired with regular dashes). The wired boundary half quads are filled in white while the free boundary half quads are filled in black.}
\label{fig:FK-domain}
\end{figure}

\section{Proof of Theorem \ref{thm:positive-magnetization}}\label{sec:proof}

We now prove Theorem \ref{thm:positive-magnetization}. The proof of positive magnetization between opposite boundaries in the alternating wired/free/wired/free spin-Ising setup is done by contradiction, and we will need to rule out two different scenarios.

Using the monotonicity of the assumption \ExpFat\, in the variable $\rho$, one can restrict itself to the case $l=L_0 \rho $. Moreover, scaling the lattice (replacing $\delta $ by $\hat{\delta}= \delta \rho^{-1} $), one can assume that $\rho = 1 $. Assume now there exist a sequence of s-embeddings $(S^{\delta_k})_{k\geq 0} $ such that Theorem \ref{thm:positive-magnetization} fails. There exist two possible scenarios:
\begin{enumerate}[label=(\Alph*)]
\item Either $\hat{\delta}_k= \delta_k \rho_{k}^{-1}  \to 0$ as $k\rightarrow \infty $
\item Or  $\hat{\delta}_k= \delta_k \rho_{k}^{-1}$ is bounded from below along some subsequence as $k\rightarrow \infty $.

\end{enumerate} 

Let us note that the second scenario will be treated without using discrete complex analysis techniques and will be carried out at the end. In particular, the next subsection treats the case where  where $\hat{\delta_k} \rightarrow 0 $ and focuses on finding a contradiction between discrete observables and their continuous counterparts.

\subsection{Analytic proof by contradiction in the case (A)}

In this subsection and the rest of the proof, we keep assuming that $\rho = 1 $, and prove Theorem \ref{thm:positive-magnetization} first in the domain associated to $\mathcal{R}^{\delta}_{\mathbf{ext}} $ constructed in the previous section. In particular we keep from now on the notation $\delta $ instead of $\hat{\delta}$ to lighten notations, and this parameter will converge to $0$ along the given subsequence. Once the proof in is performed $\mathcal{R}^{\delta}_{\mathbf{ext}} $, standard monotonicity arguments (e.g. \cite[Theorem 7.6]{duminil-parafermions}) regarding the change of boundary conditions for the spin-Ising model allow to deduce the same statements for any discretization of a usual rectangle in $\cS^\delta $. Recall that one of the important features of this extended domain $\mathcal{R}^{\delta}_{\mathbf{ext}} $ is that its boundary contains two macroscopic pieces of the square lattice, which are called the square grid districts. The special attention made in the construction to have fat-enough boundary quads in $\mathcal{R}^{\delta}_{\mathbf{ext}} $ is crucial to ensure the boundedness (and thus precompactness)  of a sequence of s-holomorphic functions $F^{\delta}$, associated to 4-points Kadanoff-Ceva correlators, at least in a suitable open subset. The associated functions $H^{\delta} $ converge in their turn to $h=\int \Im [f^2 dz]+ |f|^2 d\vartheta $ in same region. The contradiction will come from the 'sign of the outer derivative' of $h$ (by this we mean whether $h$ tends to grow or decay near one of its boundary arcs). The introduction of the pieces of the square lattice allows to use the technology developed \cite{ChSmi2} to ensure that discrete Dirichlet boundary conditions of the function $H^\delta$ survive when passing to the continuous limit, at least in the square grid districts. This allows to bypass the use of
special cuts introduced \cite[Section 5]{Che20} to control the boundary behavior of $H^\delta$. Let us point out that, instead of the partial rewiring procedure between the wired arcs introduced in \cite[Equations (6.3) and (6.4)]{ChSmi2}, we present here a more transparent derivation. Still,  the rewiring procedure remains useful if one wants to prove convergence (and not only a lower bound) of correlation in the same 4 points setup.

Let $(\Omega_{\delta};a^{\delta},b^{\delta},c^{\delta},d^{\delta})$ be a discrete simply connected domain of an \mbox{s-embedding} $\cS^{\delta}$, with two \emph{wired} boundary arcs $(b^{\delta}c^{\delta})^\circ$, $(d^{\delta}a^{\delta})^\circ$ and a two \emph{dual-wired} boundary arcs $(c^{\delta}d^{\delta})^\bullet$ and $(a^{\delta}b^{\delta})^\bullet$ (see \cite[Section 6]{ChSmi2} ,\cite[Figure 7]{Che20} or Figure \ref{fig:FK-domain}). Here $a^{\delta},b^{\delta},c^{\delta},d^{\delta}$ are corners, separating the extremities of these four arcs and correspond to the places where the discrete function $H^\delta $ defined below jumps (it is constant along each of the four arcs). One can now define the Kadanoff-Ceva four points observable by setting $X^{\delta}(\cdot):=\E_{\Omega_\delta}[\chi_{(\cdot)} \mu_{(c^{\delta}d^{\delta})^\bullet}\sigma_{(d^{\delta}a^{\delta})^\circ}]$ via \eqref{eq:KC-fermions}. It is easy to see that 
\[ \begin{array}{ll}
 X^{\delta}(a^{\delta}) = \pm \mathbb{E}_{\Omega_\delta}[\mu_{(a^{\delta}b^{\delta})^{\bullet}}\mu_{(c^{\delta}d^{\delta})^{\bullet}}],  &   X^{\delta}(d^{\delta}) =\pm 1, \\
 X^{\delta}(b^{\delta})=\pm \mathbb{E}_{\Omega_\delta}[\mu_{(a^{\delta}b^{\delta})^{\bullet}}\mu_{(c^{\delta}d^{\delta})^{\bullet}}\sigma_{(b^{\delta}c^{\delta})^\circ}\sigma_{(d^{\delta}a^{\delta})^\circ}], & X^{\delta}(c^\delta)=\pm\E_{\Omega_\delta}[\sigma_{(b^{\delta}c^{\delta})^\circ}\sigma_{(d^{\delta}a^{\delta})^\circ}]. \ 
 \end{array}\]
Choosing properly the global additive constant in the definition of $H_{X^{\delta}}$ associated to $X^{\delta}$ via \eqref{eq:HX-def}, one has
\[ \begin{array}{ll}\label{eq:HF4-bc'}
 H_{X^{\delta}}((c^{\delta} d^{\delta})^\bullet)=1,  &   H_{X^{\delta}}((d^{\delta} a^{\delta})^\circ)=0, \\
 H_{X^{\delta}}((a^{\delta} b^{\delta})^\bullet)=1- \mathbb{E}[\mu_{(a^{\delta}b^{\delta})^{\bullet}}\mu_{(c^{\delta}d^{\delta})^{\bullet}}]^2, &   H_{X^{\delta}}((b^{\delta} c^{\delta})^\circ)=1- \E[\sigma_{(b^{\delta}c^{\delta})^\circ}\sigma_{(d^{\delta}a^{\delta})^\circ}]^2. \ 
 \end{array}\]
\bigskip

We focus on the specific situation where $(\Omega_{\delta};a^{\delta},b^{\delta},c^{\delta},d^{\delta})$ is the topological rectangle $\mathcal{R}^{\mathbf{ext}}_{\delta} $ constructed in Section \ref{sec:geometry} equipped with the Ising model. Up to passing to another subsequence, one can also assume that $(\cQ^\delta)_{\delta >0 }$ converges uniformly to a $\kappa <1 $ Lipschitz function $\vartheta $ on compacts subsets of the box $\mathcal{C}_0$ (since the functions $\cQ^\delta $ are automatically $1$-Lipschitz and defined up to additive constant). One denotes by $F^{\delta}$ and $H^{\delta}= H_{F^\delta}=H_{X^{\delta}}$ the functions naturally associated to the correlator $X^{\delta}$ define above via \eqref{eq:F-from-X} and \eqref{eq:HX-def}. Using the maximum principle for $H^{\delta}$ coming from Proposition~\ref{prop:HF-comparison}, one directly deduces from the boundary conditions \eqref{eq:HF4-bc'} that the functions $H_{F^\delta}$ are uniformly bounded by $1$ on $\Omega_\delta$. Recall that all the boundary faces of $\mathcal{R}^{\mathbf{ext}}_{\delta} $ have by construction a radius $r_z \geq \delta\exp (-\delta^{-1})$. Up to passing to another sub-sequence, one can also assume that the sequence of domains $(\Omega_\delta)_{\delta >0}$ converges to a domain $\hat{\Omega}$ (for the usual Hausdorff topology on bounded planar domains), whose discrete counterparts contain the regions corresponding the square grid district on the left of the axis $x=25L_0 $ and on the right of the axis $x=-25L_0$. Let $\hat{\Omega}_{\delta}^{(d_0)}$ the be union of the $d_0$-interior of $\Omega_\delta$ and the two pieces of the square grid district that lie respectively on the right of the axis $x=30L_0 $ and on the left of the axis $x=-30L_0 $, and $\hat{\Omega}^{(d_0)}$ its naturally defined counterpart. The next proposition is crucial as it gives precompactness of the 4-points observables in $\hat{\Omega}^{(d_0)}$ (that is naturally defined as in the discrete setting).

\begin{prop}\label{prop:comapctness}
Let $d_0>0 $ such that $d_0\gamma_0=120000$. The provided $\delta$ is small enough, the four points observables $F^{\delta}$ are bounded on compacts of $\hat{\Omega}^{(d_0)}$. In particular, the functions $F^{\delta}$ form a precompact family in the topology of uniform convergence on compacts of $\hat{\Omega}^{(d_0)}$.
\end{prop}

\begin{proof}
We work in the regime $\delta \to 0 $, assuming $d_0$ is fixed and $\delta $ small enough. One can easily, with exactly the same proof as given in \cite[Theorem 2,18]{Che20} that the alternative of Theorem \ref{thm:F-via-HF} can be rewritten for any s-embedding covering a discrete simply connected domain $D^\delta $ that contains strictly the ball $B(u,0.80r)$ (up to slightly modifying the values of the constants $r_0$ and $\gamma_0$, which we assume to have been done already when formulating the statement of Theorem \ref{thm:positive-magnetization})
\[
\begin{array}{rcl}
\text{either}\ \max_{\{z:\cS(z)\in B(u,\frac{1}{2}r)\}}|F|^2&\le& C_0r^{-1}\cdot\osc_{\{v:\cS(v)\in D^\delta\}}H_F\\[4pt]
\text{or}\ \max_{\{z:\cS(z)\in B(u,\frac{3}{4}r)\}}|F|^2&\ge& \exp(\gamma_0r\delta^{-1})\cdot C_0r^{-1}\osc_{\{v:\cS(v)\in D^\delta\}}H_F
\end{array}
\]

Fix now a point $u$ belonging to the $d_0$ interior of $\mathcal{R}^{\mathbf{ext}}_{\delta} $. We are going to rule out the second alternative presented above for the domain $D^\delta = \mathcal{R}^{\mathbf{ext}}_{\delta} $ (which satisfies $\textrm{Lip}(\kappa,5\delta)$) using the fat-enough boundary quads. Denote by $M^\delta=\osc_{\{ \mathcal{R}^{\mathbf{ext}}_{\delta} \}}H_F $. The first observation is that at \emph{boundary} quads that are all $\exp(-10000\delta^{-1})$-fat, we have the crude estimate $|F^{\delta}|^2 \leq M^\delta r_{z}^{-2} \leq M^\delta \exp(20000\delta^{-1})$ coming from \eqref{eq:F-from-X}. Applying the maximum principle for s-holomorphic functions (see \cite[Remark 2.9]{Che20}), this implies the same estimate for all quads inside the domain $\mathcal{R}^{\mathbf{ext}}_{\delta}$. In particular, as $ \frac{d_0 \gamma_0}{5}>20000$, this rules automatically out the possibility of the second alternative to hold (and in fact justifies the particular choice of the value of $d_0$). Indeed, either way the maximum principle for observables would imply that 
\begin{equation}\nonumber
 M^\delta \exp(20000\delta^{-1}) \geq  \max_{\{z:\cS(z)\in B(u,\frac{3}{4}d_0)\}}|F|^2 \geq M^\delta \exp(\gamma_0r\delta^{-1})\cdot C_0r^{-1},
\end{equation}
which fails as $\delta \to 0$. In particular, in the $d_0 $ interior of $\mathcal{R}^{\mathbf{ext}}_{\delta} $, the first alternative of the modification of Theorem \ref{thm:F-via-HF} has to hold and the observables $F^{\delta}$ thus bounded in that region (as $M^\delta =1 $ in the present case). Let us now treat the case of the square grid districts (of mesh size $O(\tilde{\delta})$). There, the origami map is $\frac12$-Lipchitz starting at a scale comparable to $\tilde{\delta}$. In particular, one can see that the second alternative of Theorem \ref{thm:F-via-HF} cannot hold as in those region, the assumption $\textsc{ExpFat}(\tilde{\delta},\tilde{\delta}^\frac12) $ holds and falls directly in the context of \cite[Corollary 2.20]{Che20}).
Once boundedness on compacts of $\hat{\Omega}^{(d_0)}$, one can apply directly \cite[Remark 2.12]{Che20} to get the precompactness statement.

\end{proof}
The preceding proposition generalizes \cite[Corollary]{Che20} using simply by the existence of fat-enough boundary quads, at the cost of achieving boundedness (and thus compactness) only at a fixed distance from the boundary. Let us mention that $\hat{\Omega}^{(d_0)}$ is connected, linking the two square grid districts. The previous proposition now ensures that the observables $F^{\delta}$ form a precompact family in $\hat{\Omega}^{(d_0)}$ thus there exist a subsequential limit of the family $(F^{\delta})_{\delta >0} $ such that, uniformly on compacts of $\hat{\Omega}^{(d_0)}$,
\begin{equation}\nonumber
F^\delta\to f,\quad H_{F^\delta}\to {\textstyle h=\int\Im[(f(z))^2dz] + |f(z)|^{2}d\vartheta}. 
\end{equation}

It is clear from the discrete estimate $0 \leq H^\delta \leq 1$ that $h$ takes its values in $[0,1] $. In general $d\vartheta \neq 0 $, and $f$ is not holomorphic. Still, the function $f$ \emph{is} holomorphic inside in the two square grid districts,  as $\vartheta $ is constant there. This also implies that $h=\int\Im[(f(z))^2dz]$ in that region is a harmonic function in the two square grid districts. Assuming $\E_{\Omega_\delta}[\sigma_{(b^\delta c^\delta)^\circ}\sigma_{(d^\delta a^\delta)^\circ}]\to 0$ as $\delta\to 0$ along one subsequence, we have
  $H_{X^{\delta}}((b^{\delta} c^{\delta})^\circ)=1- \E_{\Omega_\delta}[\sigma_{(b^{\delta}c^{\delta})^\circ}\sigma_{(d^{\delta}a^{\delta})^\circ}]^2 \to 1$ as $\delta $ goes to $0$. The contradiction is obtained in three steps:
\begin{itemize}
\item We show that $h$ has Dirichlet boundary values $0$ in a piece of the arc of the top left region and has Dirichlet boundary values $1$ in a piece of the bottom right region. This requires to show that discrete Dirichlet boundary conditions for the function $H^\delta $ survive when passing to the continuous limit.
\item The \emph{sign of the outer derivative} of the functions $H^\delta$ depends on the type of the arc (i.e. wired or free), at least in the square grid regions. In particular, discrete functions $H^\delta $ tend to grow near wired arcs while the continuous functions $h$ has to decay near that same arc by the maximum principle (recall that $h\in [0,1]$ thus $h$ cannot exceed $1$ near that part of the boundary). At the level of observables, this will translate in the fact that functions $F^\delta$ are purely real at the discrete bottom boundary of the southern district while continuous function $f$ is purely imaginary at that boundary. This leaves the only option that $f$ vanishes identically at that arc and thus in the entire southern square grid district. 

\item The Sloïlov factorization for $I_{\mathbb{C}}[f] $ implies that the function $f$ should actually vanish everywhere in $\hat{\Omega}^{(d_0)}$, contradicting the change of boundary values of $h$ between different districts.
\end{itemize}

The author is grateful to Mikhail Basok for pointing out the Stoïlov factorization allowing to obtain the final contradiction.

\medskip
\underline{\textbf{Step 1: Boundary behavior of the continuous functions $f$ and $h$}}

The proof can be followed with Figure \ref{fig:extension_shape}. Recall that we assume that the correlation $\E_{\Omega_\delta}[\sigma_{(b^{\delta}c^{\delta})^\circ}\sigma_{(d^{\delta}a^{\delta})^\circ}]^2$ vanishes in the limit along one subsequence as $\delta \rightarrow 0 $, thus $H_{X^{\delta}}((b^{\delta} c^{\delta})^\circ)= 1- \E_{\Omega_\delta}[\sigma_{(b^{\delta}c^{\delta})^\circ}\sigma_{(d^{\delta}a^{\delta})^\circ}]^2\rightarrow 1$ at the wired arc $(b^{\delta}c^{\delta})$. This statement remains in particular true at the horizontal arcs of the bottom right square grid district. \emph{In the square grid districts and a priori only there}, one can apply \cite[Proposition 3.6]{ChSmi2} and deduce directly that
\begin{itemize}
\item The function $H_{F^{\delta}}$ is \emph{sub-harmonic} on $G^{\circ}$ for the natural Laplacian on $G^{\circ}$ defined in \cite[Equation (3.1)]{ChSmi2}.
\item The function $H_{F^{\delta}}$ is \emph{super-harmonic} on $G^{\bullet}$ for the natural Laplacian on $G^{\bullet}$ defined in \cite[Equation (3.1)]{ChSmi2}.
\end{itemize}
Using the \emph{boundary modification trick} of \cite[Lemma 3.14, Remark 3.15]{ChSmi2} to compare $H^{\delta}$ with discrete harmonic functions proves, exactly as in \cite[Theorem 4.3]{ChSmi2}) that discrete Dirichlet boundary conditions survive when passing to continuum i.e. $h$ extends continuously to $1$ at the horizontal arc of the bottom right square grid district. In continuum, the constant Dirichlet boundary conditions of $h$ at the bottom horizontal segment of the square grid district (drawn in orange in Figure \ref{fig:extension_shape}) and the fact that $0\leq h\leq 1$ imply that $f$ extends continuously up to the bottom boundary of the square grid district. Moreover, $f^2\in \mathbb{R}^{-}$ near that arc, i.e. $f\in i\mathbb{R}$, as in \cite[Proof of Theorem 1.3]{Che20}. One can also note that a similar reasoning on survival of Dirichlet boundary conditions applied at the top square grid district proves that $h=0$ near the left-most upper vertical segments of that district.

\underline{\textbf{Step 2: Boundary behavior of discrete functions $F^{\delta}$}}
\medskip

For discrete observables $F^\delta$, at a boundary quad $z_{\partial} \in \cS^\delta (\diamondsuit(G))$ of the arc approximating the bottom horizontal segment of the square grid district, the increment of $H_{F^{\delta}}$ between two consecutive (from left to right) vertices of $\cS(G^{\circ}) $ vanishes identically (this is a direct consequence of \eqref{eq:HF4-bc'}). That same increment is also  positively proportional to $\Im [F^{\delta}(z_{\partial})^2]$. This allows to conclude that  $F^{\delta}(z_{\partial})^2 \in \mathbb{R}$ and one can even go beyond that observation, as the value of the boundary argument of $F^{\delta}(z_{\partial})$ is given directly by the formula \cite[Lemma 5.3]{Che20}, which implies that $F^{\delta}(z_{\partial})\in \mathbb{R}$. In particular $F^{\delta}(z_{\partial})$ is purely real at the boundary, which means that $z'\mapsto \Im[F^{\delta}(z')] $ vanishes identically at the bottom boundary of the south square grid district. 

In the square grid district, $z'\mapsto \Im[F^{\delta}(z')]$ is a martingale for the standard random walk on quads (i.e. the probability to leave from $z\in \diamondsuit(G)$ to one of its four neighbors is $\frac{1}{4}$). This is a simple implication of the discrete Cauchy-Riemann equation satisfied by $F^{\delta}$ (see e.g. \cite[Equation (3.1)]{CHI}). Set $z_0=(35L_0;-40L_0+s)$ and denote by  $\mathcal{R}_{s^{\frac14}}^{\delta} $ the square of width $2s^{\frac14}$ centered at $(35L_0;-40L_0++s^{\frac14}) $. Using the standard gamblers ruin estimates for random walks on $\mathbb{Z}^{2}$, the probability that the random walk associated to $\Im[F^{\delta}(z')]$ leaves $\mathcal{R}_{s^{\frac14}}^{\delta} $ from its top side is $O(s^{\frac{3}{4}})$. On the other hand, using uniform crossing estimates for the standard random walk, the probability that this walk leaves $\mathcal{R}_{s^{\frac14}}^{\delta} $ from one of its vertical sides at a vertical distance $\ell$ from the bottom boundary of $\mathcal{R}_{s^{\frac14}}^{\delta} $ is bounded by $\ell^{\frac{1}{2}+\varepsilon} $, for some $\varepsilon >0 $ which is independent from $s$ and $\tilde{\delta}$ (the mesh size of the square lattice in the southern district). We apply now the optional stopping theorem to the time the walk started at $z_0$ exits $\mathcal{R}_{s^{\frac14}}^{\delta} $ to reconstruct the value $\Im[F^{\delta}(z_0)]$. One has
\begin{itemize}
\item The contribution to $\Im[F^{\delta}(z_0)]$ coming from the bottom side of $\mathcal{R}_{s^{\frac14}}^{\delta} $ vanishes identically as $\Im[F^{\delta}]$ vanishes there (which is a consequence of the discussion above).
\item The contribution of the top side is bounded by $O(s^{\frac{3}{4}}\cdot s^{-\frac{1}{8}} )$ as $F^{\delta}$ is bounded there by  $s^{-\frac{1}{8}}$ in the top segments of $\mathcal{R}_{s^{\frac14}}^{\delta} $.
\item The contribution of the vertical sides is polynomially small in $s$, as the probability to leave from one of the vertical sides at a height $\ell $ (dashed pinked path in Figure \ref{fig:contradiction}) from the bottom side is bounded by $\ell^{\frac{1}{2}+\varepsilon} $ and we have the bound $|F^{\delta}(z)|=O( \text{dist}(z,\partial \Omega)^{-\frac{1}{2}}) $. To get this bound on the southern district formed by a square grid of mesh size $\tilde{\delta}$, one has to separate two cases. 
\begin{itemize}
\item If the distance from $z$ to the bottom boundary is larger than $\text{cst}\tilde{\delta} $ (for some uniform constant $\text{cst}$), one can apply the first alternative of Theorem \ref{thm:F-via-HF}.
\item if the distance from $z$ to the bottom boundary is smaller than $\text{cst}\tilde{\delta} $, the reconstruction of $F^\delta$ via $ X^\delta$ given in \eqref{eq:F-from-X} ensures directly that $F^\delta = O(\tilde{\delta}^{-\frac12})$.
\end{itemize}
 
\end{itemize}
All together, this ensures that $|\Im[F^{\delta}(z_0)]| = O(s^{\beta''})$ for some positive exponent $\beta''$. Sending first $\delta $ to $0$ and then $s$ to $0$ implies that $f$ is purely real at $(35L_0;-40L_0)$ and thus has to vanish  at that point, as it is also purely imaginary according the step 1 of the proof given above). Repeating the same reasoning ensures $f$ in fact vanishes in the entire bottom arc of the south square grid district contained in $\hat{\Omega}^{(d_0)}$. Since $f$ is holomorphic in the square grid district and vanishes on a boundary arc, it vanishes everywhere in the bottom square grid district. This implies that its primitive $I_{\mathbb{C}}$ is constant in that region. 
\bigskip

\underline{\textbf{Step 3: Final contradiction using the Stoïlov factorization}}

We are now in position to get the final contradiction. Consider the function $g(\zeta)=I_{\mathbb{C}}(\zeta)$ defined by \eqref{eq:Ic-to-phi-change} in the $\zeta $ conformal parametrization of the surface $(z,\vartheta(z))$. As explained in Section \ref{sub:shol-limits}, $g$ satisfies a conjugate Beltrami equation \eqref{eq:g_beltrami} with a Beltrami coefficient which is bounded away from $1$ (this bound only depends on $\kappa$). Moreover $g$ is constant in the bottom square grid district. Writing the Stoilov factorization $g= \underline{g} \circ p $ with $p:\hat{\Omega}^{(d_0)} \rightarrow \hat{\Omega}^{(d_0)}$ a $\beta(\kappa)$-Hölder homeomorphism and $\underline{g}$ an holomorphic function, the function $\underline{g}$ is constant in the south square grid district, thus constant everywhere in $\hat{\Omega}^{(d_0)}$. In return, this proves that $f$ vanishes everywhere in $\hat{\Omega}^{(d_0)}$. This contradicts the change of boundary values from $0$ to $1$ for the function $h=\int \Im [f^2 dz]+ |f|^2 d\vartheta $ between the north and the south square grid districts.\begin{figure}
\includegraphics[clip, scale=0.60]{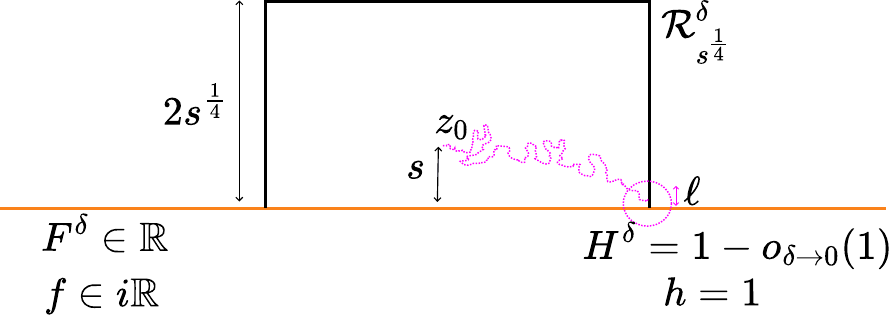}

\caption{Local notation related to the boundary analysis at the south square grid district. The sign of the outer derivative are opposite for discrete $H^\delta $ and continuous $h$, imposing that the continuous observable vanishes at that horizontal arc. The dashed pink path represents the events where the random walk started at $z_0$ exists $\mathcal{R}_{s^{\frac14}}^{\delta} $ at a distance smaller than $\ell$ from the bottom boundary of the square grid district.}
\label{fig:contradiction}
\end{figure}

\subsection{Contradiction in the scenario (B) and conclusion of the proof}

Recall that it is still possible that Theorem \ref{thm:positive-magnetization} fails under scenario (B) of the alternative stated at the beginning of Section \ref{sec:proof}, meaning there exist a sequence of s-embeddings $(S^{\delta_k})_{k\geq 0} $ such that the boundary correlation vanishes as $k\rightarrow \infty $ while $\hat{\delta}_k= \delta_k \rho_{k}^{-1}$ is bounded from below. We will rule out this scenario with simple statistical mechanics arguments. Recall that we scaled the lattice such that $\rho =1 $. In particular in that case all the $\delta_k $ are bounded from below (independently of $k$). In particular for each $k$ there exist a path of neighboring quads from top to bottom of the square $\mathcal{C}_0 $ formed by quads whose radius is \emph{bounded from below}. Such path has two features: first it has a bounded number of steps, only depending on $\kappa$ (one can easily see this as the area of each of such quad is bounded from below and the total area of $\mathcal{C}_0 $  is bounded from above). Moreover, the associated Ising coupling constant along such path are bounded away from $0$ and infinity, as one can see from the reconstruction of the Ising parameter from the geometry of the embedding given in \eqref{eq:theta-from-S}. One can now use the finite energy property for the FK-Ising model to ensure that the FK connection probability from the top to the bottom of $\mathcal{C}_0 $ is bounded from below, which rules out scenario (B).

\medskip

We are now in position to conclude. Using monotonicity arguments in scenario (A) or repeating the above arguments in scenario (B) implies the existence of the announced $p_0>0$ for the domain delimited by $\mathcal{N}_{\delta} $, $\mathcal{W}_{\delta} $, $\mathcal{S}_{\delta} $ and $\mathcal{E}_{\delta} $ with four alternating arcs (here the 'vertical' \emph{free} arcs correspond to the parts of $\mathcal{W}_{\delta} $ and $\mathcal{E}_{\delta} $ between the slicing levels) and thus transfer directly a similar estimate to discretizations of standard rectangles. This finishes the proof. Let us point out that, in order to propose a simplified description of the construction extended domain $\mathcal{R}^{\delta}_{\mathbf{ext}} $ and its extension, we actually proved a lower bound on boundary correlations in rectangle of aspect ratio (between vertical and horizontal side) which actually close to $\frac{18}{19}$ and not $3$ as originally stated in Theorem \ref{thm:positive-magnetization}. Simple modifications in the construction of the extended domain treat the original problem with the exact same proof.

\section{Construction of the embedding in explicit setups and further perspectives}

In this short section, we explain in greater details how the s-embedding setup applies in already known and new contexts.

\subsection{Description of the embedding procedure in several cases}\label{sub:explicit-construction}

\medskip

\textbf{\newline Critical isoradial lattice}

Recall that an isoradial embedding is an embedding in the plane such that each face is inscribed in a circle, the center of each circle is inside the face, and all the circles radii are equal (say to $\delta$). In that case, the graph $\Lambda(G)$ is a \emph{rhombus} tiling. When the grid is equipped with critical Z-invariant weights introduced by Baxter in \cite{Baxter-book} (invariant under star-triangle transformations, see \cite{BdTR1,BdTR2} for a modern exposition), counterparts to Theorems \ref{thm:positive-magnetization} and \ref{thm:RSW} were proven in \cite{ChSmi2}, under the bounded-angle assumption,  starting from a solution to \eqref{eq:3-terms} being $X(c):=\delta \eta_c $, defined in \eqref{eq:def-eta}, with the arbitrary embedding $\cS $ being in that case one isoradial embedding. In particular one can see that in that setup $\rho = O(\delta) $ where $O$ only depends on the bounds used of the bounded angle property. Removing the bounded angles assumption and replacing it with \ExpFat\, allows to apply directly \cite[Proposition 5.1]{DCLM-universality} to derive the box crossing property for the Quantum Ising model on $\mathbb{Z} \times \mathbb{R} $.

\textbf{Massive isoradial lattice}

In this case, we still work on an isoradial grid (chosen with the bounded angles property) but this time being equipped with near critical weights, scaled such that the \emph{nome} equals $q = \frac{m\delta}{2} $ (see \cite[Section 1]{park-iso} or \cite[Equation (1.1)-(1.3)]{Cim-universality}) on the isoradial embedding of mesh size $\delta$. In the classical formulation of the Ising model, it corresponds to looking at the homogeneous model on $\frac{1}{n}\mathbb{Z}^2 $ at the uniform inverse temperature $\beta_c + \frac{m}{2n} $. In this framework, the parameter $m$ is called the mass and measures how far from criticality is the system. Using the so-called discrete exponentials introduced and studied in \cite{BdTR1,BdTR2} as solution to \eqref{eq:3-terms}, one can re-embed the near critical model and construct an s-embedding as done in details in \cite[Section 3.3]{Cim-universality}. In particular, in the embedding given in \cite[Theorem 3.19]{Cim-universality}, the space-like surface constructed in the Minkowski space $\mathbb{R}^{2,1}$ (see discussion in the next sub-section) is \emph{of mean-curvature equal to the mass} $m$ (scaled in the surface embedded in $\mathbb{R}^{2,1} $). It is not hard to see in the formula \cite[Equation (3.12)]{Cim-universality} that, when the mass $|m|\rightarrow \infty$, the smallest possible Lipschitz constant for the origami map $\cQ^\delta $ gets closer and closer to $1$. This fact should be compared with the results of \cite[Section 6]{park-iso} that state that crossing probabilities go respectively to $0/1$ when $|m|\rightarrow \infty$, approaching the off-critical regime. In this case, crossing estimates are not bounded away from $0$ and $1$ while the optimal Lipschitz constant for the origami map gets closer and closer to $1$. In that setup $\rho = O(\delta) $ where $O$ only depends on the bounds used of the bounded angle property and the mass parameter $m$.

\textbf{Critical double-periodic lattice}

In that setup, the criticality condition was derived in \cite{cimasoni-duminil} by Cimasoni and Duminil-Copin. The question of finding an embedding for the critical model was settled by Chelkak in \cite[Lemma 2.3]{Che20}. In this lemma, it is used that the real linear vector space of periodic solutions to \eqref{eq:3-terms} is two dimensional, and one particular cleverly chosen complex combination of a given basis of that two dimensional vector space allows to construct the canonical embedding, which is in particular a doubly-periodic in the $\cS $ plane. Let us mention that from the crossing estimates perspective, such clever choice a complex combination in the vector space is \emph{not} needed as any non trivial complex combination of periodic linear independent solutions to \eqref{eq:3-terms} leads to a proper non-degenerate s-embedding (see \cite[Lemma 2.3]{Che20}), and produces an origami function which grows linearly with a Lipschitz constant strictly smaller than $1$. In that setup, one can chose $\rho = O(\delta) $ where $O$ only depends on how far the weights of the fundamental domain remain bounded away from $0$ and infinity.

\textbf{Lis' circle patterns}

In \cite{lis-kites}, Lis introduces the notion of Ising model on circle patterns depending on some inverse temperature $\beta$, the latter being constructed as follows (translating the notation of the original paper to our framework). A circle pattern is a pair of planar graphs $G=(G^\bullet,E) $ and $G^\star=(G^\circ,E^\star)$ mutually dual to each other, embedded in the complex plane in such a way that each face of $G^\star$ is inscribed in a circle inside the closure of the face, and the vertices of $G$ lie at the center of the circles. In that case, the faces of the graph $\Lambda(G)=G^\circ \cup G^\bullet $ are \emph{kites}, which fits into the s-embeddings framework and the weights fit at the inverse temperature $\beta =1$. Seen as an s-embedding, the origami map is constant (set e.g. to vanish) on $\cS(G^\circ) $ and takes the edge-length of the kites on $\cS(G^\bullet)$. In particular in the case of graphs with non-smashed angles and edge-lengths of kites comparable to some $\delta $ (assumption \Unif\,, in \cite{Che20}), one can easily see that $\cQ^\delta = O(\delta)$. In that setup one can take $\rho = O(\delta) $ where $O$ only depends on the bounded angle property of the kites. Moreover, our theorem allows to study the model without those bounded angle and edge-length comparability properties, in particular in the case of \emph{circle packings} presented in \cite[Example 2]{lis-kites}, coming potentially from wilder graphs as one can see in the next example.

\textbf{Circle patterns coming from finite random planar triangulation}

We continue the previous discussion, but this time applying Lis' circle patterns weights to a random triangulation with boundary. We keep exactly here the notations of \cite{GGJN}. The discrete mating-of-trees,  introduced by Duplantier, Gwynne, Miller and Sheffiled \cite{duplantier2021liouville,gwynne2021tutte} is a model of random triangulation in the plane of the vertex set $\varepsilon \mathbb{Z} $. This random triangulation, indexed by a positive number $\varepsilon >0$ and $\mathcal{G}^\varepsilon $, is built using a pair of correlated Brownian motions (see \cite[Section 1.2]{GGJN} for a short introduction to that setup), and belongs to the universality class of a $\gamma$-LQG surface (with $\gamma \in ]0;2[$). Doing minor modification procedures as in \cite[Section 1.2]{GGJN} (at the boundary of a finite set and removing some double edges), the work of \cite{GGJN} explains how to extract from $\mathcal{G}^\varepsilon $ a finite triangulation with boundary, denoted $(\mathcal{G}^{\varepsilon}_2, \tilde{\rho})$, with a special marked interior vertex $\tilde{\rho}$ (which corresponds to a particular vertex in the construction of the random triangulation from $\mathcal{G}^\varepsilon $). Then, $(\mathcal{G}^{\varepsilon}_2, \tilde{\rho})$ can then be circle packed in the unit disc $\mathbb{D}$, uniquely up rotations provided we impose that the circle $C_{\tilde{\rho}}$ is centered at the origin. This circle packing is denoted by $\mathcal{P}_{\tilde{\rho}}^\varepsilon = \{ C_v \}_{v\in \mathcal{V}\mathcal{G}^{\varepsilon}_{2}} $. One can now apply Theorem 1.4 of \cite{GGJN} ensure that, with probability $1-o_{\varepsilon \to 0}(1)$, the maximal size of one circle in $\mathcal{P}_{\tilde{\rho}}^\varepsilon$ is of order $O(\log(\varepsilon^{-1})^{-1})$, with the estimates $O$ and $o$ only depend on $\gamma \in ]0;2[$.

We are now going to use Lis' critical weights on circle patterns to decorate the faces of this rooted circle packed triangulation (see Figure \ref{fig:circle-packing}) with an Ising model. Set the vertices of $G^\bullet $ to be located at the center of the circles of $\mathcal{P}_{\rho}^\varepsilon$ and while the vertices of $G^\circ $ correspond to the center of the circles of the dual triangulation, which correspond here to the centers of the inscribed circles attached to faces of that triangulation. This means that each vertex of $G^\circ $ in fact corresponds exactly to one face of the original circle packing. This context falls into Lis' framework of circle patterns, and the graph $\Lambda(G) $ is indeed formed of kites. We can then assign a coupling constant $x(e)=\tan \frac{\theta_e}{2} $ to the edge of the triangulation associated to the quad $z_e \in \Lambda(G)$, where the abstract angle $\theta_e $ is given by the angle formula \eqref{eq:theta-from-S}. 

Set $\delta = -\log(\varepsilon^{-1})^{-1} $. Then with probability $1-o_{\delta \to 0}(1) $, one has the origami map that satisfies $\cQ=O(\delta)$. One still needs to check that some kind of \ExpFat\, condition holds to be able to apply Theorem \ref{thm:RSW}. This can done using the SLE/LQG estimate coming from \cite[Proposition 6.2]{holden2018sle} (see also \cite[Lemmas 2.5-2,6]{gwynne2021tutte}), that state that with probability $1-o_{\varepsilon \to 0}(1) $ (which can be made explicit in $\varepsilon $), the triangulation contains at most a polynomial power of $\log(\varepsilon)^{-1}$ vertices. Thus in that context, one can chose $\rho = O(\delta^\tau) $, where the constant $O $ and $\tau$ only depend on the explicit bounds in the $o_{\varepsilon \to 0}(1)$ estimates (and thus on $\gamma$). 

In conclusion, decorating the circle packing $\mathcal{P}_{\rho}^\varepsilon$ with weights coming naturally from the associated circle pattern, the graph satisfies, with probability $1-o_{\varepsilon \to 0}(1) $, the usual RSW property is fulfilled starting at a polynomial distance compared to $\log(\varepsilon)^{-1} $.

\medskip

\paragraph{\textbf{Layered model in the zig-zag grid}}

We discuss here the notion of s-embedding attached to the so-called Ising model on the zig-zag lattice, defined and studied in greater details in \cite[Section 5.2]{CHM-zig-zag-Ising}, (see also Figure \ref{fig:circle-packing} for a more explicit construction). One can start with the (abstract, meaning with no geometrical interpretation of the weights) rotated square grid (by $\frac{\pi}{4}$) and consider the Ising model on faces of that graph, chosen with all coupling constants attached to the edges separating the spins of two neighboring verticals columns $C_{k} $ and $C_{k-1}$ being the same and equal to some $x_{k}$. The collection of coupling constants $(x_{k})_{k \in \mathbb{Z}} = (\tan(\theta_k))_{k \in \mathbb{Z}}$ defines an Ising model on the faces of $e^{i\frac{\pi}{4}}\mathbb{Z}^2$. One natural question is to find a proper s-embedding attached to this model and describe one of its realizations in the plane. It is possible to make such a construction using basic euclidean geometry as in \cite[Figure 5]{CHM-zig-zag-Ising} or Figure \ref{fig:circle-packing}  by pasting layers of tangential quadrilaterals with an inscribed circle of radius $1$. In that case, the embedding and origami map are easy to compute and give rise to an explicit formulae. The vertical increments of $\cQ$ between vertices of the same color in $\cS(\Lambda(G))$ vanishes exactly, while the horizontal ones for $\cS $ and $\cQ $ between the columns  $\cS(C_l) $ and $ \cS (C_n) $ are given by the formulae (see \cite[Section 5.2]{CHM-zig-zag-Ising})
\begin{equation}
\cS(C_n)- \cS(C_l)= \sum^{n}_{k=l}\prod_{p=l+1}^{k} \tan^2(\theta_p) + \sum^{n}_{k=l}\prod_{p=l+1}^{k} \cot^2(\theta_p)
\end{equation}
\begin{equation}
\cQ(C_n)- \cQ(C_l)= \sum^{n}_{k=l}\prod_{p=l+1}^{k} \tan^2(\theta_p) - \sum^{n}_{k=l}\prod_{p=l+1}^{k} \cot^2(\theta_p)
\end{equation}

In particular, provided the ratios $\frac{\sum^{n}_{k=1}\prod_{p=1}^{k} \tan^2(\theta_p)}{\sum^{n}_{k=1}\prod_{p=1}^{k} \cot^2(\theta_p)} $ and $\frac{\sum^{n}_{k=l}\prod_{p=1}^{k} \tan^2(\theta_p)}{n} $ remain bounded away from $0$ and $\infty $, usual RSW property holds at any scale. This observation can be applied to two particular cases:

\begin{itemize}
\item The massive homogeneous model (already described above) on the square lattice. Indeed, in that case $\tan(\theta)=1+\frac{\text{cst}}{n}+o(n^{-1})) $ and straightforward computations allow to conclude. The RSW property holds at any scale as explained above. We also easily see that when $\text{cst} $ diverges to infinity, the optimal Lipchitz constant converges to $1$.
\item The I.I.D near critical model with coupling constants typically much larger than the correlation length. More precisely, sort first a sequence $(Z_k)_{k \in \mathbb{Z}}$ I.I.D $\pm 1 $ Rademacher centered variables. Consider then on the box $[0;n]^2$ on the \emph{standard square lattice embedding} (meaning here that the weights have no geometrical meaning), with uniform coupling constant chosen such that $\tan(\theta_k)=1+\frac{Z_k}{n^\alpha} $ for some $\alpha  \in ]\frac{1}{2};1[ $. If the coupling constant were deterministic such that $\tan(\theta_k)=1+\frac{\textrm{cst}}{n^\alpha} $, then the results of \cite[Section 6]{park-iso} ensure that the crossing probabilities would converge to $0$ or $1$ (depending on the sign of $\textrm{cst} \neq 0$. Still, in this context where we scale the lattice and the temperature at the same time, randomones between layers helps averaging, as one can easily from the central limit theorem that
$\prod_{p=1}^{k} \tan^2(\theta_p)$ is bounded away from $0$ and infinity, which in particular ensures that it falls in the aforementioned setup. Concretely, this gives an example (that is new, to the best of the author's understanding) where randomness in successive layers averages well enough to keep the systeme critical (from the percolation perspective). In our case, RSW property holds almost surely at any scale, but the constants in Theorem \ref{thm:RSW} a priori depend from the output of the sequence $(Z_k)_{k \in \mathbb{Z}}$ (one can still in principle find almost sure bounds at some larger scale or work with high probability, but we do not work it out here and leave it to the interested reader).
\end{itemize}

\medskip

\textbf{Construction of new critical grids using pairs of s-holomorphic functions}

We present now new and simple method to construct new proper s-embeddings that satisfy the RSW property out of already existing ones. We are grateful to Dmitry Chelkak and SC Park for discussions that lead to this construction (which will be detailed in a subsequent work, jointly with Park), and focus only here on the spirit of the construction. We present it in the case where the original embedding is the square lattice, but the technique applies in principle in a more general setup given a sequence of proper s-embeddings satisfying \LipKd\, (note that one should be careful with uniqueness statements in the case the limiting origami map $\vartheta$ is a very rough function). Our first input will be to find a way to discretize an holomorphic function $f$ as a limit of s-holomorphic function $(F^\delta)_{\delta >0} $ in the unit disc $\mathbb{D}$. Set $I = \int f(z)dz $ the primitive of $f$ that vanishes at the origin and consider $N^\delta $ the discrete harmonic extension of $\Re[I] $, coming from the boundary values of $\Re[I] $ on the boundary of $(\cS+i\cQ )\cap \overline{\mathbb{D}} $. Here, the harmonic function is related to the forward random walk associated to the Laplacian given in \cite[Definition 2.15]{Che20} on the S-graph $(\cS+i\cQ )$ (see \cite[Section 2.5]{Che20} for explicit references), which satisfies the so-called \emph{uniform crossing estimates} (see \cite[Section 2.3-2.6]{Che20}). The maximum principle ensures that the family  $(N^\delta)_{\delta >0}$ is bounded thus precompact in $\mathbb{D}$, and the regularity theory for t-holomorphic functions developed in \cite[Section 6]{CLR1}) ensures that the family of its gradients $(F^\delta)_{\delta >0} $ are bounded s-holomorphic (thus also precompact) in $\mathbb{D}$. It is the not hard to see (using uniform crossing estimates of the random walk associated to $N^\delta $) that functions $N^\delta $ and $F^\delta$ converge respectively to $\Re[I] $ and $f$ uniformly on compact of $\mathbb{D}$ (repeating e.g. the arguments of \cite{ChSmi1}).

We are now going to construct a proper s-embedding coming from a \emph{discrete Weierstrass parametrization} of a space-like surface in $\mathbb{R}^{2,1}$, following the route of \cite[Section 3.3]{Cim-universality}. Fix $f$ and $g$ two holomorphic functions on $\mathbb{D} $ such that $\Im[\overline{f}g] >0 $, respectively discretized by $F^\delta $ and $G^\delta$. Then setting (as in \cite[equation (3.11)]{Cim-universality}
\begin{equation}
(\Re[\cS], \Im[\cS], \cQ)=\frac{1}{2}(\Im \int 2F^\delta G^\delta,\Im \int (F^\delta)^2- (G^\delta)^2, \Im \int (F^\delta)^2 + (G^\delta)^2),
\end{equation}
constructs an s-embedding (the integration procedure we have here is understood as \eqref{eq:HF-def}). Using the identification \eqref{eq:F-from-X} we have $\cX = \varsigma (\cX_{f} - i \cX_{g}) $, $\Re[\cS_{\cX}]=2H[\cX_{f},\cX_{g}] $, $\Im[\cS_{\cX}]=H_{\cX_{f} }-H_{\cX_{g} }$ and $\cQ_{\cX}=H_{\cX_{f} }+H_{\cX_{g} }$. It is easy to see from \cite[Proposition 3.20]{Che20} that all the faces of $\Lambda(G) $ in the associated s-embedding are oriented in the same way, they all satisfy an assumption of the kind \Unif\ and that the associated s-embedding is proper provided $\delta $ is small enough (using the argument principle to prove properness as for the massive square lattice treated in \cite[Proposition 3.20]{Che20}). 

The argument principle to ensure properness of the embedding, presented in a very complete manner in \cite[Appendix]{CLR2} can be roughly summarized that one can derive properness (meaning no overlap between different faces) of a piece of an s-embedding realization inside a region delimited by a discrete oriented contour $\mathcal{C}_n $ if the image of a twice bigger oriented contour $\mathcal{C}_{2n} $ (from the graph distance) in the s-embedding realization only winds once around the image of $\mathcal{C}_n $ in the s-embedding realization. We hope this statement could be used as a starting point to study how the embedding procedure evolves when one performs small modifications in the linear system \eqref{eq:3-terms}.

\medskip

\paragraph{\textbf{Finite graphs}}

We discuss now the existence of a proper s-embedding for a given finite graph, following \cite[Section 7]{KLRR}. One of the important output of \cite{KLRR} is that if one starts with bipartite weighted planar graph with an outer face of degree $4$, it is possible to find a t-embedding of its dual. The dimer model on faces of that t-embedding has edge weights which are gauge equivalent to edge-lengths of t-embedding. There, the construction is made by an algorithm, placing new points of the embedding step by step using Miquel's six-circles theorem. Start with a weighted finite planar graph $(G,x)$ with a marked face which corresponds to its unbounded region of one of its embedding $\mathcal{E} $  into the plane and proceed as follows.

\begin{itemize}
\item Up to adding at most $3$ faces to the unbounded face of $\mathcal{E} $, one can assume that the external boundary of that unbounded face in the graph $(G,x)$ has $4n$ edges.  
\item Weld \emph{abstractly} the obtained outer-face with the inner boundary of $ [-2n;2n]^2 \setminus [-n;n]^2 $. Moreover, declare the edge-weights between the added faces to be those of the critical square lattice. This creates an $(\tilde{G},\tilde{x})$ and now an outer face with a boundary of $16n$ edges.

\item Consider the graph $\Lambda(\tilde{G})= \tilde{G}^\bullet \cup \tilde{G}^\circ $ and declare the outer-face of that graph to be the quad of $\Lambda(\tilde{G})$ that corresponds an edge of the box $\Lambda_{2n} $ (before the welding) near $(\frac{3}{2}n;\frac{3}{2}n) $.

\item Now one has a bipartite graph with outer face of degree $4$ and can then apply the construction of \cite{KLRR} to construct the associated t-embedding of the associated dimer model under the bozonisation identities of \cite{dubedat2011exact}. It is possible to use  \cite[Section 7]{KLRR} and obtain a proper s-embedding whose edge weights correspond to the one of $(\tilde{G},\tilde{x})$, except at that new marked faces. In particular, this provides a solution to the propagation equation \eqref{eq:3-terms} except at that marked outer face.
\end{itemize}

Extracting the picture coming associated to original part of $(G,x) $ provides a proper s-embedding. Let us note that this abstract welding to a piece of the critical square lattice is simply done to recover an embedding of the entire original graph $(G,x) $ and could be in principle made with another graph.

\begin{figure}

\begin{minipage}{0.48\textwidth}
\includegraphics[clip, width=0.8\textwidth]{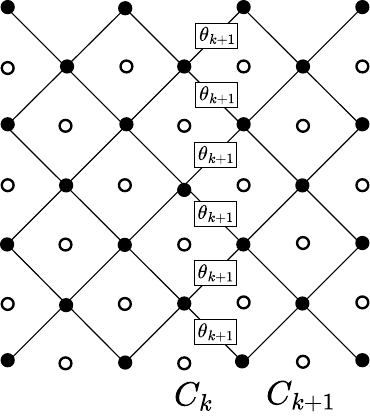}
\end{minipage}
\begin{minipage}{0.48\textwidth}
\includegraphics[clip, width=1\textwidth]{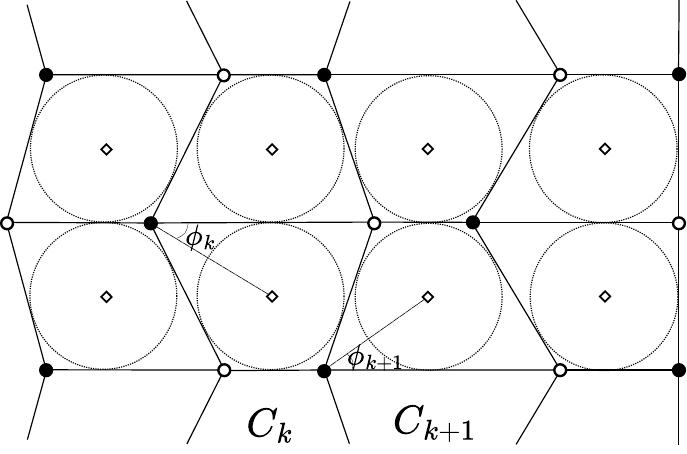}
\end{minipage}
\begin{minipage}{0.48\textwidth}
\includegraphics[clip, width=1\textwidth]{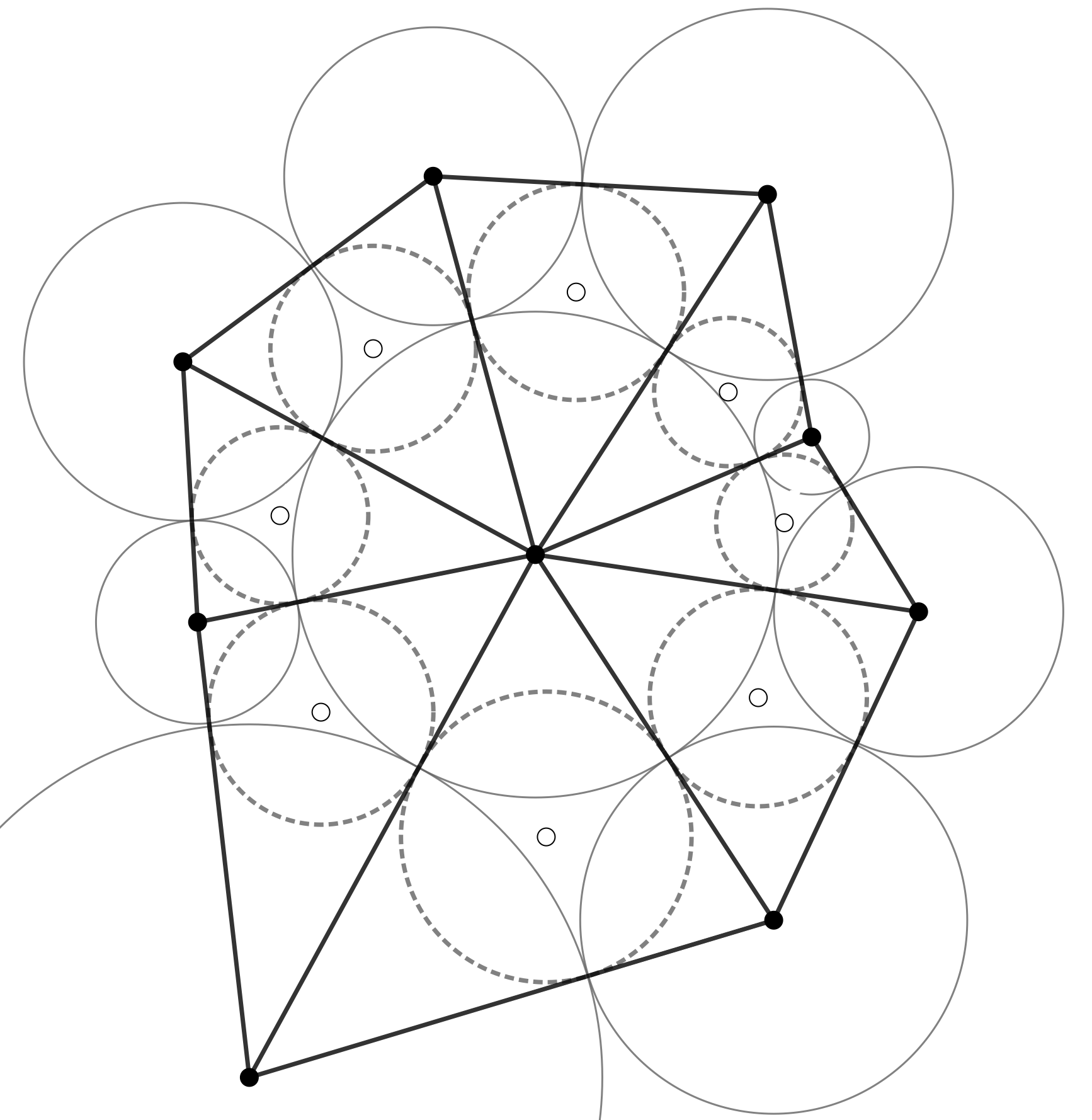}
\end{minipage}
\caption{TOP: Layered model in the zig-zag plane, with one associated proper s-embedding. All the coupling constants between the columns $C_k$ and $C_{k+1}$ are equal to $x_{k+1}= \tan \theta_{k+1}$. One constructs iteratively the layers of the s-embedding with inner circles of radius $1$. Denoting by $\phi_k $ the half-angle that links one of the black vertices of $C_k$ to the center of the quad, the link between the abstract Ising parameter (rewritten via \eqref{eq:x=tan-theta}) and the geometrical angle $\phi_k $ is given by the relation $\tan \phi_{k} = \tan^2 \theta_k \cdot \tan \phi_{k+1} $ as showed in \cite[Equation (5.6)]{CHM-zig-zag-Ising}. BOTTOM: Circle pattern associated to a circle packing. The edges of the triangulation are in solid black, the circles of the packing are in solid grey, while the circles of the dual packing are in dashed grey. Each edge of the triangulation corresponds to a kite in $\Lambda(G)$, which allows to decorate faces of the triangulation with Ising weights coming from \eqref{eq:theta-from-S}. }\label{fig:circle-packing}
\end{figure}

\subsection{Embedding as space-like surface in the Minkoski space and optimality of the assumption \LipKd\,}\label{sub:optimality}

One can wonder it if is possible to prove the same RSW type estimate without an assumption similar to \LipKd\,. The answer to this question is negative and proves the necessity of this kind on assumption on the origami map to prove crossing estimates. To begin with, one can wonder about concrete examples where it appears to fail. Consider first the off-critical homogeneous model on the square lattice (at a fixed $\beta \neq \beta_c $). The RSW box-crossing property classically fails there. Using formulae in the homogeneous layered model presented above (see also \cite[Figure 2]{Che20}), one easily sees that $\limsup_{z\rightarrow \infty} |\frac{\cQ(z)}{\cS(z)}|=1 $, which indicates that the assumption \LipKd\, fails, together with the box-crossing property. We are now going to give a more conceptual explanation to such phenomenology, as explained briefly by Chelkak in several remarks of \cite{Che20}, and is related to more natural way to embed the Ising model as space-like surface $(\cS,\cQ) $ in $\mathbb{R}^{2,1} $. Start with a proper s-embedding $\cS=\cS_{\cX}$, whose origami map is $\cQ=\cQ_\cX $. Beyond simple transformation that preserve the s-embedding setup presented above, one can note that \cite[Remark 1.2]{Che20}, replacing $\cX$ by $ \tilde{\cX}(t)= (1-t^2)^{-\frac{1}{2}} (\cX + t \overline{\cX}) $ for $|t|<1$ provides another proper s-embedding of the same statistical mechanics model. The functions $\cS $ and $\cQ $ are now replaced following the rule
\begin{equation}\label{eq:Minkowski-iso}
(\Re \cS, \Im \cS, \cQ) \mapsto ( \frac{1+t^2}{1-t^2} \Re \cS + \frac{2t}{1-t^2} \cQ, \Im \cS,  \frac{2t}{1-t^2}\Re \cS + \frac{1+t^2}{1-t^2} \cQ ).
\end{equation}
When viewed as a space-like surface in $\mathbb{R}^{2,1} $, the surfaces $(\cS_{\cX},\cQ_{\cX}) $ and $(\cS_{\tilde{\cX}(t)},\cQ_{\tilde{\cX}(t)}) $ are isometric in $\mathbb{R}^{2,1}$. Let us recall that both embedding correspond to the \emph{same} statistical mechanics model and have exactly the same connectivity properties. We are now going to explain how the existence of such isometries imply the necessity of \LipKd\,. One can start with a toy example, which turns out to give the correct reasoning in full generality. Consider the critical square lattice, with its usual square embedding of mesh size $\sqrt{2}\delta $. In that case $\cQ =\delta$ on $\cS(G^\bullet) $ and $\cQ =0$ on $\cS(G^\circ)$. One can apply the transformation of \eqref{eq:Minkowski-iso} for $t$ close to $1$. The change of coordinates ensures that the surface $(\cS_{\tilde{\cX}(t)},\cQ_{\tilde{\cX}(t)}) $ becomes closer and closer to the light cone in $\mathbb{R}^{2,1} $, meaning the optimal Lipschitz constant for the associated origami map gets closer and closer to $1$ as $t\rightarrow \infty $. Considering an $[0;1]^2 $ box over this newly constructed s-embedding in the $\cS^{\delta}_{\tilde{\cX}(t)}$ plane, it corresponds to a smashed rectangle (up to $ O(\delta)$)  $[0;(1+t^2)^{-1}(1-t^2)]\times [0;1] $ in the $\cS^{\delta}_{\cX}$ plane, where the RSW property fails (for top to bottom crossing) for  $t^\delta= (1- \delta^\frac12)^\frac12 $. Such stretching procedure ( replacing $\cX$ by $(1-t^2)^{-\frac{1}{2}} (\cX + t \cX)$ with $|t| \rightarrow 1 $) is a fully general construction and one can easily see from \eqref{eq:Minkowski-iso} that it always leads to an optimal Lipchitz constant close to $1$, while being clearly not compatible with Theorem \ref{thm:RSW}. Thus one cannot bypass the use of the assumption \LipKd\, on the origami map.

\printbibliography

@article{wang1990two,
  title={The Two-Dimensional Random Bond Ising Model at Criticality—A Monte Carlo Study},
  author={Wang, JS and Selke, W and Dotsenko, Vl S and Andreichenko, VB},
  journal={Europhysics Letters},
  volume={11},
  number={4},
  pages={301},
  year={1990},
  publisher={IOP Publishing}
}

@article{Rus78,
  title={A note on percolation},
  author={Russo, Lucio},
  journal={Zeitschrift f{\"u}r Wahrscheinlichkeitstheorie und verwandte Gebiete},
  volume={43},
  number={1},
  pages={39--48},
  year={1978},
  publisher={Springer}
}

@article{CDH,
 author = {Chelkak, Dmitry and Duminil-Copin, Hugo and Hongler, Cl{\'e}ment},
 title = {Crossing probabilities in topological rectangles for the critical planar {FK}-{Ising} model},
 fjournal = {Electronic Journal of Probability},
 journal = {Electron. J. Probab.},
 issn = {1083-6489},
 volume = {21},
 pages = {28},
 note = {Id/No 5},
 year = {2016},
 language = {English},
 doi = {10.1214/16-EJP3452},
 zbMATH = {6583728},
 Zbl = {1341.60124}
}

@article{DCHN,
  title={Connection probabilities and RSW-type bounds for the two-dimensional FK Ising model},
  author={Duminil-Copin, Hugo and Hongler, Cl{\'e}ment and Nolin, Pierre},
  journal={Communications on pure and applied mathematics},
  volume={64},
  number={9},
  pages={1165--1198},
  year={2011},
  publisher={Wiley Online Library}
}

@incollection{SW78,
  title={Percolation probabilities on the square lattice},
  author={Seymour, Paul D and Welsh, Dominic JA},
  booktitle={Annals of Discrete Mathematics},
  volume={3},
  pages={227--245},
  year={1978},
  publisher={Elsevier}
}

@article{Rus81,
  title={On the critical percolation probabilities},
  author={Russo, Lucio},
  journal={Zeitschrift f{\"u}r Wahrscheinlichkeitstheorie und verwandte Gebiete},
  volume={56},
  number={2},
  pages={229--237},
  year={1981},
  publisher={Springer}
}

@article{GGJN,
  title={A combinatorial criterion for macroscopic circles in planar triangulations},
  author={Gurel-Gurevich, Ori and Jerison, Daniel C and Nachmias, Asaf},
  journal={arXiv preprint arXiv:1906.01612},
  year={2019}
}

@article{akopyan2018incircular,
  title={Incircular nets and confocal conics},
  author={Akopyan, Arseniy and Bobenko, Alexander},
  journal={Transactions of the American Mathematical Society},
  volume={370},
  number={4},
  pages={2825--2854},
  year={2018}
}

@article{duplantier2021liouville,
  title={Liouville quantum gravity as a mating of trees},
  author={Duplantier, Bertrand and Miller, Jason R and Sheffield, Scott},
  journal={Ast{\'e}risque},
  volume={427},
  year={2021}
}

@article{holden2023convergence,
  title={Convergence of uniform triangulations under the Cardy embedding},
  author={Holden, Nina and Sun, Xin},
  journal={Acta Mathematica},
  volume={230},
  number={1},
  pages={93--203},
  year={2023},
  publisher={International Press of Boston}
}

@book {Baxter-book,
    AUTHOR = {Baxter, Rodney J.},
     TITLE = {Exactly solved models in statistical mechanics},
      NOTE = {Reprint of the 1982 original},
 PUBLISHER = {Academic Press, Inc. [Harcourt Brace Jovanovich, Publishers],
              London},
      YEAR = {1989},
     PAGES = {xii+486},
      ISBN = {0-12-083182-1},
   MRCLASS = {82-02 (82A05 82A69)},
  MRNUMBER = {998375},
}

@article{Che20,
 author = {Chelkak, Dmitry},
 title = {Ising model and s-embeddings of planar graphs},
 fjournal = {Annales Scientifiques de l'{\'E}cole Normale Sup{\'e}rieure. Quatri{\`e}me S{\'e}rie},
 journal = {Ann. Sci. {\'E}c. Norm. Sup{\'e}r. (4)},
 issn = {0012-9593},
 volume = {57},
 number = {5},
 pages = {1271--1346},
 year = {2024},
 language = {English},
 doi = {10.24033/asens.2591},
 zbMATH = {8011222}
}

@article{CIM-universality,
 author = {Chelkak, Dmitry and Izyurov, Konstantin and Mahfouf, R{\'e}my},
 title = {Universality of spin correlations in the {Ising} model on isoradial graphs},
 fjournal = {The Annals of Probability},
 journal = {Ann. Probab.},
 issn = {0091-1798},
 volume = {51},
 number = {3},
 pages = {840--898},
 year = {2023},
 language = {English},
 doi = {10.1214/22-AOP1595},
 zbMATH = {7690050},
 Zbl = {1517.82012}
}

@article{CHI-mixed,
  title={Correlations of primary fields in the critical Ising model},
  author={Chelkak, Dmitry and Hongler, Cl{\'e}ment and Izyurov, Konstantin},
  journal={arXiv preprint arXiv:2103.10263},
  year={2021}
}

@article{gwynne2021tutte,
 author = {Gwynne, Ewain and Miller, Jason and Sheffield, Scott},
 title = {Harmonic functions on mated-{CRT} maps},
 fjournal = {Electronic Journal of Probability},
 journal = {Electron. J. Probab.},
 issn = {1083-6489},
 volume = {24},
 pages = {55},
 note = {Id/No 58},
 year = {2019},
 language = {English},
 doi = {10.1214/19-EJP325},
 zbMATH = {7088996},
 Zbl = {1466.60090}
}

@article{MahPar23a,
  title={Convergence of fermionic observables and crossing probabilities on s-embeddings},
  author={Rémy Mahfouf and Sung Chul Park},
  year={In preparation}
}

@article{MahPar23b,
  title={ Convergence of the energy density on general s-embeddings},
  author={Rémy Mahfouf and Sung Chul Park},
  year={In preparation}
}

@article{CHM-zig-zag-Ising,
 author = {Chelkak, Dmitry and Hongler, Cl{\'e}ment and Mahfouf, R{\'e}my},
 title = {Magnetization in the zig-zag layered {Ising} model and orthogonal polynomials},
 fjournal = {Annales de l'Institut Fourier},
 journal = {Ann. Inst. Fourier},
 issn = {0373-0956},
 volume = {74},
 number = {6},
 pages = {2275--2330},
 year = {2024},
 language = {English},
 doi = {10.5802/aif.3605},
 zbMATH = {7929007},
 Zbl = {1551.82008}
}

@article{dubedat2011exact,
       author = {{Dub{\'e}dat}, Julien},
        title = {Exact bosonization of the {I}sing model},
      journal = {arXiv e-prints},
         year = 2011,
        month = dec,
          eid = {arXiv:1112.4399},
        pages = {arXiv:1112.4399},
archivePrefix = {arXiv},
       eprint = {1112.4399},
 primaryClass = {math.PR},
       adsurl = {https://ui.adsabs.harvard.edu/abs/2011arXiv1112.4399D}
}

@phdthesis{MahPHD,
	author = "Rémy Mahfouf",
	school = "Université Paris-Saclay",
	title = "{Universalité du modèle d'Ising au-delà de cadre isoradial}",
	year = "2022"
}

@article{park2018massive,
       author = {{Park}, S.~C.},
        title = {Massive Scaling Limit of the {I}sing Model: Subcritical Analysis and Isomonodromy},
      journal = {arXiv e-prints},
         year = 2018,
        month = nov,
          eid = {arXiv:1811.06636},
        pages = {arXiv:1811.06636},
archivePrefix = {arXiv},
       eprint = {1811.06636},
 primaryClass = {math.PR},
       adsurl = {https://ui.adsabs.harvard.edu/abs/2018arXiv181106636P}
}

@article{Smirnov_Ising,
    AUTHOR = {Smirnov, Stanislav},
     TITLE = {Conformal invariance in random cluster models. {I}.
              {H}olomorphic fermions in the {I}sing model},
   JOURNAL = {Ann. of Math. (2)},
  FJOURNAL = {Annals of Mathematics. Second Series},
    VOLUME = {172},
      YEAR = {2010},
    NUMBER = {2},
     PAGES = {1435--1467},
      ISSN = {0003-486X},
   MRCLASS = {60K35 (30G25 60J67 81T40 82B20)},
  MRNUMBER = {2680496},
MRREVIEWER = {Roland M. Friedrich},
       DOI = {10.4007/annals.2010.172.1441},
       URL = {https://doi.org/10.4007/annals.2010.172.1441},
}

@book{mccoy-wu-book,
    AUTHOR = {McCoy, Barry M. and Wu, Tai Tsun},
     TITLE = {The two-dimensional {I}sing model},
   EDITION = {Second},
      NOTE = {Corrected reprint of [ MR3618829], with a new preface and a
              new chapter (Chapter XVII)},
 PUBLISHER = {Dover Publications, Inc., Mineola, NY},
      YEAR = {2014},
     PAGES = {xvi+454},
      ISBN = {978-0-486-49335-0; 0-486-49335-0},
   MRCLASS = {82B20 (82C20)},
  MRNUMBER = {3838431},
}

@article {mercat-CMP,
    AUTHOR = {Mercat, Christian},
     TITLE = {Discrete {R}iemann surfaces and the {I}sing model},
   JOURNAL = {Comm. Math. Phys.},
  FJOURNAL = {Communications in Mathematical Physics},
    VOLUME = {218},
      YEAR = {2001},
    NUMBER = {1},
     PAGES = {177--216},
      ISSN = {0010-3616},
     CODEN = {CMPHAY},
   MRCLASS = {82B20 (39A12 58J90)},
  MRNUMBER = {1824204 (2002c:82019)},
MRREVIEWER = {Richard Kenyon},
       DOI = {10.1007/s002200000348},
       URL = {http://dx.doi.org/10.1007/s002200000348},
}

@book {friedli-velenik-book,
    AUTHOR = {Friedli, S. and Velenik, Y.},
     TITLE = {Statistical mechanics of lattice systems},
      NOTE = {A concrete mathematical introduction},
 PUBLISHER = {Cambridge University Press, Cambridge},
      YEAR = {2018},
     PAGES = {xix+622},
      ISBN = {978-1-107-18482-4},
   MRCLASS = {82-01 (82B05)},
  MRNUMBER = {3752129},
}

@article {cimasoni-duminil,
    AUTHOR = {Cimasoni, David and Duminil-Copin, Hugo},
     TITLE = {The critical temperature for the {I}sing model on planar
              doubly periodic graphs},
   JOURNAL = {Electron. J. Probab.},
  FJOURNAL = {Electronic Journal of Probability},
    VOLUME = {18},
      YEAR = {2013},
     PAGES = {no. 44, 18},
      ISSN = {1083-6489},
   MRCLASS = {82B20 (82B27)},
  MRNUMBER = {3040554},
MRREVIEWER = {Sven Bachmann},
       DOI = {10.1214/EJP.v18-2352},
       URL = {https://doi.org/10.1214/EJP.v18-2352},
}

@article{duminil-garban-pete,
  title={The near-critical planar FK-Ising model},
  author={Duminil-Copin, Hugo and Garban, Christophe and Pete, G{\'a}bor},
  journal={Communications in Mathematical Physics},
  volume={326},
  number={1},
  pages={1--35},
  year={2014},
  publisher={Springer}
}

@article{lis-kites,
  title={Circle patterns and critical Ising models},
  author={Lis, Marcin},
  journal={Communications in Mathematical Physics},
  volume={370},
  number={2},
  pages={507--530},
  year={2019},
  publisher={Springer}
}

@article{CCK,
  title={Revisiting the combinatorics of the 2D Ising model},
  author={Chelkak, Dmitry and Cimasoni, David and Kassel, Adrien},
  journal={Annales de l’Institut Henri Poincar{\'e} D},
  volume={4},
  number={3},
  pages={309--385},
  year={2017}
}

@article{KemSmi2,
  title={Conformal invariance in random cluster models. II. Full scaling limit as a branching SLE},
  author={Kemppainen, Antti and Smirnov, Stanislav},
  journal={arXiv preprint arXiv:1609.08527},
  year={2016}
}

@article{Smi-ICM06,
  title={Towards conformal invariance of 2D lattice models},
  author={Smirnov, Stanislav},
  journal={ In {\em International {C}ongress of {M}athematicians. {V}ol. {II}},
  pages 1421--1451. Eur. Math. Soc., Z\"urich},
  year={2006}
}

@article{CLR2,
  title={Bipartite dimer model: perfect t-embeddings and Lorentz-minimal surfaces},
  author={Chelkak, Dmitry and Laslier, Beno{\^\i}t and Russkikh, Marianna},
  journal={arXiv preprint arXiv:2109.06272},
  year={2021}
}

@article{park-iso,
  title={Convergence of fermionic observables in the massive planar FK-Ising model},
  author={Park, SC},
  journal={Communications in mathematical physics},
  volume={396},
  number={3},
  pages={1071--1133},
  year={2022},
  publisher={Springer}
}

@article{Mah25,
  title={Conformal invariance of the near-critical random bond Ising model via embedding deformation},
  author={Mahfouf Rémy},
  journal={arXiv preprint arXiv:2509.08928},
  year={2025}
}

@incollection{AIM,
  title={Elliptic Partial Differential Equations and Quasiconformal Mappings in the Plane (PMS-48)},
  author={Astala, Kari and Iwaniec, Tadeusz and Martin, Gaven},
  booktitle={Elliptic Partial Differential Equations and Quasiconformal Mappings in the Plane (PMS-48)},
  year={2008},
  publisher={Princeton University Press}
}

@book{palmer2007planar,
  title={Planar Ising Correlations},
  author={Palmer, John},
  volume={49},
  year={2007},
  publisher={Springer Science \& Business Media}
}

@article{BdTR1,
  title={The Z-invariant massive Laplacian on isoradial graphs},
  author={Boutillier, C{\'e}dric and de Tiliere, B{\'e}atrice and Raschel, Kilian},
  journal={Inventiones mathematicae},
  volume={208},
  number={1},
  pages={109--189},
  year={2017},
  publisher={Springer}
}

@article{BdTR2,
  title={The Z-invariant Ising model via dimers},
  author={Boutillier, C{\'e}dric and Tili{\`e}re, B{\'e}atrice de and Raschel, Kilian},
  journal={Probability Theory and Related Fields},
  volume={174},
  number={1},
  pages={235--305},
  year={2019},
  publisher={Springer}
}

@article{ChSmi1,
  title={Discrete complex analysis on isoradial graphs},
  author={Chelkak, Dmitry and Smirnov, Stanislav},
  journal={Advances in Mathematics},
  volume={228},
  number={3},
  pages={1590--1630},
  year={2011},
  publisher={Elsevier}
}

@inproceedings {Ch-ICM18,
    AUTHOR = {Chelkak, Dmitry},
     TITLE = {Planar {I}sing model at criticality: state-of-the-art and
              perspectives},
 BOOKTITLE = {Proceedings of the {I}nternational {C}ongress of
              {M}athematicians---{R}io de {J}aneiro 2018. {V}ol. {IV}.
              {I}nvited lectures},
     PAGES = {2801--2828},
 PUBLISHER = {World Sci. Publ., Hackensack, NJ},
      YEAR = {2018},
   MRCLASS = {82C20 (30G25 35Q82 60J67)},
  MRNUMBER = {3966512},
}

@article{dotsenko1983critical,
  title={Critical behaviour of the phase transition in the 2D Ising model with impurities},
  author={Dotsenko, Viktor S and Dotsenko, Vladimir S},
  journal={Advances in Physics},
  volume={32},
  number={2},
  pages={129--172},
  year={1983},
  publisher={Taylor \& Francis}
}

@article{perk1980quadratic,
  title={Quadratic identities for Ising model correlations},
  author={Perk, Jacques HH},
  journal={Physics Letters A},
  volume={79},
  number={1},
  pages={3--5},
  year={1980},
  publisher={Elsevier}
}

@article {kadanoff-ceva,
    AUTHOR = {Kadanoff, Leo P. and Ceva, Horacio},
     TITLE = {Determination of an operator algebra for the two-dimensional
              {I}sing model},
   JOURNAL = {Phys. Rev. B (3)},
  FJOURNAL = {Physical Review. B. Condensed Matter. Third Series},
    VOLUME = {3},
      YEAR = {1971},
     PAGES = {3918--3939},
      ISSN = {0163-1829},
   MRCLASS = {82.62},
  MRNUMBER = {389111},
MRREVIEWER = {P. Caldirola},
}

@article {CHI,
    AUTHOR = {Chelkak, Dmitry and Hongler, Cl\'ement and Izyurov, Konstantin},
     TITLE = {Conformal invariance of spin correlations in the planar
              {I}sing model},
   JOURNAL = {Ann. of Math. (2)},
  FJOURNAL = {Annals of Mathematics. Second Series},
    VOLUME = {181},
      YEAR = {2015},
    NUMBER = {3},
     PAGES = {1087--1138},
      ISSN = {0003-486X},
   MRCLASS = {82B20},
  MRNUMBER = {3296821},
MRREVIEWER = {Adnene Besbes},
       DOI = {10.4007/annals.2015.181.3.5},
       URL = {http://dx.doi.org/10.4007/annals.2015.181.3.5},
}

@article {hon-smi,
    AUTHOR = {Hongler, Cl{\'e}ment and Smirnov, Stanislav},
     TITLE = {The energy density in the planar {I}sing model},
   JOURNAL = {Acta Math.},
  FJOURNAL = {Acta Mathematica},
    VOLUME = {211},
      YEAR = {2013},
    NUMBER = {2},
     PAGES = {191--225},
      ISSN = {0001-5962},
   MRCLASS = {82B20},
  MRNUMBER = {3143889},
MRREVIEWER = {Xifeng Su},
       DOI = {10.1007/s11511-013-0102-1},
       URL = {http://dx.doi.org/10.1007/s11511-013-0102-1},
}

@article {ChSmi2,
    AUTHOR = {Chelkak, Dmitry and Smirnov, Stanislav},
     TITLE = {Universality in the 2{D} {I}sing model and conformal
              invariance of fermionic observables},
   JOURNAL = {Invent. Math.},
  FJOURNAL = {Inventiones Mathematicae},
    VOLUME = {189},
      YEAR = {2012},
    NUMBER = {3},
     PAGES = {515--580},
      ISSN = {0020-9910},
     CODEN = {INVMBH},
   MRCLASS = {82B20 (30C20 30C35 60J67 60K35)},
  MRNUMBER = {2957303},
MRREVIEWER = {Ben Dyhr},
       DOI = {10.1007/s00222-011-0371-2},
       URL = {http://dx.doi.org/10.1007/s00222-011-0371-2},
}

@article{holden2018sle,
  title={SLE as a mating of trees in Euclidean geometry},
  author={Holden, Nina and Sun, Xin},
  journal={Communications in Mathematical Physics},
  volume={364},
  number={1},
  pages={171--201},
  year={2018},
  publisher={Springer}
}

@article{DMT21,
  title={Planar random-cluster model: fractal properties of the critical phase},
  author={Duminil-Copin, Hugo and Manolescu, Ioan and Tassion, Vincent},
  journal={Probability Theory and Related Fields},
  volume={181},
  number={1},
  pages={401--449},
  year={2021},
  publisher={Springer}
}

@article{duplantier-sheffield-LQG,
  title={Liouville quantum gravity and KPZ},
  author={Duplantier, Bertrand and Sheffield, Scott},
  journal={Inventiones mathematicae},
  volume={185},
  number={2},
  pages={333--393},
  year={2011},
  publisher={Springer}
}

@article{CLR1,
 author = {Chelkak, Dmitry and Laslier, Beno{\^{\i}}t and Russkikh, Marianna},
 title = {Dimer model and holomorphic functions on t-embeddings of planar graphs},
 fjournal = {Proceedings of the London Mathematical Society. Third Series},
 journal = {Proc. Lond. Math. Soc. (3)},
 issn = {0024-6115},
 volume = {126},
 number = {5},
 pages = {1656--1739},
 year = {2023},
 language = {English},
 doi = {10.1112/plms.12516},
 zbMATH = {7740458},
 Zbl = {1522.82008}
}

@article{KLRR,
 author = {Kenyon, Richard and Lam, Wai Yeung and Ramassamy, Sanjay and Russkikh, Marianna},
 title = {Dimers and circle patterns},
 fjournal = {Annales Scientifiques de l'{\'E}cole Normale Sup{\'e}rieure. Quatri{\`e}me S{\'e}rie},
 journal = {Ann. Sci. {\'E}c. Norm. Sup{\'e}r. (4)},
 issn = {0012-9593},
 volume = {55},
 number = {3},
 pages = {865--903},
 year = {2022},
 language = {English},
 doi = {10.24033/asens.2507},
 zbMATH = {7594372},
 Zbl = {1497.82007}
}

@article{duminil-parafermions,
  title={Parafermionic observables and their applications to planar statistical physics models},
  author={Duminil-Copin, Hugo},
  journal={Ensaios Matematicos},
  volume={25},
  pages={1--371},
  year={2013}
}

@article {DCLM-universality,
    AUTHOR = {Duminil-Copin, Hugo and Li, Jhih-Huang and Manolescu, Ioan},
     TITLE = {Universality for the random-cluster model on isoradial graphs},
   JOURNAL = {Electron. J. Probab.},
  FJOURNAL = {Electronic Journal of Probability},
    VOLUME = {23},
      YEAR = {2018},
     PAGES = {Paper No. 96, 70},
   MRCLASS = {60K35 (82B20 82B27)},
  MRNUMBER = {3858924},
MRREVIEWER = {Ji\v{r}\'{\i} \v{C}ern\'{y}},
       DOI = {10.1214/18-EJP223},
       URL = {https://doi.org/10.1214/18-EJP223},
}

\end{document}